\documentclass[10pt]{amsart}
\usepackage{a4,amssymb,times,amsmath,cancel,mathrsfs,multirow,textcomp}
\usepackage[bbgreekl]{mathbbol}
\usepackage{graphicx,color,epsfig}

\usepackage[all]{xy}
\usepackage{tikz}
\usetikzlibrary{calc,3d,arrows,decorations.markings}
\tikzset{->-/.style={decoration={  markings,  mark=at position .75 with {\arrow{latex}}},postaction={decorate}}}
\tikzset{-<-/.style={decoration={  markings,  mark=at position .75 with {\arrow{latex reversed}}},postaction={decorate}}}

\def\Id{\id}
\def\opp{\mathrm{opp}}

\let\cal\mathcal
\def\cA{{\mathtt A}}
\def\cB{{\mathtt B}}

\def\cF{{\cal F}}

\def\cI{{\cal I}}

\def\cM{{\cal M}}

\def\cP{{\cal P}}

\def\cR{{\cal R}}

\def\cT{{\cal T}}

\def\cZ{{\cal Z}}
\def\pZ{{\cal Z}}
\def\eps{{\epsilon}}

\let\blb\mathbb

\def\H{{\blb M}}

\def \TT{{\blb T}}

\def \Rl{{\blb R}}

\newcommand{\weg}[1]{}
\newcommand{\se}[1]{\begin{equation*}\begin{split}#1\end{split}\end{equation*}}

\newcommand{\lvect}[1]{\overleftarrow{#1}}

\newcommand{\del}[1]{{#1}^{-1}}
\newcommand{\C}{\mathbb{C}}
\newcommand{\N}{\mathbb{N}}
\newcommand{\Z}{\mathbb{Z}}
\newcommand{\R}{\mathbb{R}}
\newcommand{\Sf}{\mathbb{S}}
\newcommand{\RZ}{\zeta}

\newcommand{\rest}{\mathrm{rest}}
\def\cLL{{\mathscr L}}

\def\ccO{{\mathscr O}}
\def\ccT{{\mathscr T}}

\def\ccF{{\mathscr F}}
\def\ccG{{\mathscr G}}

\newcommand{\genmu}{\bbmu}

\newcommand{\grp}[1]{\mathsf{#1}}

\newcommand{\vtx}[1]{*+[o][F-]{\scriptscriptstyle #1}}
\newcommand{\sqvtx}[1]{*+[F]{\scriptscriptstyle #1}}

\newcommand{\Cok}{\mathtt{Cok}}

\newcommand{\Spec}{\ensuremath{\mathsf{Spec}}}

\newcommand{\<}{\langle}
\renewcommand{\>}{\rangle}

\newcommand{\Aut}{\ensuremath{\mathsf{Aut}}}

\newcommand{\End}{\ensuremath{\mathsf{End}}}

\newcommand{\Ext}{\mathsf{Ext}}

\newtheorem{lemma}{Lemma}[section]
\newtheorem{proposition}[lemma]{Proposition}
\newtheorem{theorem}[lemma]{Theorem}
\newtheorem{corollary}[lemma]{Corollary}

\theoremstyle{definition}

\newtheorem{example}[lemma]{Example}
\newtheorem{definition}[lemma]{Definition}

{

}
\newtheorem{remark}[lemma]{Remark}
\newtheorem{aside}[lemma]{Remark}

\theoremstyle{remark}

\newcommand{\Proj}{\ensuremath{\mathsf{Proj}}}

\newcommand{\Hom}{\mathtt{Hom}}

\newcommand{\Ker}{\textrm{Ker}}
\newcommand{\Image }{\textrm{Im}}

\newcommand{\id}{\mathbf{1}}

\newcommand{\A}{\ensuremath{\mathbb{A}}}

\newcommand{\Ob}{\mathtt{Ob}\,}

\renewcommand{\H}{\mathtt{H}\,}
\newcommand{\Tw}{\mathtt{Tw}\,}

\newcommand{\Hoch}{\mathtt{HochC}}
\newcommand{\Coh}{\mathtt{Coh}\,}
\newcommand{\cJA}{\mathtt{Jac}\,}
\newcommand{\cQA}{\mathtt{Gtl}\,}
\def\cC{{\mathtt C}}

\def\Fuk{{\mathtt {Fuk}}}

\def\MF{{\mathtt {MF}}}

\def\cD{{\mathtt D}}

\def\fuk{{\mathtt {fuk}}}
\def\mf{{\mathtt {mf}}}

\def\qpol{\mathrm{Q}}

\newcommand{\ucover}[1]{\tilde{#1}^u}

\newcommand{\mirror}[1]{\mathrel{\reflectbox{$#1$}}}

\newcommand{\rect}{\mathrm{R}}

\newcommand{\smatrix}[1]{\left(\begin{smallmatrix}#1\end{smallmatrix}\right)}
\newcommand{\Mod}{\ensuremath{\mathtt{Mod}}}

\title{Noncommutative mirror symmetry for punctured surfaces}

\author{Raf Bocklandt\\With an appendix by Mohammed Abouzaid}
\address{Raf Bocklandt\\
Korteweg de Vries institute\\
University of Amsterdam (UvA)\\
Science Park 904\\ 
1098 XH Amsterdam\\ 
The Netherlands
}
\email{raf.bocklandt@gmail.com}

\xyoption{all}

\begin{document}
\begin{abstract}
In \cite{Auroux} Abouzaid, Auroux, Efimov, Katzarkov and Orlov showed that the 
wrapped Fukaya Categories of punctured spheres and finite unbranched covers of punctured spheres are derived equivalent
to the categories of singularities of a superpotential on certain crepant resolutions of toric 
3 dimensional singularities. We generalize this result to other punctured
Riemann surfaces and reformulate it in terms of certain noncommutative algebras coming from dimer models. 
In particular, given any consistent dimer model we can look at a subcategory of noncommutative matrix factorizations
and show that this category is $\cA_\infty$-isomorphic to a subcategory of the wrapped Fukaya category of a 
punctured Riemann surface. The connection between the dimer model and the punctured Riemann surface
then has a nice interpretation in terms of a duality on dimer models. 
\end{abstract}

\maketitle

\section{Introduction}

Originally homological mirror symmetry was developed by Kontsevich \cite{Kontsevich} as a framework to 
explain the similarities between the symplectic geometry and algebraic geometry
of certain Calabi-Yau manifolds. More precisely its main conjecture states that for any compact Calabi-Yau manifold
with a complex structure $X$, one can find a mirror Calabi-Yau manifold $X'$ equipped with a symplectic structure
such that the derived category of coherent sheaves over $X$ is equivalent to the zeroth homology of the triangulated envelop of the split closure of
the Fukaya category
of $X'$. The latter is a category that represents the intersection theory of Lagrangian submanifolds of $X'$.
\[
 \cD^b\Coh X \sim \H^0\Tw^\pi\Fuk X'
\]
Over the years it turned out that this conjecture is part of a set of equivalences which
are much broader than the compact Calabi-Yau setting \cite{katz,HoriVafa,Abou1,AKO1,AKO2}. Removing the compactness or Calabi-Yau condition often
makes the mirror a singular object, which physicists call a Landau-Ginzburg model \cite{orlov1,orlov2}.

A Landau-Ginzburg model $(X,W)$ is a pair of a smooth space $X$ and a complex valued function $W:X\to \C$, which is called the potential.
On the algebraic side we associate to it the dg-category of matrix factorizations $\MF(X,W)$.
Its objects are diagrams $\xymatrix{P_0\ar@<.5ex>[r]&P_1\ar@<.5ex>[l]}$ where $P_i$ are vector bundles and
the composition of the maps results in multiplication with $W$. The morphisms are morphisms between these vector bundles equipped with a natural differential.

On the other hand if $X'$ is noncompact we need to tweak the notion of the Fukaya category, by imposing certain conditions on the behaviour
of the Lagrangians near infinity and using a Hamiltonian flow to adjust the intersection theory.
This gives us the notion of a wrapped Fukaya category \cite{Abou2}.

In \cite{Auroux} Abouzaid, Auroux, Efimov, Katzarkov and Orlov proved an instance of mirror symmetry between such objects.
On the symplectic side they considered a sphere with $k$ punctures and on the algebraic side they considered
a special Landau Ginzburg model on a certain toric quasiprojective noncompact Calabi Yau threefold and they proved an equivalence
between the derived wrapped Fukaya category of the former and the derived category of matrix factorizations of the latter.
They also extended these results to finite unbranched covers of punctured spheres.

In this paper we aim to generalize their result to all Riemann surfaces with $k\ge 3$ punctures.
On the algebraic side though we will not construct classical Landau-Ginzburg models but instead we will
look at noncommutative Landau-Ginzburg models \cite{quintero}. This means that we replace the commutative space $X$ with a noncommutative
Calabi-Yau algebra $A$ and the potential will be a central element. The category of matrix factorizations needs also an adjustment:
instead of vector bundles we must take projective modules.

The Calabi-Yau algebras under consideration will come from dimer models, which are certain quivers embedded in a Riemann surface.
Such algebras, known as Jacobi algebras, also have a canonical choice of potential $\ell$ coming from the faces in which the quiver splits the Riemann surface.

The result we obtain is that for any consistent dimer model $\qpol$ that has a perfect matching, we can find a full subcategory $\mf(\qpol)$ of the category of all matrix factorizations of its Jacobi algebra $\MF(\cJA \qpol,\ell)$ which is $\cA_\infty$-isomorphic to a full subcategory $\fuk(\mirror{\qpol})$ of the wrapped Fukaya category $\Fuk(S)$ of a punctured Riemann surface $S$. 
The category $\fuk(\mirror{\qpol})$ is constructed using a new dimer model $\mirror{\qpol}$, embedded in the closure of $S$ such that its vertices are the punctures.
The new dimer can be obtained explicitly from the original dimer by a process called dimer duality.

We illustrate how our viewpoint and the approach in \cite{Auroux} fit together. The simplest example gives us an equivalence between the sphere with $3$ punctures $\Sf^{\bullet 3}$
and the Landau-Ginzburg model $(\C^3, xyz)$. Commutatively this corresponds to the category of singularities of the three
standard planes in affine 3-space. 
Noncommutatively\footnote{Note that in this example the Jacobi algebra $\cJA\qpol\cong \C[X,Y,Z|$ is not noncommutative, but for all other dimer models it is.} we can see this as a dimer model on the torus with a superpotential $\ell$ coming
from the faces. 
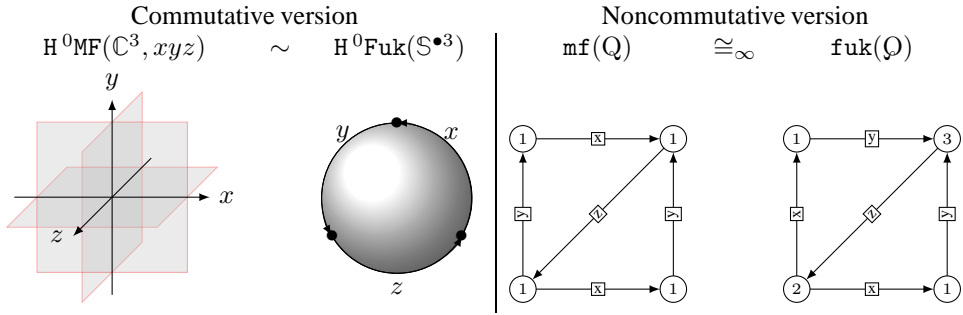
\begin{center}
\begin{tabular}{ccc|ccc}
\multicolumn{3}{c}{Commutative version}&
\multicolumn{3}{c}{Noncommutative version}\\
$\H^0\MF(\C^3,xyz)$&$\sim$&$\H^0\Fuk(\Sf^{\bullet 3})$&
$\mf(\qpol)$&$\cong_{\infty}$&$\fuk(\mirror{\qpol})$\\
\begin{tikzpicture}[x  = {(1cm,0cm)},
                    y  = {(0cm,1cm)},
                    z  = {(-.4cm,-.4cm)},
                    color = {black}]
% style of faces
\tikzset{facestyle/.style={fill=lightgray,draw=red,very thin,line join=round}}
% face "back" 
\begin{scope}[canvas is zy plane at x=0]
  \path[facestyle,opacity=.3] (-1,-1) rectangle (1,1);
\end{scope}
\begin{scope}[canvas is zx plane at y=0]
  \path[facestyle,opacity=.3] (-1,-1) rectangle (1,1);
\end{scope}
\begin{scope}[canvas is xy plane at z=0]
  \path[facestyle,opacity=.3] (-1,-1) rectangle (1,1);
\end{scope}
\draw[-latex] (-1.3,0,0) --  (1.3,0,0) node[right]{$x$};
\draw[-latex] (0,-1.3,0) -- (0,1.3,0) node[above]{$y$};
\draw[-latex] (0,0,-1.3) -- (0,0,1.3) node[below,left]{$z$};;
\end{tikzpicture}&&
\begin{tikzpicture}
\filldraw[ball color=white] (0cm,0cm) circle (1cm);
\draw (0,1) node{$\bullet$};
\draw (-.86,-.5) node{$\bullet$};
\draw (.86,-.5) node{$\bullet$};
\draw [-latex] (0,1) arc(90:210:1);
\draw [-latex] (-.86,-.5) arc(210:330:1);
\draw [-latex] (.86,-.5) arc(-30:90:1);
\draw (-.5,.86) node[above,left]{$y$};
\draw (.5,.86) node[above,right]{$x$};
\draw (0,-1) node[below]{$z$};
\end{tikzpicture}&
\begin{tikzpicture}
\draw [-latex,shorten >=5pt] (0cm,0cm) to node [rectangle,draw,fill=white,sloped,inner sep=1pt] {{\tiny x}} (2cm,0cm); 
\draw [-latex,shorten >=5pt] (0cm,2cm) to node [rectangle,draw,fill=white,sloped,inner sep=1pt] {{\tiny x}} (2cm,2cm); 
\draw [-latex,shorten >=5pt] (0cm,0cm) to node [rectangle,draw,fill=white,sloped,inner sep=1pt] {{\tiny y}} (0cm,2cm); 
\draw [-latex,shorten >=5pt] (2cm,0cm) to node [rectangle,draw,fill=white,sloped,inner sep=1pt] {{\tiny y}} (2cm,2cm); 
\draw [-latex,shorten >=5pt] (2cm,2cm) to node [rectangle,draw,fill=white,sloped,inner sep=1pt] {{\tiny z}} (0cm,0cm); 
\node at (0cm,0cm) [circle,draw,fill=white,minimum size=10pt,inner sep=1pt] {\mbox{\tiny $1$}}; 
\node at (0cm,2cm) [circle,draw,fill=white,minimum size=10pt,inner sep=1pt] {\mbox{\tiny $1$}}; 
\node at (2cm,0cm) [circle,draw,fill=white,minimum size=10pt,inner sep=1pt] {\mbox{\tiny $1$}}; 
\node at (2cm,2cm) [circle,draw,fill=white,minimum size=10pt,inner sep=1pt] {\mbox{\tiny $1$}}; 
\end{tikzpicture}
&&
\begin{tikzpicture}
\draw [-latex,shorten >=5pt] (0cm,0cm) to node [rectangle,draw,fill=white,sloped,inner sep=1pt] {{\tiny x}} (2cm,0cm); 
\draw [-latex,shorten >=5pt] (0cm,2cm) to node [rectangle,draw,fill=white,sloped,inner sep=1pt] {{\tiny y}} (2cm,2cm); 
\draw [-latex,shorten >=5pt] (0cm,0cm) to node [rectangle,draw,fill=white,sloped,inner sep=1pt] {{\tiny x}} (0cm,2cm); 
\draw [-latex,shorten >=5pt] (2cm,0cm) to node [rectangle,draw,fill=white,sloped,inner sep=1pt] {{\tiny y}} (2cm,2cm); 
\draw [-latex,shorten >=5pt] (2cm,2cm) to node [rectangle,draw,fill=white,sloped,inner sep=1pt] {{\tiny z}} (0cm,0cm); 
\node at (0cm,0cm) [circle,draw,fill=white,minimum size=10pt,inner sep=1pt] {\mbox{\tiny $2$}}; 
\node at (0cm,2cm) [circle,draw,fill=white,minimum size=10pt,inner sep=1pt] {\mbox{\tiny $1$}}; 
\node at (2cm,0cm) [circle,draw,fill=white,minimum size=10pt,inner sep=1pt] {\mbox{\tiny $1$}}; 
\node at (2cm,2cm) [circle,draw,fill=white,minimum size=10pt,inner sep=1pt] {\mbox{\tiny $3$}}; 
\end{tikzpicture}
\end{tabular}
\end{center}
On the right hand side $\qpol$ is embedded in a torus while its dual, $\mirror{\qpol}$, sits in a sphere and its 3 vertices 
are the 3 punctures in the commutative picture. 
The dual can be obtained by flipping over the clockwise faces, reversing their arrows and gluing everything back again.

By construction $\fuk(\mirror \qpol)$ will generate the full wrapped Fukaya category. On the algebra side it is not completely 
clear whether $\mf(\qpol)$ generates the full category of matrix factorizations $\MF(\cJA \qpol,\ell)$ 
because unless $\qpol$ sits inside a torus, the Jacobi algebra is not Noetherian. 

The outline of the paper is as follows: first we review some of the basics of $\cA_\infty$-structures, quivers, dimer models and some algebras associated to them.
We combine these subjects to look at a special $\cA_\infty$-structure on certain dimer models, called rectified dimers. Then we turn to both sides of mirror symmetry. 
First we show that the wrapped Fukaya category associated to any polyhedral subdivision of a Riemann-surface gives rise to one of the $\cA_\infty$-structures we considered.
Secondly we show that matrix factorizations of a consistent dimer model also give rise to such an $\cA_\infty$-structure 
and finally we tie the two sides together by constructing an explicit duality on dimer models.
We end with a discussion about the connection between the commutative results in \cite{Auroux} and
the noncommutative geometry we employed. In view of the readability of the paper, we deferred the 
proof of the classification of $\cA_\infty$-structures to an appendix.
The paper also includes a second appendix by Mohammed Abouzaid which explains
how one can simplify the construction of a wrapped fukaya category in the case of punctured Riemann surfaces.

\section{Acknowledgments}

I want to thank the anonymous referee for his helpful comments and thorough review of the earlier versions of the paper. I also want to thank Mohammed Abouzaid for helping me understand the delicate nature of the wrapped Fukaya category and providing the appendix. Finally, I want to thank Peter Jorgensen and Stefan Kolb for the seminar series on cluster categories of surfaces that inspired part of this paper.

\section{$\cA_{\infty}$-categories}

In this section we will introduce the basics of $\cA_{\infty}$-categories. For more information we refer to \cite{Keller,Konsoib2}.

An \emph{$\cA_{\infty}$-category $\cC$ with degrees in $\Z$} consists of the following data\footnote{Everything in this section can be generalized to $\cA_{\infty}$-categories with degrees in $\Z_2$ or any other cyclic group.}
\begin{itemize}
 \item a set of objects $\Ob\cC$,
 \item for each pair of objects $X,Y\in \Ob\cC$ a complex $\Z$-graded vector space $\Hom_{\cC}(X,Y)$,
 \item for each sequence of $n+1$ objects $X_0,\dots X_k\in \Ob\cC$ a multilinear map $\mu_k$ of degree $2-k$
\[
 \mu_k : \Hom_\cC(X_{1},X_{0})\otimes\cdots \otimes \Hom_\cC(X_{k},X_{k-1})\to \Hom_\cC(X_{k},X_{0})
\] 
(if there is no confusion we drop the subscript and just write $\mu$)
\end{itemize}
such that the identities
\[
 \mathrm{[M_k]}~~~\sum_{s+l+t=k}(-1)^{s+lt}\mu_{k-l+1}(\Id^{\otimes s}\otimes \mu_l \otimes \id^{\otimes t})=0
\]
hold for  all $k\ge 1$. To apply the identity to elements in the hom-spaces  we will use the Koszul sign rule:
$(\alpha\otimes \beta)(u\otimes v)= (-1)^{\deg \beta\cdot \deg u}\alpha(u)\otimes \beta(v)$.
For $k=1$, the identity becomes $\mu\mu(f)=0$, so each hom-space can be considered as a complex with $d:f\mapsto \mu(f)$.

An $\cA_{\infty}$-category is called \emph{strictly unital} if for every object $X$ there is an element $\id_X \in \Hom_\cC(X,X)$
of degree 0 such that
\begin{itemize}
 \item $\mu(a,\id_X)=a$ if $a$ is a homomorphism with source $X$ and $\mu_2(1_X,a)=a$ is a homomorphism with target $X$.
 \item $\mu(a_1,\dots,a_n)=0$ if $n\ne 2$ and one of the $a_i=\id_X$.
\end{itemize}
%A morphism $f\in \Hom_{\cC}(X,Y)$ in a stricty unital $A_{\infty}$-category is called a \emph{strict isomorphism}
%if there is a $g\in \Hom_{\cC}(X,Y)$ such that $\mu(f,g)=\id_Y$ and $\mu(g,f)=\id_X$.
If $\mu_i=0$ for all $i\ne 2$ a strictly unital $\cA_\infty$-category is the same as a
$\Z_2$-graded ordinary category. If $\mu_i=0$ for all $i\ge 3$ we get a dg-category.

\begin{aside}
Just like for ordinary categories we define an $\cA_\infty$-algebra as an $\cA_\infty$-category with one object and
identify the algebra with the hom-space from this object to itself.
We can turn any $\cA_\infty$-category into an $\cA_\infty$-algebra
by taking the direct sum of all hom-spaces in the original category and extending the products multilinearly and setting
products of maps that do not concatenate zero.
\end{aside}

An \emph{$\cA_\infty$-functor} $\cF: \cA\to \cB$ between two $\cA_\infty$-categories consists
of 
\begin{itemize}
 \item a map $\cF_0: \Ob \cA \to \Ob \cB$
 \item for each sequence of $k+1$ objects $X_0,\dots X_k\in \Ob\cA$ a linear map of degree $1-k$
\[
 \cF_k : \Hom_\cC(X_{1},X_{0})\otimes\cdots \otimes \Hom_\cC(X_{k},X_{k-1})\to \Hom_\cC(\cF_0X_{k},\cF_0 X_{0}) 
\] 
(if there is no confusion we drop the subscript and just write $\cF$)
\end{itemize}
subject to the following identities
\[
[\mathrm{F_k}]~ 
\sum_r\sum_{u_1+\dots+u_r=k}(-1)^{\eps}\mu_r^\cB(\cF_{u_1}\otimes\dots\otimes \cF_{u_r}) = \sum_{s+l+t=k}(-1)^{s+lt}\cF_{k-l+1}
(\id^{\otimes s}\otimes \mu_l^\cA \otimes \id^{\otimes t})
\]
with $\eps = (r-1)(u_1-1)+(r-2)(u_2-1)+\dots+ u_{r-1}$.
An $\cA_\infty$-functor is called \emph{strict} if $\cF_i=0$ for $i>1$, it is called an \emph{isomorphism} if $\cF_0$ is a bijection and all $\cF_1$ are isomorphisms and
it is called a \emph{quasi-isomorphism} if $\cF_0$ is a bijection and the $\cF_1$ induce isomorphisms on the level of the homology of $d=\mu_1$.

\subsection{Minimal models}

An $A_{\infty}$-category is called \emph{minimal} if $\mu_1=0$. Note that unlike general $\cA_\infty$-categories, minimal
$\cA_\infty$-categories are also genuine categories if we put $f_1f_2:= \mu(f_1,f_2)$ and forget the higher $\mu$'s. This is because in this case 
$[\mathrm{M_3}]$ becomes the standard associativity identity.

An \emph{$\cA_\infty$-structure on an ordinary $\Z$-graded $\C$-linear category $\cC$} is a set of multiplications $\mu$
that turn $\cC$ into a minimal $\cA_\infty$-category  such that as an ordinary category it is identical to the category structure of $\cC$.

\begin{theorem}[Kadeishvili]\cite{Kadeishvili}
Let $\cC$ be an $A_{\infty}$-category and denote by $\H\cC$ the category with the same objects but as hom-spaces the homology
of the hom-spaces in $\cC$. There is an $\cA_{\infty}$-structure on $\H\cC$ and an object-preserving quasi-isomorphism between $\H\cC$ and $\cC$.
This $\cA_{\infty}$-structure on $\H\cC$ is unique up to object-preserving $\cA_\infty$-isomorphisms. 
\end{theorem}

How do we construct this new $\cA_{\infty}$-structure? In order to do this we use a graphical method from \cite{Konsoib1}.
Set $d=\mu_1$ and for each $\Hom_{\cC}(X,Y)$ choose a map $h$ of degree $-1$ on $A$ such that
\[
 h^2=0 \text{ and } dhd=d.
\]
We will call this map a codifferential.
The map $\pi := 1- dh- hd$ is a projection and we can identify $\Image  \pi$ with $\Hom_{\H\cC}(X,Y)$ because $d\pi=\pi d=0$ and
if $dx=0$ then $\pi x= x - d(hdx)$. Let $\iota$ be the embedding $\Hom_{\H\cC}(X,Y)=\Image  \pi \subset \Hom_{\cC}(X,Y)$.

Given a rooted tree $\cT$ with $k+1$ leaves we can define a multilinear map 
\[
 m^\cT : \H\Hom_\cC(X_{1},X_{0})\otimes\cdots \otimes \H\Hom_\cC(X_{k},X_{k-1})\to \H\Hom_\cC(\cF_0X_{k},\cF_0 X_{0})
\]
by interpreting every leaf as the map $\iota$, every internal node as a map $h\mu$ and
the root as $\pi\mu$. 
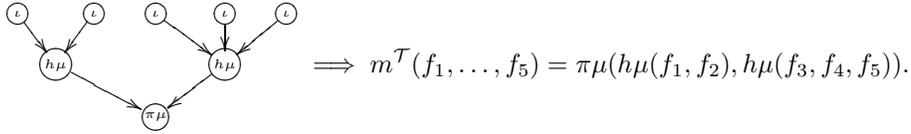
\[
\vcenter{
 \xymatrix@C=.1cm@R=.35cm{
  \vtx{\iota}\ar[rd]&&\vtx{\iota}\ar[ld]&&\vtx{\iota}\ar[rrd]&&\vtx{\iota}\ar[d]&&\vtx{\iota}\ar[lld]\\
    &\vtx{h\mu}\ar[rrrd]&&&&&\vtx{h\mu}\ar[lld]&&\\
    &&&&\vtx{\pi\mu}&&& 
 }}
\implies m^\cT(f_1,\dots, f_5)=
\pi \mu(h\mu(f_1,f_2), h\mu(f_3,f_4,f_5)).
\]
The new multiplication $\mu_n^{\H}$ is then defined as the sum $\sum_\cT m^\cT$ over all rooted trees with 
$n$ leaves for $n>1$ and $\mu_1^{\H}=0$.

\subsection{Completion}

Given an $\cA_\infty$-category $\cC$ we define a \emph{twisted object} \cite{Keller} as 
a pair $(M,\delta)$, where $M\in \N[\Ob\cC\times \Z]$ is a formal sum of objects shifted by elements in $\Z$ (or $\Z_2$ if $\cC$ is only $\Z_2$-graded).
We will write such a sum as $v_1[i_1] \oplus \dots \oplus v_k[i_k]$ where the $v_j$ are objects and the $i_j$ shifts.
The map $\delta$ is an upper-triangular $k\times k$-matrix with entries $\delta_{st}\in \Hom_{\cC}(v_{i_t},v_{i_s})$ 
of degree $i_t-i_s+1$ and subject to the identity
\[
 \sum_{n=1}^{\infty}(-1)^{\frac {n(n-1)}2}\mu_n(\delta,\dots,\delta)=0
\]
where we extended $\mu_n$ to matrices in the standard way.

The homomorphism space between two such objects $(M,\delta)$ and $(M',\delta')$ is given by
\[
 \bigoplus_{r,s} \Hom(v_r, v'_s)[i_r-i_s]
\]
which we equip with an $\cA_{\infty}$-structure as follows:
\[
 \mu(f_1,\dots,f_n) := \sum_{t=0}^{\infty}\sum_{i_0+\dots+i_n=t} \pm \mu(\underbrace{\delta,\dots,\delta}_{i_0},f_1,\underbrace{\delta,\dots,\delta}_{i_1},\dots,f_n,\underbrace{\delta,\dots,\delta}_{i_n}).
\]
The $\pm$-sign is calculated by multiplying with a factor $(-1)^{n+t-k}$ for each $\delta$ in the expression on position $k$.

The $\cA_\infty$-category of twisted objects and their homomorphism spaces is denoted by $\Tw \cC$.
It also has a minimal model, which we denote by $\H\Tw \cC$. Note that because it is a minimal model,
$\H\Tw \cC$ is a genuine category.

\begin{aside}
If $\cA$ is a genuine $\C$-linear category with a finite number of objects, such that the $\Hom_A(X,Y)$ are finite dimensional 
and contain no non-trivial idempotents
then we can consider $\cA$ as a path algebra of a quiver with relations.
We can construct the derived category $\cD\Mod \cA$ and look at the smallest
triangulated subcategory generated by $\cA$ as a module over itself.  We can construct a $\Z$-graded category from this
by putting $\Hom(X,Y) := \oplus_{i \in \Z}\Hom_{\cD\Mod \cA}(X,Y[i])$.

On the other hand we can view $\cA$ as an $\cA_\infty$-category with degrees in $\Z$, by putting all of $\Hom_{\cA}(X,Y)$ in 
degree $0$. It makes sense to look at $\H\Tw \cA$ and it turns out that this category is equivalent to the category we defined above (see  \cite{Keller}).
In this light $\H\Tw\cC$ can be seen as a useful generalization of the derived category of an algebra.
\end{aside}

\section{Embedded quivers and Dimers}

\subsection{Quivers}
As usual a \emph{quiver} $Q$ is a finite (or locally finite) oriented graph. We denote the set
of vertices by $Q_0$, the set of arrows by $Q_1$ and the maps $h,t$
assign to each arrow its head and tail.
A \emph{nontrivial path} $p$ of length $k$ is a sequence of arrows $a_{k-1}\cdots a_0$
such that $t(a_i)=h(a_{i+1})$. We write $p[i]$ to denote the arrow $a_i$ and we 
set $h(p)=h(p[k-1])$ and $t(p)=t(p[0])$.
 
A \emph{trivial path} is just a vertex.  
A path $p$ is called \emph{cyclic} if $h(p)=t(p)$ and the equivalence class of a cyclic path under
cyclic permutation is called a \emph{cycle}.

The \emph{path category} $\C Q$ is the category with as
objects the vertices, as homomorphisms linear combinations of paths and as composition
concatenation. We will concatenate our paths from right to left: $pq = \xymatrix{\vtx{}&\vtx{}\ar[l]|p&\vtx{}\ar[l]|q}$.
Every vertex $v$ corresponds to an object $v$ but it can also be seen as the trivial path, which is the identity
morphism $\id_v$ on $v$. If there is no confusion we will also use the notation $v$ to denote $\id_v$.
The path category can also be considered as an algebra $\C Q$ by taking the direct sum of all hom-spaces. In this way the vertices
become idempotents and we can recover $\Hom_{\C Q}(v,w)$ as $w\C Q v$. 

Given a quiver $Q$, one can construct its double, $\bar Q$, which has the same number of vertices
but for every arrow $a \in Q_1$ we add an extra arrow $a^{-1}$ with $h(a)=t(a^{-1})$ and $t(a)=h(a^{-1})$.
The \emph{weak path category} of $Q$ is the following quotient
\[
 \C \hat Q := \frac{\C \bar Q}{\<aa^{-1}=h(a), a^{-1}a=t(a)|a \in Q_0\>} 
\]
A \emph{weak path in $Q$} is a path in $\bar Q$, viewed as a homomorphism in the weak path category.
We also speak of weak arrows and cycles and if we want to stress that a path or cycle is not weak we will call it real. 

A quiver is called \emph{embedded in a surface $S$}, if $Q_0$ is a discrete subset of a smooth surface $S$ without boundary
and every arrow is a smooth embedding $a : [0,1] \to S$ such that $h(a)=a(1)$, $t(a)=a(0)$ and
different arrows only intersect in end points. We identify $a^{-1}$ with the map $a$ in reverse direction.
We will sometimes denote the surface in which the quiver embeds by $|Q|$.
We do not require $|Q|$ to be compact, but we will exclude surfaces with boundary. 

We say that an embedded quiver $Q$ \emph{splits a compact surface $S$} if 
the complement of the quiver consists of a disjoint union of open discs and none of the arrows is a contractible loop.
Each of these discs is bounded by a weak path $c$ of length at least $2$ that goes around it in counter-clockwise direction. 
The closure of the discs are called the faces of the embedded quiver,
the counter-clockwise cycles on the boundary are called boundary cycles and we collect these cycles in 
a set $\hat Q_2^+$.

\subsection{Covers and group actions}
A cover map between two embedded quivers $Q\subset S$ and $Q' \subset S'$ is an orientation-preserving unbranched cover map
$\pi:S\to S'$ such that $Q_0 = \pi^{-1}(Q'_0)$ and $Q_1$ is the set of all lifts of arrows in $Q_1'$. Given any unbranched cover map $S\to S'$ we
can reconstruct $Q$ out of $Q'$, so it also makes sense to speak of the universal cover of an embedded quiver.

To any cover map we can associate the group of deck transformations $\grp G:= \Aut(S\to S')$ consisting of all diffeomorphisms $\phi:S\to S$
such that $\pi\circ \phi=\pi$. 
Each deck transformation acts on the embedded quiver $Q$ and if the group of deck transformatins is finite the orbits of the arrows and vertices are all of the same size
because a deck transformation has no fixed points.
This gives an action of a finite group $G$ onto $\C Q$ and we can identify $\C Q'$ with the invariant subalgebra $\C Q^{\grp G}$ by the following 
embedding
\[
\C Q' \to \C Q : \pi(p) \to \sum_{g \in \grp G}p^g
\]
where $p$ is any path in $\C Q$.

We are interested in constructions that associate to certain embedded quivers an algebra $\cB(Q)$ which
is the quotient of the path algebra equipped with an extra $\cA_\infty$-structure. We will
say that such a construction is \emph{compatible with covers} if for every finite cover $S\to S'$ the action 
of the group of deck transformations
on $\C Q$ projects down to $\cB(Q)$. Moreover we want this action to be equivariant for the $\cA_\infty$-structure 
$(g\circ \mu_k = \mu_k\circ g^{\otimes k})$ and $\cB(Q)^{\grp G}\cong \cB(Q')$.

Note that a construction $\C Q/\cI$ without additional $\cA_\infty$-structure 
is compatible with covers if $g\cI=\cI$ for any cover $Q\to Q'$,
$\cI$ is generated by all lifts of $\cI'\subset \C Q'$. 

\subsection{Dimers}

A \emph{dimer model} \cite{HanKen,FranHan,HanHer} is a quiver that splits a \emph{compact orientable surface} and
for which every boundary cycle \emph{has length at least 3} and \emph{is either real or the inverse of a real cycle}.
The real cycles that are boundary cycles are called the positive cycles and are grouped in the set $Q_2^+:= \hat Q_2^+ \cap \C Q$.
Negative cycles are those for which the inverse is a boundary cycle and we set $Q_2^-:=(\hat Q_2^+)^{-1} \cap \C Q$.
Note that the orientability of the surface implies that every arrow will be contained in one positive cycle 
and one negative.

\begin{example}\label{dimex}
We give 4 examples of embedded quivers that split a surface. The first 3 are embedded in a torus, 
the last in a genus 2 surface.
Arrows and vertices with the same label are identified.
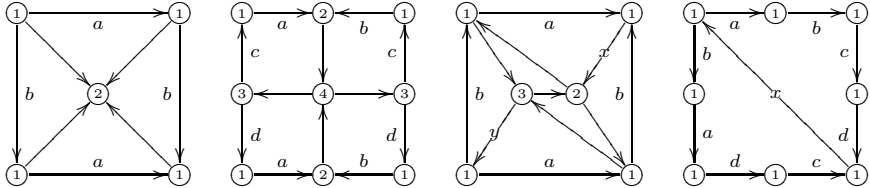
\[
\xymatrix@C=.75cm@R=.75cm{
\vtx{1}\ar[rr]_a\ar[dd]^b\ar[dr]&&\vtx{1}\ar[dd]_b\ar[dl]\\
&\vtx{2}&\\
\vtx{1}\ar[rr]^a\ar[ur]&&\vtx{1}\ar[ul]
}
\hspace{.5cm}
\xymatrix@C=.75cm@R=.75cm{
\vtx{1}\ar[r]_{a}&\vtx{2}\ar[d]&\vtx{1}\ar[l]^b\\
\vtx{3}\ar[u]_c\ar[d]^d&\vtx{4}\ar[l]\ar[r]&\vtx{3}\ar[u]^c\ar[d]_d\\
\vtx{1}\ar[r]^{a}&\vtx{2}\ar[u]&\vtx{1}\ar[l]_b
}
\hspace{.5cm}
\xymatrix@C=.4cm@R=.75cm{
\vtx{1}\ar[rrr]_a\ar[dr]&&&\vtx{1}\ar[ld]|x\\
&\vtx{3}\ar[r]\ar[ld]|y&\vtx{2}\ar[ull]\ar[dr]&\\
\vtx{1}\ar[rrr]^a\ar[uu]_b&&&\vtx{1}\ar[ull]\ar[uu]^b
}
\hspace{.5cm}
\xymatrix@C=.75cm@R=.75cm{
\vtx{1}\ar[r]_{a}\ar[d]^{b}&\vtx{1}\ar[r]_b&\vtx{1}\ar[d]_c\\
\vtx{1}\ar[d]^a&&\vtx{1}\ar[d]_d\\
\vtx{1}\ar[r]^{d}&\vtx{1}\ar[r]^c&\vtx{1}\ar[uull]|x
}
\]
The last 3 are dimer models, the first one is not because the faces are not bounded by cycles. 
\end{example}

The \emph{Jacobi category of a dimer model} is the quotient of the path category by the ideal generated by relations
of the form $r_a := r_+-r_-$ where $r_+a \in \qpol_2^+$ and $r_-a\in \qpol_2^-$ for some arrow $a\in \qpol_1$:
\[
 \cJA(\qpol) := \frac{\C \qpol}{\<r_a| a \in \qpol_1\>}.
\]
For the last dimer model in example \ref{dimex} we have $\qpol_2^+=\{abxcd\}$ and $\qpol_2^-=\{baxdc\}$ and
\[
 \cJA(\qpol) := \frac{\C \<a,b,c,d,x\>}{\<bxcd-xdcb, xcda-axdc, cdab-dcba, dabx -baxd, abxc-cbax\>}.
\]
Note that because we demanded that the positive and negative cycles have at least $3$ arrows, all
terms $r^+$, $r^-$ in the relations are at least quadratic. This ensures that we can recover the quiver from
$\cJA(\qpol)$.

\begin{remark}
It is clear from the definition that the Jacobi category construction is compatible with covers because
the positive and negative cycles in the cover are precisely the lifts of positive and negative cycles.
Therefore the relations in the cover are precisely the lifts of the relations in the original dimer.
\end{remark}

\begin{aside}
The reason these Jacobi algebras are important is because they appear as noncommutative analogues
of Calabi-Yau manifolds. A compact 3-Calabi-Yau manifold can be defined as a smooth variety $X$ which has the following duality
\[
\Ext_X^i(\ccF, \ccG) = \Ext_X^{3-i}(\ccG,\ccF)^*
\]
for all coherent sheafs $\ccF$, $\ccG$ on $X$.

Similarly a 3-Calabi-Yau algebra \cite{Ginzburg} can be defined such that it has a similar duality
\[
\Ext_X^i(M, N) = \Ext_X^{3-i}(N,M)^*
\]
for all finite dimensional left $A$-modules.
\end{aside}

\subsection{Consistency}
Not every Jacobi algebra will be Calabi-Yau, this will depend on the structure of the dimer model.
To characterize such dimer models we introduce a notion of consistency. Several different notions are available in the literature \cite{gulotta,broomhead,davison,MR,IUcons,Bocklandtcons}
but we will restrict to one: zigzag consistency.

Fix a dimer model $\qpol$ and its universal cover $\ucover\qpol$, which is again a dimer. For any arrow $\tilde a\in \tilde \qpol_1$ 
we can construct its \emph{zig ray} $\pZ_{\tilde a}^+$. This is an infinite path
\[
\dots \tilde a_2\tilde a_1\tilde a_0
\]
such that $\tilde a_0=\tilde a$ and $\tilde a_{i+1}\tilde a_{i}$ sits in a positive cycle if $i$ is even and in a negative cycle if $i$ is odd.
Similarly the \emph{zag ray} $\pZ_{\tilde a}^-$ is the path where  $\tilde a_{i+1}\tilde a_{i}$ sits in a positive cycle if $i$ is odd and in a negative cycle if $i$ is even.
The projection of a zig or a zag ray down to $Q$ will give us a cyclic path if $Q$ is finite. Such a cyclic
path will be called a \emph{zigzag cycle}.
A dimer model is called \emph{zigzag consistent} if for every arrow $\tilde a$ the zig and the zag ray only meet in $\tilde a$:
\[
 \pZ_{\tilde a}^-[i]= \pZ_{\tilde a}^+[j] \implies i=j=0.
\]

In example \ref{dimex} the second and fourth quiver are zigzag consistent.
The third quiver is not zigzag consistent 
because $\pZ_{\tilde x}^-[3]= \pZ_{\tilde x}^+[3]=\tilde y$.
Note that a dimer model on a sphere is never zigzag consistent as its universal cover is finite. 

\begin{theorem}
If a dimer model is zigzag consistent then its Jacobi Algebra is 3-Calabi-Yau.
\end{theorem}
\begin{remark}\label{historyconsistent}
This result follows from the culmination of work by many people.
There are several notions of consistency for dimer models, the most basic being the
cancellation property. This property states that $\cJA(\qpol)$ embeds in the weak Jacobi algebra $\C\hat \qpol/\<r_a| a \in \qpol_1\>$.
In \cite{MR} Mozgovoy and Reineke showed that if a dimer on an aspherical surface
satisfies the cancellation property and a second technical condition, then its Jacobi algebra is 3-Calabi-Yau.
In \cite{davison} Davison proved that the cancellation property alone implies the Calabi-Yau property and
that this second technical condition follows from the cancellation property. In \cite{Bocklandtcons} it is
shown that zigzag consistency implies the cancellation property. 
Similar results are also obtained by Ishii and Ueda in \cite{IU} and \cite{IUcons}.

In \cite{broomhead} Broomhead introduces the notion of geometric consistency and algebraic consistency and proves
the former implies the latter, which then implies the Calabi-Yau property. Geometric consistency is a stronger condition
than cancellation as it only applies to genus 1-surfaces and there are also examples of zigzag consistent dimers on the torus
that are not geometrically consistent. It is also equivalent to the notion of properly ordered dimers introduced by Gulotta in \cite{gulotta}. 
Geometric consistency can be relaxed to the notion of $R$-charge consistency: the existence of
a $\R_{>0}$-grading $\cR$ such that for every cycle $c \in \qpol_2$ $\sum_{a\in c}\cR_a=2$ and every vertex $v$
$\sum_{h(a)=v}(1-\cR_a)+ \sum_{t(a)=v}(1-\cR_a)=0$. In \cite{Bocklandtcons} it is shown that $\cR$-charge consistency is equivalent
to zigzag consistency on a torus. 
\end{remark}

\section{Rectified quivers and gentle categories}\label{rect}

\subsection{Rectification}
Given a quiver $\qpol$ that splits a surface $S$, we construct a new embedded quiver in the same surface, 
which we call \emph{the rectified quiver} $\rect \qpol$.
\begin{itemize}
 \item The vertices of $\rect \qpol$ are the centres of the arrows in $\qpol$
\[
\rect \qpol_0 := \{v_a := a(1/2)| a\in Q_1\} .
\]
 \item The arrows of $\rect{\qpol}$ are segments $\alpha = \lvect{v_av_b}$ connecting the centres of two (weak) arrows that follow each other 
counterclockwise in one of the boundary cycles (for both clockwise and anticlockwise boundary cycles). 
\end{itemize}
Each of the boundary cycles of $\qpol$ will give a positive cycle in
$\rect \qpol$ and each of the vertices in $\rect \qpol$ will give us a negative cycle.
We chose the name rectified quiver because in Euclidean geometry the process of cutting of
the vertices of a polyhedron at the midpoints of the edges is called rectification.

To each of the arrows we assign a $\Z_2$-degree. If $abr$ is a boundary cycle of $\qpol$ and
$a$ and $b$ are both real or both weak we give the arrow $\alpha =\lvect{v_av_b}$ degree $1$.
Otherwise we give it degree 0.

Furthermore we will use Greek letters to denote paths and arrows in a rectified quiver and Roman letters
for arrows and paths in the original embedded quiver. Note that the rectified quiver is not always a dimer in the strict sense
because some of the positive or negative cycles can have length $2$.
 
\begin{example}
Below are two examples of rectification. In the first we rectify a quiver on a torus with one vertex and two loops to obtain
a dimer model with two vertices and four arrows. In the second example we rectify a tetrahedron to obtain an octahedron.
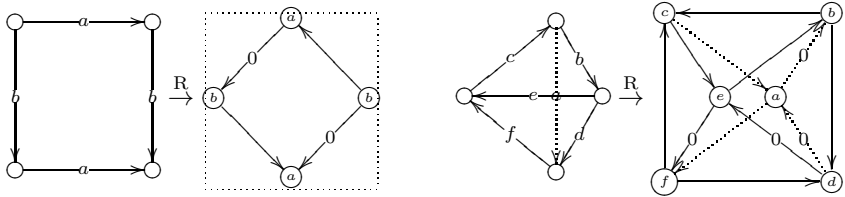
\[
\vcenter{\xymatrix@C=.7cm@R=.75cm{
\vtx{}\ar[rr]|a\ar[dd]|b&&\vtx{}\ar[dd]|b\\
&&\\
\vtx{}\ar[rr]|a&&\vtx{}
}}
\stackrel{\rect{}}{\to}
\vcenter{\xymatrix@C=.7cm@R=.75cm{
\save [].[dd].[ddrr].[rr]*[F.]\frm{}="back"
\restore
&\vtx{a}\ar@{<-}[rd]&\\
\vtx{b}\ar@{<-}[ur]|0&&\vtx{b}\ar[ld]|0\\
&\vtx{a}\ar@{<-}[lu]&
}}
\hspace{1cm}
\vcenter{
\xymatrix@C=.4cm@R=.75cm{
&&\vtx{}\ar@{.>}[dd]|a\ar[rd]|b&\\
\vtx{}\ar[rru]|c&&&\vtx{}\ar[ld]|d\ar[lll]|e\\
&&\vtx{}\ar[llu]|f&
}}
\stackrel{\rect{}}{\to}
\vcenter{\xymatrix@C=.4cm@R=.8cm{
\vtx{c}\ar@{<-}[dd]\ar@{<-}[rrr]&&&\vtx{b}\ar@{<-}[dll]\ar@{<.}[dl]|0\\
&\vtx{e}\ar@{<-}[ul]\ar@{<-}[drr]|0&\vtx{a}\ar@{<.}[ull]\ar@{<.}[dr]|0&\\
\vtx{f}\ar@{<-}[ur]|0\ar@{<.}[urr]&&&\vtx{d}\ar@{<-}[uu]\ar@{<-}[lll]
}}
\]
To indicate the grading we marked the degree zero arrows. The unmarked arrows all have degree 1.
\end{example}

\subsection{Gentle categories}
Instead of associating to this rectified quiver a Jacobi category, we will look at a different category.
The \emph{gentle category of a rectified quiver} is the quotient of the path category by the ideal generated by relations
of the form $\beta_i\beta_{i+1}$ where $\beta=\beta_1\dots \beta_l$ is a positive cycle (the index $i$ should be interpreted mod $l$) 
\[
 \cQA(\rect\qpol) := \frac{\C \rect \qpol}{\<\beta_i\beta_{i+1}| \beta \in \rect \qpol_2^+\>}.
\]

\begin{remark}
The gentle category of a rectified dimer is a generalization of the gentle algebra of a triangulation which was introduced by 
Assem, Br\"ustle, Charbonneau-Jodoin and Plamondon in \cite{plamondon}. If the original quiver that splits the surface comes from a triangulation, these two algebras coincide.
One can easily check that $\cQA(\rect \qpol)$ is a gentle algebra in the sense of \cite{plamondon} and \cite{gentleassem}.
\end{remark}

\begin{remark}
If we consider $\rect\qpol$ as a quiver embedded in $S$, it is again
clear from the definition that the gentle category construction is compatible with covers because
the positive cycles in the cover are precisely the lifts of positive cycles.

Sometimes it makes more sense to look at $\rect\qpol$ as a quiver embedded in the punctured surface $S\setminus \qpol_0$ 
instead of $S$. In this case it is also clear that the construction of the gentle algebra is compatible with punctured covers.
\end{remark}

\begin{lemma}\label{basispaths}
The subpaths of powers of negative cycles in $\rect \qpol$ form a basis for $\cQA(\rect \qpol)$.
\end{lemma}
\begin{proof}
Because of the nature of the defining relations of $\cQA(\rect \qpol)$, the set of all nonzero paths forms a basis.
Any nonzero nontrivial path in $\rect \qpol$ can only be extended in one way to a nonzero path that is one arrow longer:
one must add the arrow that sits in the same negative cycle as the last arrow. 
So, if $\rho$ is a nonzero path it must be a subpath of a power of a negative cycle.
\end{proof}

We end this section with a property of rectified quivers that will be important in appendix \ref{appendix}.
\begin{definition}\label{well-behaved}
We will call a rectified quiver \emph{well-behaved} if 
every arrow that connects two vertices of a positive boundary cycle is an arrow of that boundary cycle. 
\end{definition}
The main idea behind this notion is that for well-behaved quivers certain calculations become easier and we can cover non-well-behaved quivers 
by well-behaved ones, do the calculations in the cover and project them down.

\section{An $\cA_\infty$-structure}
We will now describe a specific $\cA_{\infty}$-structure on $\cQA(\rect\qpol)$ which
can be constructed inductively.
For any sequence of paths $\rho_1,\dots,\rho_k$ and any cycle $\beta_1\dots \beta_l \in \rect Q_2^+$ with $h(\beta_1)=t(\rho_i)$ 
we set 
\[
\genmu(\rho_1,\dots, \rho_i\beta_{1},\beta_{2},\dots, \beta_{l-1},\beta_{l}\rho_{i+1},\dots, \rho_k) := (-1)^s \genmu(\rho_1,\dots,\rho_k)
\]
with sign convention $s =l(\rho_1+\dots + \rho_{i} + k-i)$. Pictorially this gives rise to the following diagram:
\[
\genmu\left(
\vcenter{\xymatrix@=.3cm{&\ar[dr]&\\
\ar[ur]&&\ar[lldd]_(.8){\rho_{i}\beta_1}\\
&&\\
\dots&&\dots\ar[lluu]_(.2){\beta_l\rho_{i+1}}
}}
\right)=\pm 
\genmu\left(
\vcenter{\xymatrix@=.3cm{
&\ar[ld]_{\rho_{i}}&\\
\dots&&\dots\ar[lu]_{\rho_{i+1}}
}}
\right).
\]
For $k>2$ we set $\genmu(\sigma_1,\dots, \sigma_k)=0$ if we cannot perform any reduction of the form above and for 
$k=2$ we use the ordinary product on $\cQA(\rect \qpol)$. 

\begin{remark}
Note that the construction of this $\cA_\infty$-structure is compatible with both punctured and unpunctured covers because
the reduction rule is equivariant for deck transformations of these two types of covers: the reduction rules lift
to reduction rules in the cover.
\end{remark}

\begin{remark}
We can turn $\genmu$ also in a $\Z$-graded $\cA_{\infty}$-structure. This can be done by assigning to each arrow $\beta$ in $\rect \qpol$ a degree
$\deg \beta$ such that every positive cycle $\deg\beta_1\dots\beta_l=l-2$. This makes the reduction rule homogeneous. In this case we will call $\deg$ a compatible $\Z$-degree. Because every arrow only sits in one positive cycle it is always possible to find a compatible $\Z$-degree.
\end{remark}

\begin{theorem}\label{mainthm}
Let $\qpol$ be an embedded quiver that splits a compact surface.
\begin{enumerate}
 \item[1] $\genmu$ defines a $\cA_{\infty}$-structure on $\cQA(\rect \qpol)$. 
 \item[2] If $\tilde\mu$ is an $\cA_\infty$-structure on $\cQA(\rect \qpol)$ such that
for all paths $\rho_1,\dots,\rho_k$ we can find a $\lambda\in \C^*$ such that 
\[
\tilde\mu(\rho_1,\dots,\rho_k) = \lambda \genmu(\rho_1,\dots,\rho_k),  
\]
then $\tilde \mu$ is isomorphic to $\genmu$.
\end{enumerate}
Moreover, if $\qpol$ is a dimer, $\rect \qpol$ is well-behaved and $\deg$ is a compatible $\Z$-degree then
\begin{enumerate}
\item[3]
If $\tilde\mu$ is a $\deg$-homogeneous $\cA_\infty$-structure on $\cQA(\rect \qpol)$ such that
for every cycle $c=\beta_1\dots \beta_{l} \in \rect \qpol_2^+$ there is a $\lambda\in \C^*$ such that
\[
 \tilde\mu(\beta_{i},\dots,\beta_{j})=\begin{cases}
                       \lambda h(\beta_i)&j-i+1=l\\
                       0&j-i+1\ne l
                      \end{cases}
\]
then $\tilde\mu$ is $\cA_\infty$-isomorphic to $\genmu$.
\item[4]
If $\grp G$ is a finite group that acts on $\cQA(\rect \qpol)$ $\deg$-homogeneously such that $\tilde\mu$ and $\genmu$ are equivariant, then we can choose the $\cA_\infty$-isomorphism between 
$\tilde\mu$ and $\genmu$ to be $\grp G$-equivariant.
\end{enumerate}
\end{theorem}
\begin{proof}
This is a combination of theorems \ref{welldef}, \ref{uptoiso}, \ref{uptoaiso} and lemma \ref{gaction}. Proofs of these
results can be found in the appendix. 
\end{proof}

Different quivers on the same surface
will give different gentle $\cA_\infty$-categories, but these categories are closely related.
In fact if we go to the twisted completion the difference disappears.

\begin{lemma}\label{reversedirection}
Suppose $\qpol$ is an embedded quiver that splits the surface $|\qpol|$ and $a$ is one of the arrows of $\qpol$.
If $\qpol'$ is the new quiver obtained by changing the direction of $a$ (i.e. swapping its head and tail) then 
$\Tw\cQA (\rect \qpol),\genmu$ and $\Tw\cQA(\rect \qpol')$ are isomorphic $\Z_2$-graded $\cA_{\infty}$-categories.
\end{lemma}
\begin{proof}
Let $v_0$ be the object in $\cQA(\rect \qpol)$ corresponding to the arrow $a$ we want to reverse. We denote the corresponding
object in $\cQA (\rect \qpol')$ as $v_0'$. All other objects we denote by $v_i$ in both categories.
We now define a strict functor $\cF: \Tw\cQA( \rect Q)\to \Tw\cQA (\rect Q')$:
\begin{itemize}
 \item $\cF_0( v_0[0]^{\oplus m_0}+v_0[1]^{\oplus m_1}+\rest ,\delta)=( v_0'[1]^{\oplus m_0}+v_0'[0]^{\oplus m_1}+\rest ,\delta)$
 \item $\cF_1(f)=f$
\end{itemize}
It is easy to check that the degrees match up and that this is an isomorphism.
\end{proof}
The lemma above allows us to bring the arrows in a given cycle all in the same direction, without changing
the twisted completion. In the next lemma we are going to investigate what happens if one splits one of the cycles in 2.

\begin{lemma}\label{putinarrow}
Suppose $\qpol$ is an embedded quiver that splits the surface $|\qpol|$.
Suppose that $a_1\dots a_k \in \C \qpol$ is a boundary cycle and let $b$ be a new arrow in this face connecting $h(a_1)$ and
$h(a_i)$ with $2<i<k$. Denote the quiver obtained by adding $b$ to $\qpol$ as $\qpol'$,
then $\H\Tw \cQA (\rect \qpol)$ and $\H\Tw \cQA (\rect \qpol')$ are equivalent as $\Z_2$-graded $A_\infty$-categories.
\end{lemma}
\begin{proof}
The assumption that $a_1\dots a_k$ is a boundary cycle implies that it goes counterclockwise around its face in the surface, which will be important
for the degrees of the arrows in the rectified quiver.

Let $v_0$ be the object in $\cQA(\rect Q')$ corresponding to the arrow $b$ we want to add.
Denote the object corresponding to the arrow $a_j$ by $v_j$ in both categories.
We use $\alpha_j$ to denote the arrow between $v_{j+1}$ and $v_{j}$ and the arrow 
$\beta_0$ connects $v_1$ with $v_0$ and $\beta_i$ connects $v_0$ to $v_i$. All these arrows have degree $1$.
Finally there are the degree zero arrows
$\beta_k$ and $\beta_{i+1}$ which connect $v_0$ to $v_k$ and $v_{i+1}$ to $v_{0}$.

The $A_\infty$-category $\cQA(\rect Q)$ is a full subcategory of $\Tw\cQA(\rect Q')$ because
we can identify $\alpha_i$ with $\beta_{i+1}\beta_i$ and $\alpha_k$ with $\beta_0\beta_{k}$.
Also the products are compatible because if 
$$\mu(\dots, p\alpha_k,\dots, \alpha_1q,\dots)\to \mu(\dots, p,q,\dots)$$ 
is a valid reduction for $\rect Q$ then
$$\mu(\dots, p\beta_0\beta_{k},\dots,\beta_{i+1}\beta_i,\dots, \alpha_1q,\dots)\to 
\mu(\dots, p\beta_0,\beta_i,\dots, \alpha_1q,\dots)\to \mu(\dots, p,q,\dots)$$ 
is a double reduction in $\rect Q'$.
Analogous reductions can be made to reduce a cyclic permutation of $\alpha_1\dots \alpha_k$.

To show that $\H\Tw \cQA (\rect \qpol)$ and $\H\Tw \cQA (\rect \qpol')$ are equivalent, it suffices to find
a complex in $\H\Tw \cQA (\rect \qpol)$ that is isomorphic to $v_0$ in $\H\Tw \cQA (\rect \qpol')$.

The complex we are looking for is
\[
w = \left( v_1[1]+\dots+v_i[1],\delta :=\smatrix{0&\alpha_1&&\\
&\ddots&\ddots&\\
&&&\alpha_{i-1}\\
&&&0
}\right).
\]
Indeed we have a map $f_1 :v_0\to w$ given by $\smatrix{\beta_0&0&\cdots&0}^\top$ and a map $f_2:w\to v_0$ given by $\smatrix{0&\cdots&0&\beta_i}$.

These maps are each other's inverses in $\H\Tw\cQA(\rect\qpol')$.
Clearly $$\mu_2(f_1,f_2)=\mu(f_1,\delta,\dots,\delta,f_2)=\mu(\beta_0,\alpha_1,\dots,\alpha_i,\beta_i)=\Id_{v_0}.$$
On the other hand we have
\se{
 \mu(f_2,f_1) &= \mu(f_2,f_1,\delta,\dots, \delta) + \mu(\delta,f_2,f_1,\delta,\dots, \delta)+ \dots + \mu(\delta,\dots, \delta,f_2,f_1)\\
&=\left(\begin{smallmatrix}
\mu(\alpha_1,\dots,\alpha_{i-1},\beta_i,\beta_0)&&&\\
&\mu(\alpha_2,\dots,\alpha_{i-1},\beta_i,\beta_0,\alpha_1)&&\\
&&\ddots&\\
&&&\mu(\beta_i,\beta_0,\alpha_1,\dots,\alpha_{i-1})\\
\end{smallmatrix}\right)\\
&=
\smatrix{
\Id_{v_1}&&&\\
&\Id_{v_2}&&\\
&&\ddots&\\
&&&\Id{v_i}\\
}=
\Id_{w}.
}
So $w$ and $v_0$ are isomorphic in $\H\Tw\cQA(\rect\qpol')$.
\end{proof}

\begin{corollary}
As an $\cA_{\infty}$-category $\H\Tw \cQA(\rect \qpol),\genmu$ only depends on the genus of the surface and the number of vertices of $\qpol$.
\end{corollary}
\begin{proof}
By adding enough arrows we can turn $\qpol$ into a triangulation, so the statement only needs to be
proven for triangulations. By \cite{Hatcher} we know that
two triangulations of punctured surfaces can be turned into each other by a process of mutation.

This process removes an arrow to create a quadrangle and then puts in a new arrow coming from the other diagonal of the quadrangle.
By lemma \ref{putinarrow} this does not change $\H\Tw\cQA(\rect \qpol)$.
\end{proof}
 
\section{The relation with the wrapped Fukaya category}

A Liouville structure on a manifold with punctures is a 1-form $\theta$ such that $\omega=d\theta$ is a symplectic form.
The symplectic form allows us to transform the 1-form into a vector field which generates a flow called the Liouville flow and we demand that the Liouville flow points towards the punctures near them.
To any Liouville manifold one can associate an $A_{\infty}$-category called the wrapped Fukaya category.
The objects are graded exact Lagrangian submanifolds which are invariant under the Liouville flow outside a compact subset of the punctured manifold. 

To an embedded quiver $\qpol$ that splits a compact surface $S$ 
one can associate a punctured surface by removing the vertices: $S\setminus \qpol_0$ and put a Liouville structure on it.
In this case any curve that connects two punctures
is an exact Lagrangian submanifold. Therefore
it is natural to consider each arrow $a \in \qpol_1$ as an object $\cLL_a$ in the wrapped Fukaya category of $S\setminus \qpol_0$. 

The hom-spaces in the wrapped Fukaya category are quite tricky to define 
and for details we refer to \cite{Abou2}. Using results by Abouzaid from appendix \ref{appendixabouzaid} we can simplify the construction a lot in our setting.  

A \emph{Reeb chord} between two Lagrangians $\cLL_0$ and $\cLL_1$ is a time 1 chord for the flow $\phi_H$ of a fixed Hamiltonian $H:S\setminus\qpol_0\to \Rl$ which is quadratic at infinity (i.e. near the punctures).
In other words these are curves $\gamma: [0,1]\to S: t \mapsto \phi_H^t(\gamma(0))$ such that $\gamma(0)\in \cLL_0$ and $\gamma(1)\in \cLL_1$. Reeb chords are also in one to one correspondence with intersection points between $\phi_H^1\cLL_0$ and $\cLL_1$. As a vector space $\Hom(\cLL_0,\cLL_1)$ is defined as the $\C$-linear span of the set of Reeb chords. 

To give this space a $\Z$-grading, we
choose a real vector field $X$ on the compact surface $S$ with zeros at the punctures. A grading of the Lagrangian submanifold $\cLL$ is a smooth function $g : \cLL \to \Rl$ such that $g(p) \mod 2\pi$ is the angle between $\frac{d\cLL}{dt}|_p$ and $X_{p}$.\footnote{Note that although we need a complex structure on $S$ to define angles, the definition of the grading does not depend on this because it can also be framed in terms of lifts of the Lagrangians to the line bundle that covers the circle bundle $TS/\Rl^+$.}
Adding $\pi$ to $g$ gives a grading of the Lagrangian submanifold with reversed orientation. 

Because the arrows split the surface in polygons we can easily find a vector field $X$ such that all the Lagrangians coming from arrows are integral curves of $X$. In this case we can set all $g's$ to zero.

The flow transports the grading on $\cLL_0$ to a grading 
on $\phi_H^1(\cLL_0)$ and the degree we assign 
to a Reeb chord $\gamma$ is the $k$ for which $\phi_H^1g_0(\gamma(0))-g_1(\gamma(1)) + k\pi \in [0,\pi)$.

For our special vector field $X$ the degree of a Reeb chord can easily be calculated graphically from $X$.
Suppose the Reeb chord winds around a puncture $v \in \qpol_0$ then we can cut
the vector field around $v$ in wedges of the forms:

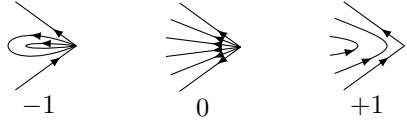
\begin{center}
\begin{tikzpicture} 
\draw[-<-] (360*2/5:1)--(0,0); 
\draw[-<-] (0,0) .. controls +(1008/5:1.3) and +(792/5:1.3) .. (0,0);
\draw[-<-] (0,0) .. controls +(936/5:.9) and +(864/5:.9) .. (0,0);
\draw[->-] (360*3/5:1)--(0,0); 
\draw (-.5,-.8) node{$-1$};
\end{tikzpicture}
\hspace{1cm}
\begin{tikzpicture} 
\draw[-<-] (360*2/5:1)--(0,0); 
\draw[-<-] (360*11/25:1)--(0,0); 
\draw[-<-] (360*12/25:1)--(0,0); 
\draw[-<-] (360*13/25:1)--(0,0); 
\draw[-<-] (360*14/25:1)--(0,0); 
\draw[-<-] (360*3/5:1)--(0,0); 
\draw (-.5,-.8) node{$0$};
\end{tikzpicture}
\hspace{1cm}
\begin{tikzpicture} 
\draw[-<-] (360*2/5:1)--(0,0); 
\draw[-<-] (1152/5-360/5:1) .. controls +(2052/5-360/5:1) and +(2268/5-360/5:1) .. (1368/5-360/5:1);
\draw[-<-] (1224/5-360/5:1) .. controls +(2124/5-360/5:.5) and +(2196/5-360/5:.5) .. (1296/5-360/5:1);
\draw[->-] (360*3/5:1)--(0,0); 
\draw (-.5,-.8) node{$+1$};
\end{tikzpicture}
\end{center}
For each wedge the Reeb chord traverses, one adds the corresponding term
to its degree.

Without loss of generality we can assume that the Hamiltonian $H:S\to \Rl$ restricted to a given Lagrangian has a unique extremal
point and therefore there is also a unique stationary Reeb chord on every Lagrangian. The fact that the Hamiltonian is quadratic near the punctures will make 
the flow circle clockwise around the punctures. Again without loss of generality we can assume that for every Lagrangrian submanifold $\cLL$ coming from an arrow in the quiver,
$\phi^1_H(\cLL)$ will spiral clockwise around the puncture and intersect the other arrows transversally. This implies that for every non-positive winding number and every pair of arrows meeting at a vertex $v$ we can find a unique Reeb chord connecting these Lagrangians with that specified winding number. By looking at the wedges around a puncture it is easy to deduce that if two Reeb chords between the same Lagrangians differ by a winding number $-w$ the difference between their degrees will be $2w (1-i_v)$ where $i_v$ is the index of the vector field $X$ at the vertex $v$.

From now on we will denote the the wrapped Fukaya category of $S\setminus \qpol_0$ with degrees given by $X$ as $\Fuk_X(S\setminus\qpol_0)$, the full subcategory containing as object the $\cLL_a$ (with $g=0$) will be denoted by $\fuk_X(\qpol)$. If we reduced the $\Z$-grading to a $\Z_2$ grading (which is independent of $X$) we will write $\Fuk(S\setminus\qpol_0)$ and $\fuk(\qpol)$.

\begin{lemma}
Given an embedded quiver $\qpol$ there is a natural one to one correspondence
between Reeb chords from $\cLL_a$ to $\cLL_b$  with $a,b \in \qpol_1$ and nonzero paths from $v_a$ to $v_b$ in  $\cQA(\rect \qpol)$.
\end{lemma}
\begin{proof}
This statement follows from the discussion above and lemma \ref{basispaths}. 
\end{proof}

The statement above implies that each Reeb chord $\gamma$ can be seen as a path in the quiver $\rect \qpol$ and we denote by $[\gamma]$ be its homotopy class. Composition of homotopy classes induces a product which we denote by $\star$ (see also appendix \ref{appendixabouzaid}). 

\begin{lemma}\label{degdeg}
For an embedded quiver $\qpol$ and a vectorfield $X$ with zeros on the vertices we have that if $[\gamma_1]\star \dots \star [\gamma_k] = [\tau]$ then
\[
\deg_X \tau \ge 2 - k + \deg_X \gamma_1 + \dots + \deg_X \gamma_k
\]
Equality holds if and only if
\[
\genmu(\gamma_1,\dots,\gamma_k)= \pm \tau \text{ in $\cQA(\rect \qpol)$}
\] 
\end{lemma} 
\begin{proof}
If two Reeb chords $\gamma_i, \gamma_{i+1}$ wind around the same puncture
then we can find a $\gamma$ with $[\gamma] =[\gamma_i]\star[\gamma_{i+1}]$.
The degree of this new $\gamma$ is the sum of the degrees of the 2 old Reeb chords 
because it traverses the same wedges. If we can prove the statement for the sequence with the $\gamma_{i},\gamma_{i+1}$ replaced by the new $\gamma$ 
it also holds for the original sequence (but the equality becomes strict because the new sequence has one term less but the degrees remain the same).

Assume that consecutive $\gamma's$ wind around different punctures.
If $\tau$ and $\gamma_1$ wind around the same puncture we can shorten them until one of them becomes trivial. If $\gamma_1$ is trivial
we can concatenate $\gamma_1$ with $\gamma_2$. A similar argument holds for $\gamma_k$, so we can assume that either $\tau$ is trivial
or it winds around a different puncture than $\gamma_1$ and $\gamma_k$.
The latter is impossible because $\tau$ winds in an anticlockwise direction
around this puncture and $[\gamma_1]\star\dots \star [\gamma_k]$ in a clockwise direction, so they cannot be homotopic. 

Now assume that $\tau$ is trivial and lift the picture to the universal cover of the punctured surface, which is contractible. The lifts of the $\gamma_i$ are the Reeb chords that go around the internal
angles of a big polygon which is obtained by glueing together lifts of
polygons of the split surface in a treelike way.

The sum of the degrees of the $\gamma_i$ is the sum of all the wedges in
the corners of this polygon. So if we can prove that this is equal to $k-2$
we are done. It is easy to do this by induction on the number of
homotopy classes of integral curves of the vector field inside the polygon. None of these are closed because they cannot leave the polygon and the vectorfield is nonzero inside the polygon.
If there is only one homotopy type for the integral curves inside the polygon, there is either one negative wedge and $k-1$ positive (if the homotopy class is a loop starting in a puncture)
or two zero wedges (at the source and target puncture) and $k-2$ positive.

If there are more homotopy types we cut the polygon along a curve 
where the homotopy type changes. We get $2$ new polygons with $k_1$ and $k_2$ sides and $k=k_1+k_2-2$. The wedges get distributed over the two polygons and by induction we get that the sum of all wedges is $2-k_1+2-k_2=2-k$.

To prove the second part, first observe that the reduction move is homogeneous for the grading $\deg_X$ because every $k$-gon has degree $2-k$. This establishes the if-part. 
To prove the only if part we will use induction on $k$. If $k=2$ then $\genmu$ and $\star$ coincide. For higher $k$ the discussion above shows that equality can only hold if all consecutive $\gamma's$ wind around different punctures. 
Now look closely at $[\gamma_1]\star \dots \star [\gamma_k]$. Because the universal cover of consists of polygons glued together in a treelike way,
there is a subpath $\gamma_i\star \dots \star \gamma_j$ such that $\gamma_j=\beta_l\star\gamma'_j$ enters a certain polygon $\beta_1, \gamma_{i+1},\dots,\gamma_{j-1},\beta_l$ are Reeb chords
that lie inside the polygon and $\gamma_i=\gamma'_i\beta_1$ leaves the polygon.
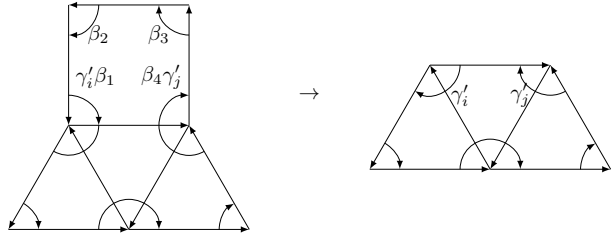
\begin{center}
\resizebox{!}{3cm}{
\begin{tikzpicture}
\draw [-latex] (7,0)--(7,-2);
\draw [-latex] (7,-2)--(9,-2);
\draw [-latex] (9,-2)--(9,0);
\draw [-latex] (9,0)--(7,0);
\draw [-latex] (7,-2)--(6,-3.73);
\draw [-latex] (6,-3.73)--(8,-3.73);
\draw [-latex] (8,-3.73)--(10,-3.73);
\draw [-latex] (10,-3.73)--(9,-2);
\draw [-latex] (9,-2)--(8,-3.73);
\draw [-latex] (8,-3.73)--(7,-2);
\draw [latex-] (7,-0.5) arc (270:360:.5);
\draw [latex-] (7.5,-2) arc (0:90:.5);
\draw [latex-] (9,-1.5) arc (90:180:.5);
\draw [latex-] (8.5,0) arc (180:270:.5);
\draw [latex-] (6.5,-3.73) arc (0:60:.5);
\draw [latex-] (8.5,-3.73) arc (0:180:.5);
\draw [latex-] (9.75,-3.30) arc (120:180:.5);
\draw  (8.5,-2) arc (180:300:.5);
\draw  (6.75,-2.43) arc (240:360:.5);
\draw (7.45,-1.2) node{$\gamma'_{i}\beta_1$};
\draw (7.5,-0.5) node{$\beta_2$};
\draw (8.5,-0.5) node{$\beta_3$};
\draw (8.55,-1.2) node{$\beta_4\gamma'_{j}$};
\draw (11,-1.5) node{$\to$};
\begin{scope}[xshift=6cm,yshift=1cm]
\draw [-latex] (7,-2)--(9,-2);
\draw [-latex] (7,-2)--(6,-3.73);
\draw [-latex] (6,-3.73)--(8,-3.73);
\draw [-latex] (8,-3.73)--(10,-3.73);
\draw [-latex] (10,-3.73)--(9,-2);
\draw [-latex] (9,-2)--(8,-3.73);
\draw [-latex] (8,-3.73)--(7,-2);
\draw [latex-] (6.5,-3.73) arc (0:60:.5);
\draw [latex-] (8.5,-3.73) arc (0:180:.5);
\draw [latex-] (9.75,-3.30) arc (120:180:.5);
\draw [latex-] (8.5,-2) arc (180:300:.5);
\draw [latex-] (6.75,-2.43) arc (240:360:.5);
\draw (7.5,-2.5) node{$\gamma'_{i}$};
\draw (8.5,-2.5) node{$\gamma'_{j}$};
\end{scope}
\end{tikzpicture}}
\end{center}
This shows that we can apply the reduction move to $\genmu$. After the reduction move the new sequence of $\gamma's$ will have degree $l-2$ less and the new $k$ is also $l-2$ less, so the equality holds for the new sequence if and only if it holds for the old
sequence. 
\end{proof}
\begin{remark}\label{Xgrading}
This lemma shows that $\deg_X$ turns $\cQA(\rect \qpol)$ into a $\Z$-graded $\cA_{\infty}$-algebra. This grading is a refinement of the $\Z_2$-grading 
because a wedge with odd parity changes the orientation of the bounding integral curves. To make the difference with the $\Z_2$-graded version
we will denote it by $\cQA_X(\rect \qpol)$.
\end{remark}

To define the $\A_\infty$-structure on the Hom-spaces we need an extra notion.
A \emph{ribbon tree map} is map $u:\mathbb{T}_n \to S\setminus\qpol_0$ which maps the ribbon tree $\mathbb{T}_n$ 
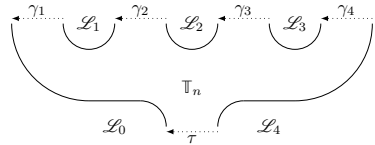
\begin{center}
\resizebox{!}{2cm}{
\begin{tikzpicture}
\draw (0,0) arc (180:360:.5);
\draw (2,0) arc (180:360:.5);
\draw (4,0) arc (180:360:.5);
\draw (-1,0) arc (180:270:1.5);
\draw (2,-2) arc (0:90:.5);
\draw[dotted,-latex] (0,0.1)--(-1,0.1);
\draw[dotted,-latex] (2,0.1)--(1,0.1);
\draw[dotted,-latex] (4,0.1)--(3,0.1);
\draw[dotted,-latex] (6,0.1)--(5,0.1);
\draw[dotted,-latex] (3,-2.1)--(2,-2.1);
\draw (3.5,-1.5) arc (90:180:.5);
\draw (4.5,-1.5) arc (270:360:1.5);
\draw (4.5,-1.5) -- (3.5,-1.5);
\draw (0.5,-1.5) -- (1.5,-1.5);
\draw (0.5,0) node{$\cLL_1$};
\draw (2.5,0) node{$\cLL_2$};
\draw (4.5,0) node{$\cLL_3$};
\draw (1,-2) node{$\cLL_0$};
\draw (4,-2) node{$\cLL_4$};
\draw (-.5,0.25) node{$\gamma_1$};
\draw (1.5,0.25) node{$\gamma_2$};
\draw (3.5,0.25) node{$\gamma_3$};
\draw (5.5,0.25) node{$\gamma_4$};
\draw (2.5,-2.25) node{$\tau$};
\draw (2.5,-1.25) node{$\mathbb{T}_n$};
\end{tikzpicture}}
\end{center}
to the punctured surface such
that the boundaries lie on Lagrangian submanifolds and the limits of the strips at infinity become Reeb chords (with an appropriate scaling for $\tau$). Furthermore we also want this map to satisfy a perturbed Cauchy-Riemann equation \cite{Abou2}. 

Every ribbon tree map also gives us an immersed convex polygon, with edges
lying on the flowed Lagrangians $\phi_H^{i}(\cLL_i)$ and corners
flowed endpoints of the Reeb chords $\phi_H^{i}\gamma_i(1)$. As is explained in \cite{Auroux}, in the case
of surfaces each such immersed convex polygon gives rise to a ribbon tree map.

\begin{lemma}\label{ribbontreemaps}
If $\qpol$ is an embedded quiver then there is a one to one correspondence between ribbon tree maps and sequences of paths $\gamma_1,\dots,\gamma_k,\tau$  in $\cQA(\rect \qpol)$
such that 
\[
\genmu(\gamma_1,\dots, \gamma_k)=\pm \tau
\]
\end{lemma}
\begin{proof}
If $\gamma_1,\dots,\gamma_k$ and $\tau$ are connected by a ribbon tree map
then clearly $[\gamma_1]\star \dots\star [\gamma_k]$ is homotopic to $[\tau]$ because the ribbon tree is contractible.
We can also pull back the vector field $X$ to the ribbon tree to show that
\[
\deg_X \tau = 2 - k + \deg_X \gamma_1 + \dots + \deg_X \gamma_k.
\]
By lemma \ref{degdeg} this implies that any ribbon map corresponds to 
a product 
\[
\genmu(\gamma_1,\dots, \gamma_k)=\pm \tau
\]
Because the universal cover of the punctured surface is contractible
there can be at most one ribbon tree map for each sequence, so correspondence is injective.
To show that it is also surjective we can induction on $k$.  If $k=2$ we can use the same argument as in lemma 4.4 in \cite{Auroux} to show that in each case $\genmu(\gamma_1,\gamma_2)=\tau$ there is also corresponding ribbon tree map.

If $k>2$ then lemma \ref{degdeg} shows that there is a polygon and a sequence of Reeb chords $\beta_1,\gamma_{i+1},\dots,\gamma_{j-1},\beta_l$ that goes around it. Using induction we can assume that there is a ribbon tree map for the Reeb chords with this sequence cut out. To this ribbon tree map we can glue the extra polygon. To show that this gives a new ribbon tree map we have to show that the corresponding polygon that bounds the flowed Lagrangians is convex.

At any puncture the angle between a Lagrangian and a flowed Lagrangian that is open towards the polygon's interior is convex. Therefore the polygon bounded by flowed Lagrangians is convex 
if all Reeb chords $\gamma_i$ wind around different punctures or in other words $\gamma_i\gamma_{i+1}=0$ in $\cQA(\rect \qpol)$. This is the case if $\mu(\gamma_1,\dots,\gamma_k)\ne 0$ by lemma \ref{rules} (3).  
\end{proof}

To describe the $\cA_\infty$ structure we will use proposition \ref{propabouzaid} in the appendix by Abouzaid. This result in combination with \cite{Seidelbook} tells us that in the case when all flowed Lagrangians intersect transversally and lemma \ref{degdeg} holds, the $\cA_\infty$ structure is given by a signed count of
ribbon tree maps (or equivalently immersions of polygons between the flowed Lagrangians)
\[
\mu(\gamma_1\dots,\gamma_k)= \sum_{\TT_k\to \gamma_1,\dots,\gamma_k,\tau} \pm \tau.
\]
Note that in our case there can be at most one term in this sum because of lemma \ref{ribbontreemaps}. 

\begin{remark}
In $\Fuk_X(S\setminus \qpol_0)$ there can also be hom-spaces between Lagrangians that do not intersect transversally or for which \ref{propabouzaid} does not apply. 
In that case the definitions of the hom-spaces and the products become more involved but for our Lagrangians $\cLL_a$ this does not occur 
because all their flowed versions intersect transversally.
\end{remark}

\begin{theorem}
Consider a quiver $\qpol$ that splits a compact surface $S=|\qpol|$ and put a Liouville structure, a Hamiltonian and a vector field $X$ as described above.
\begin{enumerate}
\item
The full subcategory $\fuk_X(\qpol)$ of the wrapped Fukaya category $\Fuk_X(S\setminus \qpol_0)$
is isomorphic to $\cQA_{X}(\rect \qpol)$ as an $\cA_\infty$-category.
\item
If $\Fuk(S\setminus \qpol_0)$ is the $\Z_2$-graded version of $\Fuk_X(S)$ then $\H\Tw\Fuk(|\qpol|\setminus \qpol_0)$ is equivalent to $\H\Tw \cQA(\rect \qpol)$.
\end{enumerate}
\end{theorem}
\begin{proof}
From lemma \ref{ribbontreemaps} it is clear that $\mu$
and $\genmu$ are the same up to signs. Lemma \ref{uptoiso} now allows us to
that they are isomorphic.

The second statement is basically a consequence of the fact that $\H\Tw \cQA (\rect \qpol)$ does not depend on the quiver.
Given any noncontractible Lagrangian in $\Fuk(S\setminus \qpol_0)$ we can add Lagrangians until we get a quiver $Q$ that splits the surface.
This implies that this particular Lagrangian sits in $\H\Tw \cQA (\rect \qpol)$. So every Lagrangian with end points at the punctures sits in a
$\H\Tw \cQA (\rect \qpol)$ and $\H\Tw \cQA (\rect \qpol)\subset \H\Tw\Fuk (S\setminus \qpol_0)\subset \H\Tw \cQA (\rect \qpol)$.
\end{proof}

\section{Consistent dimer models and matrix factorizations}

\subsection{Matrix factorizations in general}
Consider either an algebra $\cA$ or a smooth algebraic variety $X$.
A potential is defined as a central element $W\in Z(A)$ in the former case or a polynomial function $X\to\C$.
The pair $(X,W)$ is called a \emph{commutative Landau-Ginzburg model}, the pair
$(\cA,W)$ is a \emph{noncommutative Landau-Ginzburg model}.

A \emph{matrix factorization} of a Landau-Ginzburg model is a diagram $\bar P$
\[
 \xymatrix{P_0 \ar@<.5ex>[r]^{p_0}&P_1 \ar@<.5ex>[l]^{p_1}}
\]
where $P_i$ are projective $\cA$-modules or vector bundles over $X$ such that
that $p_0p_1=W$ and $p_1p_0=W$.

Given $2$ matrix factorizations $\bar P$ and $\bar Q$
we define $\Hom(\bar P,\bar Q)$ as the following $\Z_2$-graded space of module morphisms/sheaf morphisms
\[
\Hom(\bar P,\bar Q)= \underbrace{\Hom_\cC(P_0,Q_0)\oplus \Hom_\cC(P_1,Q_1)}_{\text{even}}\oplus
\underbrace{\Hom_\cC(P_0,Q_1)\oplus \Hom_\cC(P_1,Q_0)}_{\text{odd}}.
\]
On this space we have a differential of odd degree:
\[
 d\begin{pmatrix}f_{00}&f_{01}\\f_{10}&f_{11}\end{pmatrix} := 
\begin{pmatrix}0&p_0\\p_1&0\end{pmatrix}\begin{pmatrix}f_{00}&f_{01}\\f_{10}&f_{11}\end{pmatrix}
 - \begin{pmatrix}f_{00}&-f_{01}\\-f_{10}&f_{11}\end{pmatrix}\begin{pmatrix}0&q_0\\q_1&0\end{pmatrix}.
\]
It is easy to check that $d^2=0$.

The category of all matrix factorizations is denoted by $\MF(\cA,W)$ or $\MF(X,W)$ and it has the 
structure of a $\Z_2$-graded dg-category or $\cA_\infty$-category. 
In many cases it is interesting to look at small subcategories
of this $\cA_\infty$-category.

\begin{remark}
If $\cA$ is a $\Z$-graded algebra and the central element has degree $2$ we can also make do a $\Z$-graded construction. For a graded matrix factorization $\bar P$ we demand that $P_0$ and $P_1$ are graded
projective modules and $p_0$ and $p_1$ are homogeneous maps of degree $1$.

For two graded matrix factorizations $\bar P,\bar Q$ we set 
\[
\Hom(\bar P,\bar Q)= \Hom_\cC(P_0,Q_0)\oplus \Hom_\cC(P_1,Q_1)\oplus
\Hom_\cC(P_0,Q_1)(1)\oplus \Hom_\cC(P_1,Q_0)(1).
\]
where $(1)$ denotes the degree shift. With this definition $d$ becomes
homogeneous of degree $1$ and $\MF^{gr}(\cA,W)$, the category of all
graded matrix factorizations, becomes a $\Z$-graded
$dg$-category.
\end{remark}

\subsection{The $\cA_\infty$-algebra for a consistent dimer model}\label{aifromdimer}
Let $\qpol$ be a zigzag consistent dimer model on a surface (with nonzero genus) and $A$ its Jacobi algebra.
The Jacobi algebra has a central element coming from the cycles
\[
 \ell := \sum_{v\in\qpol_0} c_v
\]
where $c_v\in \qpol_2$ is a cycle starting in $v$. Note that the relations in the Jacobi algebra
ensure that this expression does not depend on the choices of $v$ and is indeed central.

For each arrow $a$ we can define a matrix factorization of $\ell$
\[
\bar P_a :=  \xymatrix{Ah(a) \ar@<.5ex>[r]^{a}&At(a) \ar@<.5ex>[l]^{\bar a}}
\]
where $\bar a$ is defined such that $a\bar a\in \qpol_2$. In the weak Jacobi algebra $\bar a$ can be expressed as $\ell a^{-1}$.

We define $\mf(\qpol)$ as the dg-category with objects the $\bar P_a$ and morphisms
as in the previous section. As we already know $\mf(\qpol)$ is a dg-algebra and hence its homology
$\H\mf(\qpol)$ allows an $\cA_\infty$-structure. 

If $\ucover\qpol$ is the universal cover of $\qpol$ and $\pi$ the cover map, then we can also define matrix 
factorizations in this cover $\bar P_{\tilde a}$ for every $\tilde a\in \ucover \qpol_1$.
By lifting paths to the universal cover we can easily see that
\[
 \Hom(\bar P_a,\bar P_b) = \bigoplus_{\tilde b, \pi(\tilde b)=b}\Hom(\bar P_{\tilde a},\bar P_{\tilde b}).
\]
where we have chosen a fixed lift of $a$ and vary over the lifts of $b$.
This identification is also compatible with the differentials on both sides.

\subsection{A nice basis for $\H\mf(\ucover\qpol)$}
Recall that the \emph{zig ray} $\cZ^+_{a_0}$ (zag ray $\cZ^-_{a_0}$) of an arrow $a_0$ is the infinite sequence $\dots a_2a_1a_0$ of arrows in the universal cover, $\tilde \qpol$, 
such that $a_{i+1}a_{i}$ sits in a positive 
cycle if $i$ is even (odd) and in a negative cycle if $i$ is odd (even).
Given a piece of a zag ray $\pZ=a_u\dots a_1a_0$ with even length, we define the left (right) opposite path as 
the path formed by all arrows that are not in $\pZ$ but are in positive (negative) cycles which meet $\pZ$.
If $\pZ$ has odd length we take the left opposite path $a_{u-1}\dots a_0$ and the right opposite path of $a_u\dots a_1$.

We can put these two opposite paths in the correct entries
of a $2\times 2$-matrix and add an appropriate minus sign on the upper right entry. 
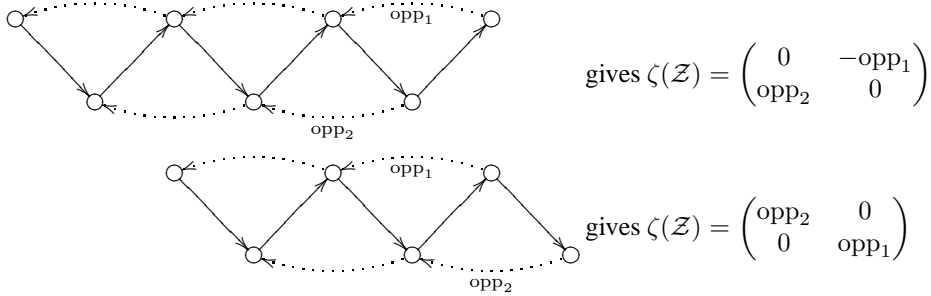
\se{ 
\vcenter{\xymatrix{
\vtx{}\ar[dr]&&\vtx{}\ar[dr]\ar@/_/@{.>}[ll]&&\vtx{}\ar[dr]\ar@/_/@{.>}[ll]&&\vtx{}\ar@/_/@{.>}[ll]^{\opp_1}&\\
&\vtx{}\ar[ur]&&\vtx{}\ar[ur]\ar@/^/@{.>}[ll]&&\vtx{}\ar[ur]\ar@/^/@{.>}[ll]^{\opp_2}&&
}}
&\text{ gives }
\RZ(\pZ)= \begin{pmatrix}
           0&-\opp_1\\
           \opp_2&0
          \end{pmatrix}\\
\vcenter{\xymatrix{
\vtx{}\ar[dr]&&\vtx{}\ar[dr]\ar@/_/@{.>}[ll]&&\vtx{}\ar[dr]\ar@/_/@{.>}[ll]^{\opp_1}&\\
&\vtx{}\ar[ur]&&\vtx{}\ar[ur]\ar@/^/@{.>}[ll]&&\vtx{}\ar@/^/@{.>}[ll]^{\opp_2}
}}
&\text{ gives }
\RZ(\pZ)= \begin{pmatrix}
           \opp_2&0\\0&\opp_1
          \end{pmatrix}
}
In this way we
obtain an element in $\RZ(\pZ) \in \Hom(P_{a_u},P_{a_0})$ which is an off-diagonal or diagonal matrix,
depending on whether the length of $\cZ$ is even or odd.

\begin{lemma}
If $\qpol$ is a consistent dimer model on a surface with nonpositive Euler charachteristic
and $\cZ$ is a zigzag path then both opposite paths are minimal (i.e. not a multiple of $\ell$).
\end{lemma}
\begin{proof}
From \cite{davison} we know there is a unique minimal path between every two vertices in the universal cover and
such that every other path between these vertices is a power of $\ell$ times this path.

Let $p$ be the minimal path in the universal cover that connects $h(\cZ)$ to $t(\cZ)$ and let $a$ be the last arrow of $\cZ$.
First suppose that $p$ does not cross $\cZ$ and assume
we cannot apply any reductions to $p$ that make the piece that $p$ and $\cZ$ cut out smaller.
We will prove the lemma by induction on the number of faces inside this piece.

If this number is $1$ we are done because then the path is precisely the opposite path.
If the number is bigger than $1$, then the other zigzag path $\cZ'$ that goes through the penultimate arrow of $\cZ$ 
must cut the piece in two. Let $b$ be the arrow where $\cZ'$ intersects $p$ and denote by $p'$ the path which consists of 
the piece of $p$ before $b$. If $p$ is minimal then so must $p'a$ be and by induction it must be the path that goes opposite
to $\cZ'$.

The arrow $b$ makes that $bp'$ and a forteriori $p$ can be reduced.
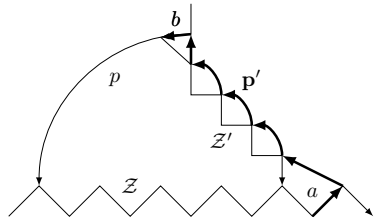
\begin{center}
\resizebox{5cm}{!}{
\begin{tikzpicture}
\draw [-latex] (-.5,-.5) -- (0,0) -- (.5,-.5) -- (1,0) -- (1.5,-.5) -- (2,0) -- (2.5,-.5) -- (3,0)-- (3.5,-.5) -- (4,0)-- (4.5,-.5) -- (5,0)-- (5.5,-.5);
\draw [latex-] (0,0) arc (180:90:2.5);

\draw [-latex] (2,2.45) -- (2.5,2) -- (2.5,1.5) -- (3,1.5) -- (3,1) -- (3.5,1) -- (3.5,.5)-- (4,.5) -- (4,0);
\draw [latex-] (2,2.45) -- (2.5,2.5) -- (2.5,3);
\draw [latex-, very thick] (2,2.45) -- (2.5,2.5);
\draw [-latex,very thick] (4.5,-.5) -- (5,0);
\draw [-latex,very thick] (5,0) -- (4,.5);
\draw [-latex,very thick] (4,.5) arc (0:90:.5);
\draw [-latex,very thick] (3.5,1) arc (0:90:.5);
\draw [-latex,very thick] (3,1.5) arc (0:90:.5);
\draw [-latex,very thick] (2.5,2) -- (2.5,2.5);
\draw (4.5,-.1) node{$a$};
\draw (3.0,.75) node{$\pZ'$};
\draw (1.5,0) node{$\pZ$};
\draw (2.25,2.75) node{$b$};
\draw (2.25,2.75) node{$b$};
\draw (1.25,1.75) node{$p$};
\draw (3.5,1.75) node{$\mathbf{p'}$};
\end{tikzpicture}
}
\end{center}
If $p$ does cross $\cZ$ in an arrow or a vertex, we can 
apply the above argument for each of the pieces separately and get the same result.
\end{proof}

\begin{lemma}
The $\RZ(\pZ)$ for which $\pZ$ starts in $a$ and ends in $b$ form a basis for $\H\Hom(\bar P_b,\bar P_a)$.
\end{lemma}
\begin{proof}
Because of $\Z_2$-degree and path-degree reasons we can look for basis elements of the forms
$$\left(\begin{smallmatrix}f_{00}&0\\0&f_{11}\end{smallmatrix}\right),~~\left(\begin{smallmatrix}0&f_{01}\\f_{10}&0\end{smallmatrix}\right)$$ 
where the $f_{ij}$ are paths up to a sign.
We will only treat the first case.
Being in the kernel of $d$ implies that $f_{11}=\del a(f_{00}b)=(\bar af_{00})\del{\bar b}\in \cJA(\qpol)$.
These two conditions are the same as asking that $\bar a f_{00}b$ is a multiple of $\ell$.
Being in the image of $d$ implies that $f_{00}$ is a left multiple of $a$ or a right multiple of $\bar b$ (which is equivalent to
$f_{11}$ being a right multiple of $b$ or a left multiple of $\bar a$). 

So we are looking for paths $f_{00}$ such that $\bar a f_{00}b$ is a multiple of $\ell$ but
neither $\bar a f_{00}$ or  $f_{00}b$ is a multiple of $\ell$. 
In particular this means that $f_{00}$ and $f_{11}$ are both minimal paths.

Now look at the paths $f_{00}$ and $f_{11}$ in the universal cover. If these $2$ paths intersect at a vertex $v$ we can
split $f_{ii}$ in two $f_{iv}vf_{vi}$. Now $\bar a f_{0v}$ and $f_{1v}$ are both minimal and run between
the same vertices so they must be the same. But then $f_{11}=a\bar a f_{0v}f_{v1}$ is a left multiple of $\bar a$ which is impossible.
So $a$, $b$, $f_{00}$ and $f_{11}$ bound a simply connected piece $S$ in the universal cover. After applying the relations
we can assume that this piece is as small as possible.

Look at the zigzag path $\pZ$ 
that starts from $a$ and enters this simply connected piece. This zigzag path must leave the piece at an arrow $c$.
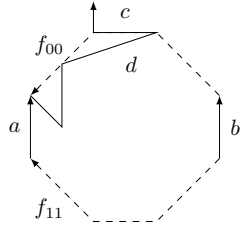
\begin{center}
\resizebox{!}{3cm}{
\begin{tikzpicture}
\draw[-latex] (0,0)--(0,1);
\draw[-latex] (3,0)--(3,1);
\draw[-latex,dashed] plot[tension=.5] coordinates{(3,0) (2,-1) (1,-1) (0,0)};
\draw[-latex,dashed] plot[tension=.5] coordinates{(3,1) (2,2) (1,2) (0,1)};
\draw[latex-] (1,2.5) -- (1,2) -- (2,2) -- (.5,1.5) -- (.5,.5) -- (0,1);
\draw (-.25,.5) node{$a$};
\draw (3.25,.5) node{$b$};
\draw (1.5,2.3) node{$c$};
\draw (1.6,1.5) node{$d$};
\draw (.3,1.8) node{$f_{00}$};
\draw (.3,-.8) node{$f_{11}$};
\end{tikzpicture}}
\end{center}
If $c$ is in $f_{00}$ then we can split $f_{00}=f_{0c}cf_{c0}$. Because of minimality $f_{0c}$ and the opposite path running 
along the zigzag path from $h(a)$ to $h(c)$ must be equal. But then $f_{0c}c$ ends in  $\bar d$ where $d$ precedes
$c$ in the zigzag path. This contradicts the fact that we chose $S$ as small as possible.
Similarly $\pZ$ cannot leave $S$ through an arrow of $f_{11}$, so $b$ must lie on the zigzag path through $a$.
In order for $\pZ$ to leave $S$ at $b$, we must have that $\pZ[i]=b$ with $b$ even. 
\end{proof}
\begin{corollary}
In the universal cover $\ucover \qpol$ the space $\H\Hom(\bar P_{\tilde b},\bar P_{\tilde a})$ is either
zero dimensional or one dimensional, depending on whether $\tilde b$ sits in a zigzag ray from $\tilde a$ or not.
\end{corollary}

Now we are going to express the $A_\infty$-products in terms of these bases.
To do this we will expand this basis to a full basis of $\Hom(\bar P_b,\bar P_a)$. 

Consider the set of matrices
\begin{itemize}
 \item $
 \begin{pmatrix}
  f_{00}&0\\
0&0
 \end{pmatrix}$,
 $\begin{pmatrix}
  0&f_{01}\\
0&0
 \end{pmatrix}$ where $f_{00}$ and $f_{01}$ are paths.
 \item $
 \begin{pmatrix}
  0&0\\
0&f_{11}
 \end{pmatrix}$,
 $\begin{pmatrix}
  0&0\\
f_{10}&0
 \end{pmatrix}$ where $f_{11}$ and $f_{10}$ are paths and
$af_{11}\del{b},af_{10}\del{\bar b}\cancel \in \cJA(\qpol)$
\end{itemize}
and let $U$ be the span of these matrices.

By construction it is clear that the matrices above are linearly independent and
\[
 \Hom(P_a,P_b) = U \oplus \H \Hom(P_a,P_b) \oplus d(U). 
\]
This splitting allows us to define a codifferential
\[
 h:\Hom(P_a,P_b)\to \Hom(P_a,P_b):u_1\oplus \RZ(\pZ_a^{\pm,j})\oplus du_2 \mapsto u_2\oplus 0 \oplus 0.
\]

\begin{lemma}
\begin{itemize}
 \item[]
 \item $dhd=d$,
 \item $h^2=0$,
 \item $h(\RZ(\pZ))=0$.
 \item $h$ respects the split $\Hom(\bar P_a,\bar P_b) = \bigoplus_{\tilde b, \pi(\tilde b)=b}\Hom(\bar P_{\tilde a},\bar P_{\tilde b})$. 
 \item For any cover $d$ and $h$ are equivariant with respect to the group action of the deck transformations.  
\end{itemize}
\end{lemma}
\begin{proof}
These facts follow easily from the construction. The last statement is in fact equivalent to the penultimate.
\end{proof}

We are now going to determine the $A_{\infty}$-structure on $\H\mf(\qpol)$.
First we start with the ordinary multiplication.
\begin{theorem}
Suppose $\pZ_1$ is a zigzag path from $b$ to $a$ and $\pZ_2$ is a zigzag path from $a$ to $c$
\[
 \mu(\RZ(\pZ_1),\RZ(\pZ_2))=
\begin{cases}
\RZ(\pZ_2a^{-1}\pZ_1)&\text{if $\pZ_2a^{-1}\pZ_1$ is a zigzag path} \\
0&\text{otherwise}
\end{cases}
\]
\end{theorem}
\begin{proof}
From the formula for the $\cA_{\infty}$-structure we get
\[ 
\mu(\RZ(\pZ_1),\RZ(\pZ_2))= (1-hd+dh)(\RZ(\pZ_1)\RZ(\pZ_2)).
\]
Now $\RZ(\pZ_1)\RZ(\pZ_2)$ is a matrix that either consists
of $2$ paths (possibly with signs) on the off-diagonals
or on the diagonals.
The expression $(1-hd+dh)(\RZ(\pZ_1)\RZ(\pZ_2))$ 
will be nonzero only if these paths are opposite paths
of some other zigzag path.

If $\pZ_2a^{-1}\pZ_1$ is a zigzag path, one can easily check that
$\RZ(\pZ_1)\RZ(\pZ_2)=\RZ(\pZ_2a^{-1}\pZ_1)$ and hence
$\mu(\RZ(\pZ_1),\RZ(\pZ_2))= \RZ(\pZ_2a^{-1}\pZ_1)$.

If $\pZ_1$ and $\pZ_2$ are zigzag paths in different directions,
we will show that $\H\Hom(\bar P_b,\bar P_c)=0$ in the universal cover.
If this were not the case then there must be a zigzag path $\pZ$ from $c$ to $b$. 

If $\pZ_2$ lies to the left of $\pZ_1$ then in order for $\pZ$ to intersect with $\pZ_2$,
$\pZ$ must also lie to the left of $\pZ_1$ and therefore the length of $\pZ_1$ must be odd.
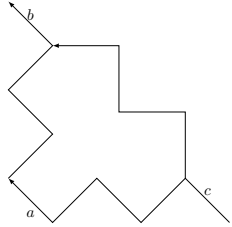
\begin{center}
\resizebox{3cm}{!}{
\begin{tikzpicture}
\draw[-latex] (0,0) -- (-1,1) -- (-2,0) -- (-3,1) -- (-4,0) --(-5,1);
\draw[-latex] (-5,1) -- (-4,2) -- (-5,3) -- (-4,4) -- (-5,5);
\draw[-latex] (-1,1) -- (-1,2.5) -- (-2.5,2.5) -- (-2.5,4) -- (-4,4);
\draw (-.5,.7) node{$c$};
\draw (-4.5,.2) node{$a$};
\draw (-4.5,4.7) node{$b$};
\end{tikzpicture}}
\end{center}
For similar reasons the length of $\pZ_2$ must be odd while the length of $\pZ$ must be even.
This is impossible because of the degree of $\mu(\RZ(\pZ_1),\RZ(\pZ_2))$ would then be even while
the degree of $\RZ(\pZ)$ is odd.
\end{proof}

\begin{theorem}\label{productsare1}
Let $a_1 \dots a_k$ be a positive or negative cycle in $\qpol$ then 
\[
 \mu(\RZ(a_1a_2),\RZ(a_2a_3),\dots,\RZ(a_la_{l+1}))=\begin{cases}
                                                            0& l\ne k\\
 \RZ(a_1) &l=k
                                                           \end{cases}
\]
\end{theorem}
\begin{proof}
First note that
\[
 \RZ(a_ia_{i+1})=\begin{pmatrix}
                  0&-t(a_i)\\
a_{i+2}\dots a_{i+k-1}& 0
                 \end{pmatrix}.
\]
Therefore
\[
  h(\RZ(a_ia_{i+1})\RZ(a_{i+1}a_{i+2}))=
\begin{pmatrix}
                  0&0\\
-a_{i+3}\dots a_{i+k-1}&0
                 \end{pmatrix}
\]
and by induction
\se{
 &h(\dots h(h(\RZ(a_ia_{i+1})\RZ(a_{i+1}a_{i+2})\RZ(a_{i+2}a_{i+3})\dots)\RZ(a_{i+r}a_{i+r+1})\\&=
\begin{pmatrix}
                  0&0\\
-a_{i+r+2}\dots a_{i+k-1}&0
                 \end{pmatrix}.
}
if $r+2\le k-1$ and zero otherwise.

If we multiply two such expressions together we get zero. If we multiply one such expression on the
left with $\RZ(a_{i-1}a_{i})$ we get a matrix with only the upper left diagonal nonzero.
If we apply $h$ to this we also get zero.

The only trees that might give a nonzero contribution to $\mu$ are
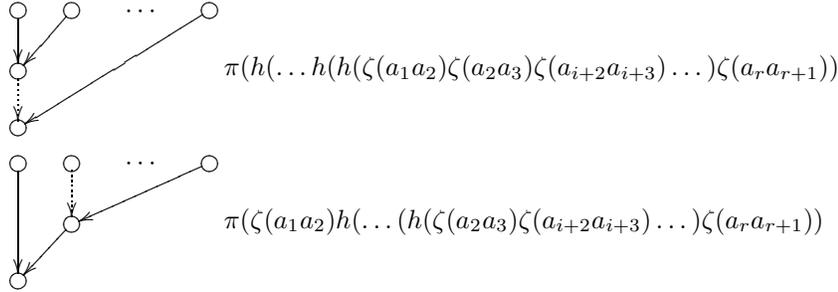
\se{
\vcenter{\xymatrix@C=.5cm@R=.5cm{\vtx{}\ar[d]&\vtx{}\ar[ld]&\cdots&\vtx{}\ar[llldd]\\
\vtx{}\ar@{.>}[d]&&&\\
\vtx{}&&&
}}~~
 &\pi(h(\dots h(h(\RZ(a_1a_{2})\RZ(a_{2}a_{3})\RZ(a_{i+2}a_{i+3})\dots)\RZ(a_{r}a_{r+1}))\\
\vcenter{\xymatrix@C=.5cm@R=.5cm{\vtx{}\ar[dd]&\vtx{}\ar@{.>}[d]&\cdots&\vtx{}\ar[lld]\\
&\vtx{}\ar[ld]&&\\
\vtx{}&&&
}}~~
 &\pi(\RZ(a_1a_{2}) h(\dots (h(\RZ(a_{2}a_{3})\RZ(a_{i+2}a_{i+3})\dots)\RZ(a_{r}a_{r+1}))
}
Both expressions can only give a nonzero result if $r=k$. Indeed by degree reasons these expressions are even: 
$r\deg \RZ(a_1a_2)+(r-2)\deg h +\deg \pi=0$, but the only path that goes opposite a zigzag path
of odd length which is also a subpath of $a_{1}\dots a_{i+r}$ is the trivial path and this happens when $r=0 \mod k$.
If $r>k+1$ the expression $h(h\dots h(\dots))$ will be zero.
 
If $r=k$ then we see that
\se{
 &h(\dots h(h(\RZ(a_1a_{2})\RZ(a_{2}a_{3})\RZ(a_{i+2}a_{i+3})\dots)\RZ(a_{r}a_{r+1}))\\
+&
 \RZ(a_1a_{2}) h(\dots (h(\RZ(a_{2}a_{3})\RZ(a_{i+2}a_{i+3})\dots)\RZ(a_{r}a_{r+1}))
\\=&\begin{pmatrix}
                                                                                      h(a_1)&0\\
0&t(a_1)
                                                                                     \end{pmatrix}=\RZ(a_1).
}
Because the latter is a basis element of the homology, $\pi(\RZ(a_1))= \RZ(a_1)$.
\end{proof}

\begin{lemma}
The construction $\H\mf(\qpol)$ is compatible with covers.
\end{lemma}
\begin{proof}
Because $d,h$ and $\pi$ are equivariant for the deck transformations of any finite cover $\qpol\to \qpol'$, we have that
$\H\mf(\qpol)=\pi\mf(\qpol)$ also has a free $\Aut(\qpol\to \qpol')$-action and $\H\mf(\qpol')$ can be identified with
the $\Aut(\qpol\to \qpol')$-invariant part. The tree-construction of $\mu$ only uses $d,h$ and $\pi$ so clearly it is also equivariant.  
\end{proof}

\begin{remark}\label{gradingmf}
A \emph{perfect matching} is a subset of arrows $\cP\subset \qpol_1$ such that every positive and negative cycle contain precisely one arrow of $\cP$. This enables us to define a degree function 
$$
\deg_\cP=
\begin{cases}
0 & a \not\in \cP\\
2 & a \in \cP
\end{cases}
$$ 
that turns $\cJA(\qpol)$ in a $\Z$-graded algebra, for which every cycle 
in $\qpol_2$ has degree $2$. Note that not every dimer has a perfect matching \cite{broomhead} but on a torus zigzag consistency implies the existence of a perfect matching \cite{Bocklandtcons}.

If $\qpol$ has a perfect matching $\cP$ then $\ell$ is a homogeneous element of degree $\deg_\cP \ell = 2$. The matrix factorization $P_a$ can be turned into a graded one:
\[
\bar P_a :=  \xymatrix{Ah(a) \ar@<.5ex>[r]^{a}&At(a)(\deg_\cP a-1) \ar@<.5ex>[l]^{\bar a}}
\]
This turns $\mf(\qpol)$ into a $\Z$-graded dg-algebra which we denote by $\mf(\qpol,\cP)$. 
It is easy to check that the degree of $\RZ(ab)$ is $1-\deg_\cP(a)$, so $-1$ if $a$ is in the perfect matching and $+1$ otherwise.
\end{remark}

\section{Dimer duality as mirror symmetry}

Let $\qpol$ be any, not necessarily consistent, dimer. 
We define its mirror dimer $\mirror{\qpol}$ as follows

\begin{enumerate}
 \item The vertices of $\mirror{\qpol}$ are the zigzag cycles of $\qpol$.
 \item The arrows of $\mirror{\qpol}$ are the arrows of $\qpol$, $h(a)$ is the zigzag cycle coming from the zig ray, and $t(a)$
is the cycle coming from the zag ray.
 \item The positive faces of $\mirror{\qpol}$ are the positive faces of $\qpol$.
 \item The negative faces of $\mirror{\qpol}$ are the negative faces of $\qpol$ in reverse order. 
\end{enumerate}
We illustrate this with a couple of examples:
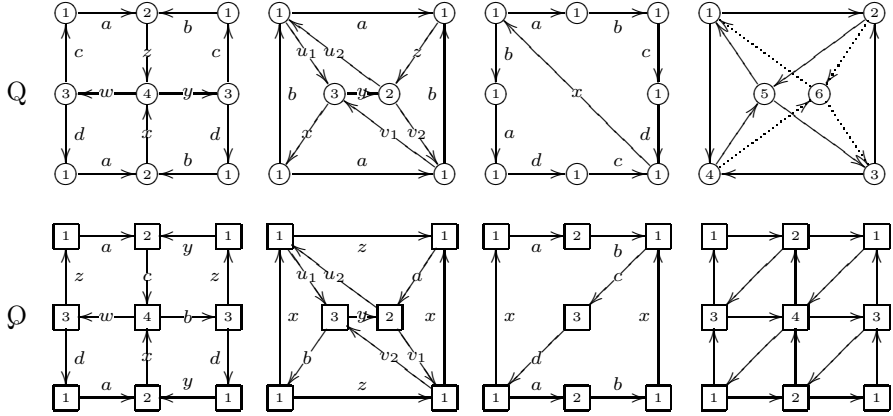
\begin{center}
\begin{tabular}{ccccc}
$\qpol$&$\vcenter{
\xymatrix@C=.75cm@R=.75cm{
\vtx{1}\ar[r]_{a}&\vtx{2}\ar[d]|z&\vtx{1}\ar[l]^b\\
\vtx{3}\ar[u]_c\ar[d]^d&\vtx{4}\ar[l]|w\ar[r]|y&\vtx{3}\ar[u]^c\ar[d]_d\\
\vtx{1}\ar[r]^{a}&\vtx{2}\ar[u]|x&\vtx{1}\ar[l]_b
}}$
&
$\vcenter{
\xymatrix@C=.4cm@R=.75cm{
\vtx{1}\ar[rrr]_a\ar[dr]|{u_1}&&&\vtx{1}\ar[ld]|z\\
&\vtx{3}\ar[r]|y\ar[ld]|x&\vtx{2}\ar[ull]|{u_2}\ar[dr]|{v_2}&\\
\vtx{1}\ar[rrr]^a\ar[uu]_b&&&\vtx{1}\ar[ull]|{v_1}\ar[uu]^b
}}$
&
$\vcenter{
\xymatrix@C=.75cm@R=.75cm{
\vtx{1}\ar[r]_{a}\ar[d]^{b}&\vtx{1}\ar[r]_b&\vtx{1}\ar[d]_c\\
\vtx{1}\ar[d]^a&&\vtx{1}\ar[d]_d\\	
\vtx{1}\ar[r]^{d}&\vtx{1}\ar[r]^c&\vtx{1}\ar[uull]|x
}}$
&
$\vcenter{
\xymatrix@C=.4cm@R=.75cm{
\vtx{1}\ar[dd]\ar[rrr]&&&\vtx{2}\ar[dll]\ar@{.>}[dl]\\
&\vtx{5}\ar[ul]\ar[drr]&\vtx{6}\ar@{.>}[ull]\ar@{.>}[dr]&\\
\vtx{4}\ar[ur]\ar@{.>}[urr]&&&\vtx{3}\ar[uu]\ar[lll]
}}$
\vspace{.5cm}
\\
$\mirror{\qpol}$&
$\vcenter{
\xymatrix@C=.75cm@R=.75cm{
\sqvtx{1}\ar[r]_{a}&\sqvtx{2}\ar[d]|c&\sqvtx{1}\ar[l]^y\\
\sqvtx{3}\ar[u]_z\ar[d]^d&\sqvtx{4}\ar[l]|w\ar[r]|b&\sqvtx{3}\ar[u]^z\ar[d]_d\\
\sqvtx{1}\ar[r]^{a}&\sqvtx{2}\ar[u]|x&\sqvtx{1}\ar[l]_y
}}$
&
$\vcenter{
\xymatrix@C=.4cm@R=.75cm{
\sqvtx{1}\ar[rrr]_z\ar[dr]|{u_1}&&&\sqvtx{1}\ar[ld]|a\\
&\sqvtx{3}\ar[r]|y\ar[ld]|b&\sqvtx{2}\ar[ull]|{u_2}\ar[dr]|{v_1}&\\
\sqvtx{1}\ar[rrr]^z\ar[uu]_x&&&\sqvtx{1}\ar[ull]|{v_2}\ar[uu]^x
}}$
&
$\vcenter{
\xymatrix@C=.75cm@R=.75cm{
\sqvtx{1}\ar[r]_{a}&\sqvtx{2}\ar[r]_b&\sqvtx{1}\ar[ld]|c\\
&\sqvtx{3}\ar[ld]|d&\\	
\sqvtx{1}\ar[uu]_x\ar[r]^{a}&\sqvtx{2}\ar[r]^b&\sqvtx{1}\ar[uu]^x
}}$
&
$\vcenter{
\xymatrix@C=.75cm@R=.75cm{
\sqvtx{1}\ar[r]&\sqvtx{2}\ar[r]\ar[ld]&\sqvtx{1}\ar[ld]\\
\sqvtx{3}\ar[r]\ar[u]&\sqvtx{4}\ar[r]\ar[u]\ar[ld]&\sqvtx{3}\ar[u]\ar[ld]\\
\sqvtx{1}\ar[r]\ar[u]&\sqvtx{2}\ar[r]\ar[u]&\sqvtx{1}\ar[u]
}}$
\end{tabular}
\end{center}
\begin{remark}
The dual can also be obtained by cutting the dimer along the arrows, flipping over the clockwise faces, reversing their arrows and gluing everything back again.
This construction is basically the same construction that was introduced by Feng, He, Kennaway and Vafa in 
\cite{quivering} applied to all possible dimers.
\end{remark}

\begin{lemma}
$\qpol \mapsto \mirror{\qpol}$ is an involution on the set of all dimer models. 
\end{lemma}
\begin{proof}
Let $\pZ=a_1 \dots a_k$ be a zigzag cycle in $\mirror{\qpol}$. 
If $i$ is odd then $a_ia_{i+1}$ sits in a positive cycle in $\mirror{\qpol}$ while 
$a_ia_{i+1}$ sits in a negative cycle of $\mirror{\qpol}$ when $i$ is even.

This implies that $a_ia_{i+1}$ sits in a positive cycle in $\qpol$ for $i$ odd while
$a_{i+1}a_i$ sits in a negative cycle of $\qpol$ when $i$ is even.
Hence in $\qpol$ the odd arrows of $\pZ$ must have the same head in $\qpol$ which is equal
to the tail of the even arrows. This means there is a well defined map from the zigzag paths in $\mirror{\qpol}$
and the vertices of $\qpol$. Because $a$ is an even arrow in its zig cycle and an odd in its zag cycle, under this map
the zig cycle and the zag cycle of $a$ in $\mirror{\qpol}$ correspond to the original head and tail.
Finally this map is a bijection because given any vertex $v \in \qpol$ we can make a zigzag cycle in $\mirror{\qpol}$ by
listing all arrows incident with $v$ in a clockwise direction.
\end{proof}

\begin{theorem}
If $\qpol$ is a consistent dimer with a perfect matching $\cP$ then 
 we can find a vector field $X$ on 
$|\!\!\mirror{\qpol}\!\!|$ such that 
$\H\mf(\qpol,\cP)$ is $\cA_\infty$-isomorphic to ${\cQA_X(\rect\!\!\mirror{\qpol}})$. 
\end{theorem} 
\begin{proof}
The underlying category of $\H\mf(\qpol)$ is indeed isomorphic to $\cQA(\rect\!\!\mirror{\qpol})$.
The paths in $\cQA(\rect\!\!\mirror{\qpol})$ are those that cycle around vertices of $\mirror{\qpol}$, which are precisely
the zigzag paths of $\qpol$. Two paths multiply to zero in $\cQA(\rect\!\!\mirror{\qpol})$ if they go around different
vertices in $\mirror{\qpol}$ just as the product of $\RZ(p_1)\cdot \RZ(p_2)$ is zero if they belong to different zigzag cycles.

To get an isomorphism of graded vector spaces we a vector
field on $|\!\!\mirror{\qpol}\!\!|$ that induces the grading $\deg_\cP$.
This can easily be done by gluing together the positive/negative cycles each equipped with a vector field for which the integral curves are anticlockwise/clockwise loops starting from the head/tail of the unique arrow in contained the perfect matching. This construction introduces a negative wedge at this head/tail and positive wedges  at the other vertices.
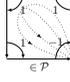
\begin{center}
\resizebox{1.5cm}{!}{\begin{tikzpicture}
\draw [-latex] (7,0)--(7,-2);
\draw [-latex] (7,-2)--(9,-2);
\draw [-latex] (9,-2)--(9,0);
\draw [-latex] (9,0)--(7,0);
\draw [latex-] (7,-0.5) arc (270:360:.5);
\draw [latex-] (7.5,-2) arc (0:90:.5);
\draw [latex-] (9,-1.5) arc (90:180:.5);
\draw [latex-] (8.5,0) arc (180:270:.5);
\draw (7.5,-1.5) node{$1$};
\draw (7.5,-0.5) node{$1$};
\draw (8.5,-0.5) node{$1$};
\draw (8.5,-1.5) node{$-1$};
\draw (8,-2.25) node{$\in \cP$};
\draw[->-,dotted] (9,-2) .. controls +(100:3.5) and +(170:3.5) .. (9,-2);
\draw[->-,dotted] (9,-2) .. controls +(120:3) and +(150:3) .. (9,-2);
\end{tikzpicture}}
\end{center}
This gives precisely the degrees we get from remark \ref{gradingmf}.

To show that the $\cA_\infty$-structure is $\cA_\infty$-isomorphic to $\genmu$ we must distinguish 2 cases.
If $\rect\!\!\mirror{\qpol}$ is well-behaved in the sense of definition \ref{well-behaved} we can use
theorem \ref{productsare1} and theorem \ref{mainthm}.

If $\mirror{\qpol}$ is not well-behaved we will show that we can find a cover $\tilde \qpol\to \qpol$
such that  $\mirror{\tilde\qpol}$ is well-behaved. 
Not being well-behaved for $\mirror{\qpol}$ implies there is a zigzag path $ba$ of length $2$ in $\qpol$ containing two nonconsecutive
arrows of a boundary face. This implies that a subpath of this positive cycle forms a cycle. If this cycle were contractible, the other zigzag path containing
$a$ would enter the interior of this contractible cycle and could only leave this interior via $b$, contradicting zigzag consistency.
So, there is a noncontractible curve contained in just 2 boundary faces.
Because fundamental groups of surfaces are residually finite, we can always take a finite cover of $\qpol$ such that the lifts of all loops contained in just 2 boundary cycles are not loops any more. This implies that mirror of the cover is well-behaved.

The deck transformation group of the finite cover $\tilde\qpol\to\qpol$ acts as a group of
automorphisms of $\cQA(\mirror{\tilde\qpol})$ with its $\cA_\infty$-structure $\genmu$.
Theorem \ref{mainthm} [4] shows that there is an equivariant $\cA_\infty$-isomorphism between
$\cQA(\mirror{\tilde\qpol})$ and $\mf(\tilde\qpol)$. If we look at the invariant
subalgebras we get an $\cA_\infty$-isomorphism between $\cQA(\mirror{\qpol})$ and $\mf(\qpol)$.
\end{proof}
\begin{corollary}
For any consistent dimer $\qpol$ that admits a perfect matching we have that $\H\Tw\mf(\qpol)$ and $\H\Tw\Fuk(|\!\!\mirror{\qpol}\!\!|\setminus \mirror{\qpol}_0)$ are $\cA_\infty$-isomorphic
$\cA_\infty$-categories.
\end{corollary}
Note that the above result applies to all zigzag consistent dimers with a perfect matching, not only those on the torus. This means that the Jacobi algebra of these dimers is not necessarily Noetherian. So in some cases this result does not derive from a commutative
instance of mirror symmetry in the background.

\section{Recovering the commutative version: an example}

In this section we relate our work back to the work by Abouzaid et al. in \cite{Auroux}. In particular we will look at the $4$-punctured sphere in full detail and
indicate the differences between the two approaches if one increases the number of punctures. 

According to \cite{Auroux} the mirror of the $4$-punctured sphere is a hypersurface in a crepant resolution of the conifold,
which is by definition the spectrum of the ring 
$$R = \C[x_{11},x_{12},x_{21},x_{22}]/(x_{11}x_{22}-x_{12}x_{21}).$$
This ring is equal to the degree zero part of the graded ring
\[
S = \C[a_1,a_2,b_1,b_2] \text{ with }\deg a_i=1\text{ and }\deg b_i=-1.
\]
The conifold $X=\Spec R$ has a crepant resolution $\tilde X = \Proj S_{\ge 0}$ and the graded $S_{\ge 0}$-module
$S_{\ge 0}\oplus S_{\ge 0}(1)$ corresponds to a tilting bundle $\ccT=\ccO(0)\oplus \ccO(1)$ on $\tilde X$. 
This means that the category of finitely generated left modules of 
\[
\End(\ccT)= \End_{\mathtt{Grmod}S_{\ge 0}}(S_{\ge 0}\oplus S_{\ge 0}(1))
\]
is derived equivalent to the category of coherent sheaves on $\tilde X$.

The algebra $\End(\ccT)$ can be seen as the path algebra of a quiver with relations: it has two vertices corresponding to the two summands of $\ccT$
and two arrows in both directions corresponding to multiplication with the $a_i$ and the $b_i$. 
The relations between these arrows come from
the commuting relations between the $a_i$ and $b_i$ in $S$. It is well-known that $\End(\ccT)$ is the Jacobi algebra of the following dimer on the torus.
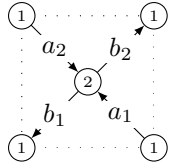
\begin{center}
\begin{tikzpicture} 
\begin{scope}\clip (0pt,0pt) rectangle (50pt,50pt);
\draw[loosely dotted] (0pt,0pt) rectangle (50pt,50pt);
\draw [-latex,shorten >=5pt] (50pt,0pt) to node [fill=white,minimum size=10pt,inner sep=1pt]{{$a_1$}} (25pt,25pt); 
\draw [-latex,shorten >=5pt] (0pt,50pt) to node [fill=white,minimum size=10pt,inner sep=1pt]{{$a_2$}} (25pt,25pt); 
\draw [-latex,shorten >=5pt] (25pt,25pt) to node [fill=white,minimum size=10pt,inner sep=1pt]{{$b_1$}} (0pt,0pt); 
\draw [-latex,shorten >=5pt] (25pt,25pt) to node [fill=white,minimum size=10pt,inner sep=1pt]{{$b_2$}} (50pt,50pt); 
\end{scope}
\node at (0pt,0pt) [circle,draw,fill=white,minimum size=10pt,inner sep=1pt] {\mbox{\tiny $1$}}; 
\node at (50pt,0pt) [circle,draw,fill=white,minimum size=10pt,inner sep=1pt] {\mbox{\tiny $1$}}; 
\node at (0pt,50pt) [circle,draw,fill=white,minimum size=10pt,inner sep=1pt] {\mbox{\tiny $1$}}; 
\node at (50pt,50pt) [circle,draw,fill=white,minimum size=10pt,inner sep=1pt] {\mbox{\tiny $1$}}; 
\node at (25pt,25pt) [circle,draw,fill=white,minimum size=10pt,inner sep=1pt] {\mbox{\tiny $2$}}; 
\end{tikzpicture}
\end{center}
The center of $\cJA(\qpol)$ is isomorphic to $R$ and the special central element $\ell$ corresponds to $a_1b_1a_2b_2$. There are $4$ perfect matchings
each consisting of one arrow.

In the language of toric varieties $X=\Spec R$ comes from the cone
\[
 \sigma =\<(0,0,1),(1,0,1),(1,1,1),(0,1,1)\>
\]
and if  we refine $\sigma$ to the fan generated by 
\[
\<(0,0,1),(1,0,1),(1,1,1)\>\text{ and }\<(0,0,1),(1,1,1),(0,1,1)\>
\]
we get the crepant resolution $\tilde X \to X$. We can draw the intersection of the cones with the plane $Z=1$ for
a pictorial representation of the resolution.
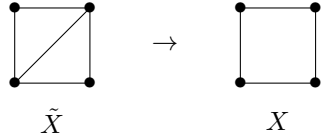
\begin{center}
 \begin{tikzpicture}
\draw (0,0) -- (1,0) -- (1,1) -- (0,1) -- (0,0) -- (1,1);
\node at (.5,-.5) {$\tilde X$};
\node at (2,0.5) {$\to$};
\node at (0,0) {$\bullet$};
\node at (0,1) {$\bullet$};
\node at (1,0) {$\bullet$};
\node at (1,1) {$\bullet$};
\begin{scope}[xshift=3cm]
\draw (0,0) -- (1,0) -- (1,1) -- (0,1) -- (0,0);
\node at (.5,-.5) {$X$};
\node at (0,0) {$\bullet$};
\node at (0,1) {$\bullet$};
\node at (1,0) {$\bullet$};
\node at (1,1) {$\bullet$};
\end{scope}
\end{tikzpicture}
\end{center}
Each of the $4$ lattice points in the square corresponds to a divisor $\tilde X$ given by the graded ideals $a_i S_{\ge 0}$ and
$b_i S\cap S_{\ge 0}$. The former are affine planes and correspond to the upper left and lower right corners of the square, while
the latter are affine planes blown up in a point and correspond to the corners that meet the diagonals. The union of these divisors is
the zero locus of $\ell$ and we denote it by $\tilde X_0$.

In \cite{IU} Ishii and Ueda show that the equivalence
\[
 \ccT \stackrel{L}{\otimes}_A - : \cD^b \mathtt{mod} A \to \cD^b \mathtt{Coh} \tilde X
\]
also induces an equivalence between $\H^0\MF(A,\ell)$ and $\H^0\MF(\tilde X,\ell)$.
The images of the matrix factorizations $\bar P_{a_i}$, $\bar P_{b_j}$ become 
\se{
 \xymatrix{\ccO(0)\ar@<.5ex>[rr]^{a_i}&&\ccO(1)\ar@<.5ex>[ll]^{\ell/a_i}}\text{ and }
 \xymatrix{\ccO(1)\ar@<.5ex>[rr]^{b_i}&&\ccO(0)\ar@<.5ex>[ll]^{\ell/b_i}}.
}
In \cite{orlov1} Orlov shows that $\H^0\MF(\tilde X,\ell)$ is equivalent to $\cD^b\mathtt{coh} \tilde X_0/\mathtt{Perf} \tilde X_0$.
Under this equivalence a matrix factorization $P$ maps to $\Cok p_0$. In our example
each of the matrix factorizations above corresponds to the structure sheaf of one of the divisors.

These are precisely the generators Abouzaid et al. use to generate $\H^0\MF(\tilde X,Z)$. The dual dimer $\mirror \qpol$ in this case consists of 2 squares glued together 
to form a 4-punctured sphere.  This is also the generator that Abouzaid et al. use on the symplectic side. 
So in this particular case we get exactly the same result and $\mf(\qpol)$ 
and $\fuk(\mirror{\qpol})$ are $\cA_\infty$-isormorphic to the category $A$ in \cite{Auroux}. 

If we move beyond the $4$-punctured sphere things become more complicated. If $n=2k+2$ Abouzaid et al. use the toric crepant resolution given by the diagram
illustrated below for $k=5$.
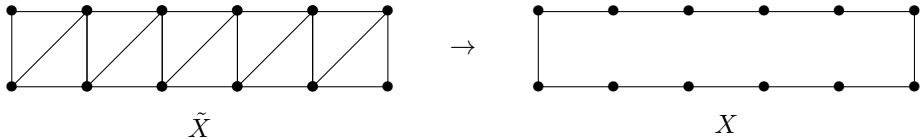
\begin{center}
 \begin{tikzpicture}
\draw (0,0) -- (1,0) -- (1,1) -- (0,1) -- (0,0) -- (1,1);
\node at (0,0) {$\bullet$};
\node at (0,1) {$\bullet$};
\node at (1,0) {$\bullet$};
\node at (1,1) {$\bullet$};
\begin{scope}[xshift=1cm]
\draw (0,0) -- (1,0) -- (1,1) -- (0,1) -- (0,0) -- (1,1);
\node at (0,0) {$\bullet$};
\node at (0,1) {$\bullet$};
\node at (1,0) {$\bullet$};
\node at (1,1) {$\bullet$};
\end{scope}
\begin{scope}[xshift=2cm]
\draw (0,0) -- (1,0) -- (1,1) -- (0,1) -- (0,0) -- (1,1);
\node at (.5,-.5) {$\tilde X$};
\node at (0,0) {$\bullet$};
\node at (0,1) {$\bullet$};
\node at (1,0) {$\bullet$};
\node at (1,1) {$\bullet$};
\end{scope}
\begin{scope}[xshift=3cm]
\draw (0,0) -- (1,0) -- (1,1) -- (0,1) -- (0,0) -- (1,1);
\node at (0,0) {$\bullet$};
\node at (0,1) {$\bullet$};
\node at (1,0) {$\bullet$};
\node at (1,1) {$\bullet$};
\end{scope}
\begin{scope}[xshift=4cm]
\draw (0,0) -- (1,0) -- (1,1) -- (0,1) -- (0,0) -- (1,1);
\node at (2,.5) {$\to$};
\node at (0,0) {$\bullet$};
\node at (0,1) {$\bullet$};
\node at (1,0) {$\bullet$};
\node at (1,1) {$\bullet$};
\end{scope}
\begin{scope}[xshift=7cm]
\draw (0,0) -- (5,0) -- (5,1) -- (0,1) -- (0,0);
\node at (2.5,-.5) {$X$};
\node at (0,0) {$\bullet$};
\node at (0,1) {$\bullet$};
\node at (1,0) {$\bullet$};
\node at (1,1) {$\bullet$};
\node at (2,0) {$\bullet$};
\node at (2,1) {$\bullet$};
\node at (3,0) {$\bullet$};
\node at (3,1) {$\bullet$};
\node at (4,0) {$\bullet$};
\node at (4,1) {$\bullet$};
\node at (5,0) {$\bullet$};
\node at (5,1) {$\bullet$};
\end{scope}
\end{tikzpicture}
\end{center}
Again we can find a tilting bundle and its dimer will look like a $k$-fold cover of the previous dimer.
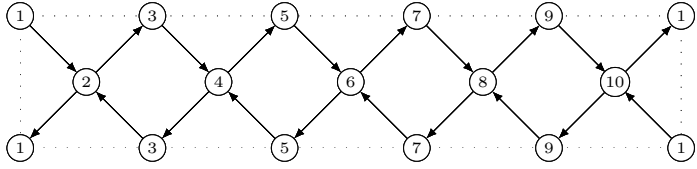
\begin{center}
\begin{tikzpicture} 
\begin{scope}\clip (0pt,0pt) rectangle (250pt,50pt);
\draw[loosely dotted] (0pt,0pt) rectangle (250pt,50pt);
\draw [-latex,shorten >=5pt] (0pt,0pt) to (25pt,-25pt); 
\draw [-latex,shorten >=5pt] (50pt,0pt) to (75pt,-25pt); 
\draw [-latex,shorten >=5pt] (100pt,0pt) to (125pt,-25pt); 
\draw [-latex,shorten >=5pt] (150pt,0pt) to (175pt,-25pt); 
\draw [-latex,shorten >=5pt] (200pt,0pt) to (225pt,-25pt); 
\draw [-latex,shorten >=5pt] (250pt,0pt) to (275pt,-25pt); 
\draw [-latex,shorten >=5pt] (0pt,0pt) to (-25pt,25pt); 
\draw [-latex,shorten >=5pt] (0pt,50pt) to (25pt,25pt); 
\draw [-latex,shorten >=5pt] (250pt,0pt) to (225pt,25pt); 
\draw [-latex,shorten >=5pt] (25pt,25pt) to (0pt,0pt); 
\draw [-latex,shorten >=5pt] (25pt,-25pt) to (50pt,0pt); 
\draw [-latex,shorten >=5pt] (50pt,50pt) to (75pt,25pt); 
\draw [-latex,shorten >=5pt] (50pt,0pt) to (25pt,25pt); 
\draw [-latex,shorten >=5pt] (75pt,25pt) to (50pt,0pt); 
\draw [-latex,shorten >=5pt] (75pt,-25pt) to (100pt,0pt); 
\draw [-latex,shorten >=5pt] (100pt,50pt) to (125pt,25pt); 
\draw [-latex,shorten >=5pt] (100pt,0pt) to (75pt,25pt); 
\draw [-latex,shorten >=5pt] (125pt,25pt) to (100pt,0pt); 
\draw [-latex,shorten >=5pt] (125pt,-25pt) to (150pt,0pt); 
\draw [-latex,shorten >=5pt] (150pt,50pt) to (175pt,25pt); 
\draw [-latex,shorten >=5pt] (150pt,0pt) to (125pt,25pt); 
\draw [-latex,shorten >=5pt] (175pt,25pt) to (150pt,0pt); 
\draw [-latex,shorten >=5pt] (175pt,-25pt) to (200pt,0pt); 
\draw [-latex,shorten >=5pt] (200pt,50pt) to (225pt,25pt); 
\draw [-latex,shorten >=5pt] (200pt,0pt) to (175pt,25pt); 
\draw [-latex,shorten >=5pt] (225pt,25pt) to (200pt,0pt); 
\draw [-latex,shorten >=5pt] (-25pt,-25pt) to (0pt,0pt); 
\draw [-latex,shorten >=5pt] (250pt,50pt) to (275pt,25pt); 
\draw [-latex,shorten >=5pt] (275pt,25pt) to (250pt,0pt); 
\draw [-latex,shorten >=5pt] (225pt,-25pt) to (250pt,0pt); 
\draw [-latex,shorten >=5pt] (0pt,50pt) to (-25pt,75pt); 
\draw [-latex,shorten >=5pt] (250pt,50pt) to (225pt,75pt); 
\draw [-latex,shorten >=5pt] (25pt,75pt) to (0pt,50pt); 
\draw [-latex,shorten >=5pt] (25pt,25pt) to (50pt,50pt); 
\draw [-latex,shorten >=5pt] (50pt,50pt) to (25pt,75pt); 
\draw [-latex,shorten >=5pt] (75pt,75pt) to (50pt,50pt); 
\draw [-latex,shorten >=5pt] (75pt,25pt) to (100pt,50pt); 
\draw [-latex,shorten >=5pt] (100pt,50pt) to (75pt,75pt); 
\draw [-latex,shorten >=5pt] (125pt,75pt) to (100pt,50pt); 
\draw [-latex,shorten >=5pt] (125pt,25pt) to (150pt,50pt); 
\draw [-latex,shorten >=5pt] (150pt,50pt) to (125pt,75pt); 
\draw [-latex,shorten >=5pt] (175pt,75pt) to (150pt,50pt); 
\draw [-latex,shorten >=5pt] (175pt,25pt) to (200pt,50pt); 
\draw [-latex,shorten >=5pt] (200pt,50pt) to (175pt,75pt); 
\draw [-latex,shorten >=5pt] (225pt,75pt) to (200pt,50pt); 
\draw [-latex,shorten >=5pt] (-25pt,25pt) to (0pt,50pt); 
\draw [-latex,shorten >=5pt] (275pt,75pt) to (250pt,50pt); 
\draw [-latex,shorten >=5pt] (225pt,25pt) to (250pt,50pt); 
\draw [-latex,shorten >=5pt] (0pt,0pt) to (25pt,-25pt); 
\draw [-latex,shorten >=5pt] (50pt,0pt) to (75pt,-25pt); 
\draw [-latex,shorten >=5pt] (100pt,0pt) to (125pt,-25pt); 
\draw [-latex,shorten >=5pt] (150pt,0pt) to (175pt,-25pt); 
\draw [-latex,shorten >=5pt] (200pt,0pt) to (225pt,-25pt); 
\draw [-latex,shorten >=5pt] (250pt,0pt) to (275pt,-25pt); 
\draw [-latex,shorten >=5pt] (0pt,0pt) to (-25pt,25pt); 
\draw [-latex,shorten >=5pt] (0pt,50pt) to (25pt,25pt); 
\draw [-latex,shorten >=5pt] (250pt,0pt) to (225pt,25pt); 
\draw [-latex,shorten >=5pt] (25pt,25pt) to (0pt,0pt); 
\draw [-latex,shorten >=5pt] (25pt,-25pt) to (50pt,0pt); 
\draw [-latex,shorten >=5pt] (50pt,50pt) to (75pt,25pt); 
\draw [-latex,shorten >=5pt] (50pt,0pt) to (25pt,25pt); 
\draw [-latex,shorten >=5pt] (75pt,25pt) to (50pt,0pt); 
\draw [-latex,shorten >=5pt] (75pt,-25pt) to (100pt,0pt); 
\draw [-latex,shorten >=5pt] (100pt,50pt) to (125pt,25pt); 
\draw [-latex,shorten >=5pt] (100pt,0pt) to (75pt,25pt); 
\draw [-latex,shorten >=5pt] (125pt,25pt) to (100pt,0pt); 
\draw [-latex,shorten >=5pt] (125pt,-25pt) to (150pt,0pt); 
\draw [-latex,shorten >=5pt] (150pt,50pt) to (175pt,25pt); 
\draw [-latex,shorten >=5pt] (150pt,0pt) to (125pt,25pt); 
\draw [-latex,shorten >=5pt] (175pt,25pt) to (150pt,0pt); 
\draw [-latex,shorten >=5pt] (175pt,-25pt) to (200pt,0pt); 
\draw [-latex,shorten >=5pt] (200pt,50pt) to (225pt,25pt); 
\draw [-latex,shorten >=5pt] (200pt,0pt) to (175pt,25pt); 
\draw [-latex,shorten >=5pt] (225pt,25pt) to (200pt,0pt); 
\draw [-latex,shorten >=5pt] (-25pt,-25pt) to (0pt,0pt); 
\draw [-latex,shorten >=5pt] (250pt,50pt) to (275pt,25pt); 
\draw [-latex,shorten >=5pt] (275pt,25pt) to (250pt,0pt); 
\draw [-latex,shorten >=5pt] (225pt,-25pt) to (250pt,0pt); 
\draw [-latex,shorten >=5pt] (0pt,50pt) to (-25pt,75pt); 
\draw [-latex,shorten >=5pt] (250pt,50pt) to (225pt,75pt); 
\draw [-latex,shorten >=5pt] (25pt,75pt) to (0pt,50pt); 
\draw [-latex,shorten >=5pt] (25pt,25pt) to (50pt,50pt); 
\draw [-latex,shorten >=5pt] (50pt,50pt) to (25pt,75pt); 
\draw [-latex,shorten >=5pt] (75pt,75pt) to (50pt,50pt); 
\draw [-latex,shorten >=5pt] (75pt,25pt) to (100pt,50pt); 
\draw [-latex,shorten >=5pt] (100pt,50pt) to (75pt,75pt); 
\draw [-latex,shorten >=5pt] (125pt,75pt) to (100pt,50pt); 
\draw [-latex,shorten >=5pt] (125pt,25pt) to (150pt,50pt); 
\draw [-latex,shorten >=5pt] (150pt,50pt) to (125pt,75pt); 
\draw [-latex,shorten >=5pt] (175pt,75pt) to (150pt,50pt); 
\draw [-latex,shorten >=5pt] (175pt,25pt) to (200pt,50pt); 
\draw [-latex,shorten >=5pt] (200pt,50pt) to (175pt,75pt); 
\draw [-latex,shorten >=5pt] (225pt,75pt) to (200pt,50pt); 
\draw [-latex,shorten >=5pt] (-25pt,25pt) to (0pt,50pt); 
\draw [-latex,shorten >=5pt] (275pt,75pt) to (250pt,50pt); 
\draw [-latex,shorten >=5pt] (225pt,25pt) to (250pt,50pt); 
\draw [-latex,shorten >=5pt] (0pt,0pt) to (25pt,-25pt); 
\draw [-latex,shorten >=5pt] (50pt,0pt) to (75pt,-25pt); 
\draw [-latex,shorten >=5pt] (100pt,0pt) to (125pt,-25pt); 
\draw [-latex,shorten >=5pt] (150pt,0pt) to (175pt,-25pt); 
\draw [-latex,shorten >=5pt] (200pt,0pt) to (225pt,-25pt); 
\draw [-latex,shorten >=5pt] (250pt,0pt) to (275pt,-25pt); 
\draw [-latex,shorten >=5pt] (0pt,0pt) to (-25pt,25pt); 
\draw [-latex,shorten >=5pt] (0pt,50pt) to (25pt,25pt); 
\draw [-latex,shorten >=5pt] (250pt,0pt) to (225pt,25pt); 
\draw [-latex,shorten >=5pt] (25pt,25pt) to (0pt,0pt); 
\draw [-latex,shorten >=5pt] (25pt,-25pt) to (50pt,0pt); 
\draw [-latex,shorten >=5pt] (50pt,50pt) to (75pt,25pt); 
\draw [-latex,shorten >=5pt] (50pt,0pt) to (25pt,25pt); 
\draw [-latex,shorten >=5pt] (75pt,25pt) to (50pt,0pt); 
\draw [-latex,shorten >=5pt] (75pt,-25pt) to (100pt,0pt); 
\draw [-latex,shorten >=5pt] (100pt,50pt) to (125pt,25pt); 
\draw [-latex,shorten >=5pt] (100pt,0pt) to (75pt,25pt); 
\draw [-latex,shorten >=5pt] (125pt,25pt) to (100pt,0pt); 
\draw [-latex,shorten >=5pt] (125pt,-25pt) to (150pt,0pt); 
\draw [-latex,shorten >=5pt] (150pt,50pt) to (175pt,25pt); 
\draw [-latex,shorten >=5pt] (150pt,0pt) to (125pt,25pt); 
\draw [-latex,shorten >=5pt] (175pt,25pt) to (150pt,0pt); 
\draw [-latex,shorten >=5pt] (175pt,-25pt) to (200pt,0pt); 
\draw [-latex,shorten >=5pt] (200pt,50pt) to (225pt,25pt); 
\draw [-latex,shorten >=5pt] (200pt,0pt) to (175pt,25pt); 
\draw [-latex,shorten >=5pt] (225pt,25pt) to (200pt,0pt); 
\draw [-latex,shorten >=5pt] (-25pt,-25pt) to (0pt,0pt); 
\draw [-latex,shorten >=5pt] (250pt,50pt) to (275pt,25pt); 
\draw [-latex,shorten >=5pt] (275pt,25pt) to (250pt,0pt); 
\draw [-latex,shorten >=5pt] (225pt,-25pt) to (250pt,0pt); 
\draw [-latex,shorten >=5pt] (0pt,50pt) to (-25pt,75pt); 
\draw [-latex,shorten >=5pt] (250pt,50pt) to (225pt,75pt); 
\draw [-latex,shorten >=5pt] (25pt,75pt) to (0pt,50pt); 
\draw [-latex,shorten >=5pt] (25pt,25pt) to (50pt,50pt); 
\draw [-latex,shorten >=5pt] (50pt,50pt) to (25pt,75pt); 
\draw [-latex,shorten >=5pt] (75pt,75pt) to (50pt,50pt); 
\draw [-latex,shorten >=5pt] (75pt,25pt) to (100pt,50pt); 
\draw [-latex,shorten >=5pt] (100pt,50pt) to (75pt,75pt); 
\draw [-latex,shorten >=5pt] (125pt,75pt) to (100pt,50pt); 
\draw [-latex,shorten >=5pt] (125pt,25pt) to (150pt,50pt); 
\draw [-latex,shorten >=5pt] (150pt,50pt) to (125pt,75pt); 
\draw [-latex,shorten >=5pt] (175pt,75pt) to (150pt,50pt); 
\draw [-latex,shorten >=5pt] (175pt,25pt) to (200pt,50pt); 
\draw [-latex,shorten >=5pt] (200pt,50pt) to (175pt,75pt); 
\draw [-latex,shorten >=5pt] (225pt,75pt) to (200pt,50pt); 
\draw [-latex,shorten >=5pt] (-25pt,25pt) to (0pt,50pt); 
\draw [-latex,shorten >=5pt] (275pt,75pt) to (250pt,50pt); 
\draw [-latex,shorten >=5pt] (225pt,25pt) to (250pt,50pt); 
\end{scope}\begin{scope}\clip (-20pt,-20pt) rectangle (270pt,70pt);
\node at (0pt,0pt) [circle,draw,fill=white,minimum size=10pt,inner sep=1pt] {\mbox{\tiny $1$}}; 
\node at (250pt,0pt) [circle,draw,fill=white,minimum size=10pt,inner sep=1pt] {\mbox{\tiny $1$}}; 
\node at (0pt,50pt) [circle,draw,fill=white,minimum size=10pt,inner sep=1pt] {\mbox{\tiny $1$}}; 
\node at (250pt,50pt) [circle,draw,fill=white,minimum size=10pt,inner sep=1pt] {\mbox{\tiny $1$}}; 
\node at (0pt,0pt) [circle,draw,fill=white,minimum size=10pt,inner sep=1pt] {\mbox{\tiny $1$}}; 
\node at (250pt,0pt) [circle,draw,fill=white,minimum size=10pt,inner sep=1pt] {\mbox{\tiny $1$}}; 
\node at (0pt,50pt) [circle,draw,fill=white,minimum size=10pt,inner sep=1pt] {\mbox{\tiny $1$}}; 
\node at (250pt,50pt) [circle,draw,fill=white,minimum size=10pt,inner sep=1pt] {\mbox{\tiny $1$}}; 
\node at (0pt,0pt) [circle,draw,fill=white,minimum size=10pt,inner sep=1pt] {\mbox{\tiny $1$}}; 
\node at (250pt,0pt) [circle,draw,fill=white,minimum size=10pt,inner sep=1pt] {\mbox{\tiny $1$}}; 
\node at (0pt,50pt) [circle,draw,fill=white,minimum size=10pt,inner sep=1pt] {\mbox{\tiny $1$}}; 
\node at (250pt,50pt) [circle,draw,fill=white,minimum size=10pt,inner sep=1pt] {\mbox{\tiny $1$}}; 
\node at (25pt,25pt) [circle,draw,fill=white,minimum size=10pt,inner sep=1pt] {\mbox{\tiny $2$}}; 
\node at (25pt,25pt) [circle,draw,fill=white,minimum size=10pt,inner sep=1pt] {\mbox{\tiny $2$}}; 
\node at (25pt,25pt) [circle,draw,fill=white,minimum size=10pt,inner sep=1pt] {\mbox{\tiny $2$}}; 
\node at (50pt,0pt) [circle,draw,fill=white,minimum size=10pt,inner sep=1pt] {\mbox{\tiny $3$}}; 
\node at (50pt,50pt) [circle,draw,fill=white,minimum size=10pt,inner sep=1pt] {\mbox{\tiny $3$}}; 
\node at (50pt,0pt) [circle,draw,fill=white,minimum size=10pt,inner sep=1pt] {\mbox{\tiny $3$}}; 
\node at (50pt,50pt) [circle,draw,fill=white,minimum size=10pt,inner sep=1pt] {\mbox{\tiny $3$}}; 
\node at (50pt,0pt) [circle,draw,fill=white,minimum size=10pt,inner sep=1pt] {\mbox{\tiny $3$}}; 
\node at (50pt,50pt) [circle,draw,fill=white,minimum size=10pt,inner sep=1pt] {\mbox{\tiny $3$}}; 
\node at (75pt,25pt) [circle,draw,fill=white,minimum size=10pt,inner sep=1pt] {\mbox{\tiny $4$}}; 
\node at (75pt,25pt) [circle,draw,fill=white,minimum size=10pt,inner sep=1pt] {\mbox{\tiny $4$}}; 
\node at (75pt,25pt) [circle,draw,fill=white,minimum size=10pt,inner sep=1pt] {\mbox{\tiny $4$}}; 
\node at (100pt,0pt) [circle,draw,fill=white,minimum size=10pt,inner sep=1pt] {\mbox{\tiny $5$}}; 
\node at (100pt,50pt) [circle,draw,fill=white,minimum size=10pt,inner sep=1pt] {\mbox{\tiny $5$}}; 
\node at (100pt,0pt) [circle,draw,fill=white,minimum size=10pt,inner sep=1pt] {\mbox{\tiny $5$}}; 
\node at (100pt,50pt) [circle,draw,fill=white,minimum size=10pt,inner sep=1pt] {\mbox{\tiny $5$}}; 
\node at (100pt,0pt) [circle,draw,fill=white,minimum size=10pt,inner sep=1pt] {\mbox{\tiny $5$}}; 
\node at (100pt,50pt) [circle,draw,fill=white,minimum size=10pt,inner sep=1pt] {\mbox{\tiny $5$}}; 
\node at (125pt,25pt) [circle,draw,fill=white,minimum size=10pt,inner sep=1pt] {\mbox{\tiny $6$}}; 
\node at (125pt,25pt) [circle,draw,fill=white,minimum size=10pt,inner sep=1pt] {\mbox{\tiny $6$}}; 
\node at (125pt,25pt) [circle,draw,fill=white,minimum size=10pt,inner sep=1pt] {\mbox{\tiny $6$}}; 
\node at (150pt,0pt) [circle,draw,fill=white,minimum size=10pt,inner sep=1pt] {\mbox{\tiny $7$}}; 
\node at (150pt,50pt) [circle,draw,fill=white,minimum size=10pt,inner sep=1pt] {\mbox{\tiny $7$}}; 
\node at (150pt,0pt) [circle,draw,fill=white,minimum size=10pt,inner sep=1pt] {\mbox{\tiny $7$}}; 
\node at (150pt,50pt) [circle,draw,fill=white,minimum size=10pt,inner sep=1pt] {\mbox{\tiny $7$}}; 
\node at (150pt,0pt) [circle,draw,fill=white,minimum size=10pt,inner sep=1pt] {\mbox{\tiny $7$}}; 
\node at (150pt,50pt) [circle,draw,fill=white,minimum size=10pt,inner sep=1pt] {\mbox{\tiny $7$}}; 
\node at (175pt,25pt) [circle,draw,fill=white,minimum size=10pt,inner sep=1pt] {\mbox{\tiny $8$}}; 
\node at (175pt,25pt) [circle,draw,fill=white,minimum size=10pt,inner sep=1pt] {\mbox{\tiny $8$}}; 
\node at (175pt,25pt) [circle,draw,fill=white,minimum size=10pt,inner sep=1pt] {\mbox{\tiny $8$}}; 
\node at (200pt,0pt) [circle,draw,fill=white,minimum size=10pt,inner sep=1pt] {\mbox{\tiny $9$}}; 
\node at (200pt,50pt) [circle,draw,fill=white,minimum size=10pt,inner sep=1pt] {\mbox{\tiny $9$}}; 
\node at (200pt,0pt) [circle,draw,fill=white,minimum size=10pt,inner sep=1pt] {\mbox{\tiny $9$}}; 
\node at (200pt,50pt) [circle,draw,fill=white,minimum size=10pt,inner sep=1pt] {\mbox{\tiny $9$}}; 
\node at (200pt,0pt) [circle,draw,fill=white,minimum size=10pt,inner sep=1pt] {\mbox{\tiny $9$}}; 
\node at (200pt,50pt) [circle,draw,fill=white,minimum size=10pt,inner sep=1pt] {\mbox{\tiny $9$}}; 
\node at (225pt,25pt) [circle,draw,fill=white,minimum size=10pt,inner sep=1pt] {\mbox{\tiny $10$}}; 
\node at (225pt,25pt) [circle,draw,fill=white,minimum size=10pt,inner sep=1pt] {\mbox{\tiny $10$}}; 
\node at (225pt,25pt) [circle,draw,fill=white,minimum size=10pt,inner sep=1pt] {\mbox{\tiny $10$}}; 
\end{scope}\end{tikzpicture}
\end{center}
The 4 sets of all arrows that point in the same direction will give us 4 perfect matchings (but these are not the only ones).

Unfortunately the matrix factorizations $\bar P_a$ will not correspond anymore to the divisors of the lattice points and hence there is no isomorphism
between $\mf(\qpol)$ and the category $A$ in \cite{Auroux}. This is to be expected because the dual dimer $\mirror{\qpol}$ tiles
the sphere with $2k$ squares, while $A=\fuk(\qpol')$ for a dimer $\qpol'$ that tiles the sphere with $2$ $2k+2$-gons.
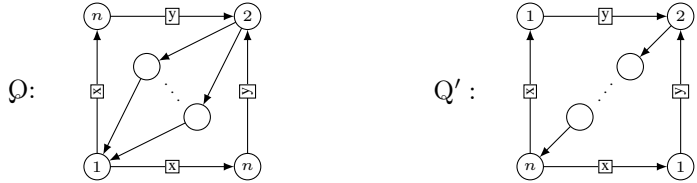
\begin{center}
\begin{tikzpicture}
\draw [-latex,shorten >=5pt] (0cm,0cm) to node [rectangle,draw,fill=white,sloped,inner sep=1pt] {{\tiny x}} (2cm,0cm); 
\draw [-latex,shorten >=5pt] (0cm,2cm) to node [rectangle,draw,fill=white,sloped,inner sep=1pt] {{\tiny y}} (2cm,2cm); 
\draw [-latex,shorten >=5pt] (0cm,0cm) to node [rectangle,draw,fill=white,sloped,inner sep=1pt] {{\tiny x}} (0cm,2cm); 
\draw [-latex,shorten >=5pt] (2cm,0cm) to node [rectangle,draw,fill=white,sloped,inner sep=1pt] {{\tiny y}} (2cm,2cm); 
\draw [-latex,shorten >=5pt] (2cm,2cm) to (1.33cm,0.66 cm); 
\draw [-latex,shorten >=5pt] (0.66 cm,1.33cm) to (0cm,0cm); 
\draw [-latex,shorten >=5pt] (2cm,2cm) to (0.66cm,1.33cm); 
\draw [-latex,shorten >=5pt] (1.33cm,0.66 cm) to (0cm,0cm); 
\node at (-1cm,1cm) {\mbox{$\mirror{\qpol}:$}}; 
\node at (0cm,0cm) [circle,draw,fill=white,minimum size=10pt,inner sep=1pt] {\mbox{\tiny $1$}}; 
\node at (0cm,2cm) [circle,draw,fill=white,minimum size=10pt,inner sep=1pt] {\mbox{\tiny $n$}}; 
\node at (2cm,0cm) [circle,draw,fill=white,minimum size=10pt,inner sep=1pt] {\mbox{\tiny $n$}}; 
\node at (2cm,2cm) [circle,draw,fill=white,minimum size=10pt,inner sep=1pt] {\mbox{\tiny $2$}}; 
\node at (1.33cm,0.66cm) [circle,draw,fill=white,minimum size=10pt,inner sep=1pt] {}; 
\node at (0.66cm,1.33cm) [circle,draw,fill=white,minimum size=10pt,inner sep=1pt] {}; 
\node at (1cm,1.1cm) {\mbox{\tiny $\ddots$}}; 
\end{tikzpicture}
\hspace{2cm}
\begin{tikzpicture}
\node at (-1cm,1cm) {\mbox{$\qpol':$}}; 
\draw [-latex,shorten >=5pt] (0cm,0cm) to node [rectangle,draw,fill=white,sloped,inner sep=1pt] {{\tiny x}} (2cm,0cm); 
\draw [-latex,shorten >=5pt] (0cm,2cm) to node [rectangle,draw,fill=white,sloped,inner sep=1pt] {{\tiny y}} (2cm,2cm); 
\draw [-latex,shorten >=5pt] (0cm,0cm) to node [rectangle,draw,fill=white,sloped,inner sep=1pt] {{\tiny x}} (0cm,2cm); 
\draw [-latex,shorten >=5pt] (2cm,0cm) to node [rectangle,draw,fill=white,sloped,inner sep=1pt] {{\tiny y}} (2cm,2cm); 
\draw [-latex,shorten >=5pt] (2cm,2cm) to (1.33 cm,1.33 cm); 
\draw [-latex,shorten >=5pt] (0.66 cm,0.66cm) to (0cm,0cm); 
\node at (0cm,0cm) [circle,draw,fill=white,minimum size=10pt,inner sep=1pt] {\mbox{\tiny $n$}}; 
\node at (0cm,2cm) [circle,draw,fill=white,minimum size=10pt,inner sep=1pt] {\mbox{\tiny $1$}}; 
\node at (2cm,0cm) [circle,draw,fill=white,minimum size=10pt,inner sep=1pt] {\mbox{\tiny $1$}}; 
\node at (2cm,2cm) [circle,draw,fill=white,minimum size=10pt,inner sep=1pt] {\mbox{\tiny $2$}}; 
\node at (1cm,1.1cm) {\mbox{\tiny $\mirror{\ddots}$}}; 
\node at (1.33cm,1.33cm) [circle,draw,fill=white,minimum size=10pt,inner sep=1pt] {}; 
\node at (0.66cm,0.66cm) [circle,draw,fill=white,minimum size=10pt,inner sep=1pt] {}; 
\end{tikzpicture}
\end{center}
These two dimers have the same genus and number of punctures, so the twisted completions of $\fuk(\mirror{\qpol})$ and $A$ are the same.
One can construct the arrows in $\qpol'$ as complexes in $\fuk(\mirror{\qpol})$ and analogously one can check that the structure sheaves of the divisors
can be constructed by taking cones between the matrix factorizations coming from the arrows in $\qpol$.
\vspace{.3cm}

All this fits together in a broader framework which uses dimer models and GIT-quotients to construct
crepant resolutions \cite{Ishii,Moz,MB}. The main idea is that for each consistent dimer model on a torus and each generic stability condition
one can construct an equivalence between the Jacobi algebra and a commutative crepant resolution of a toric Gorenstein singularity. 
This also induces an equivalence between matrix factorizations of $\ell$ in both the dimer and the commutative crepant resolution.
In a follow-up paper we will explore this in more detail and construct an equivalence between the Karoubi completions of the category of singularities of the hypersurface defined by $\ell$ 
and the derived wrapped Fukaya category of the dual dimer.

\setcounter{section}{0}
\renewcommand\thesection{\Alph{section}}

\section{Appendix: Hochschild Cohomology and $\cA_\infty$-structures}\label{appendix}

In this appendix we look at the connection
between Hochschild cohomology and $\cA_{\infty}$-structures and calculate the Hochschild cohomology of the gentle categories coming from rectified quivers. 

We will define a certain $\cA_\infty$-structure $\genmu$ on these gentle categories and 
use our calculation to find criteria for when a given $\cA_{\infty}$-structure will be $\cA_\infty$-isomorphic to $\genmu$.

\subsection{Hochschild cohomology and $\cA_{\infty}$-structures}

The discussion in this section closely matches \cite{Auroux}.
If we are interested in $\cA_{\infty}$-structures on an ordinary category $\cB$ up to $\cA_\infty$-isomorphism, 
we need to study the Hochschild cohomology of $\cB$. 

A length $n$ multi-functor $\cF$ consists of linear maps for 
each sequence of $n+1$ objects $X_0,\dots, X_n\in \Ob\cB$. 
\[
 \cF : \Hom_\cB(X_{1},X_{0})\otimes\cdots \otimes \Hom_\cB(X_{n},X_{n-1})\to \Hom_\cC(X_{n},X_{0}). 
\] 
The set of all length $n$-multifunctors forms a $\Z$-graded vector space which we denote by $\cM^n$.

We can construct a differential $d:\cM^n\to \cM^{n+1}$
\se{
 d\cF(f_1,\dots,f_{n+1}) =& f_1\cF(f_2,\dots,f_{n+1})\\
-\cF(f_1f_2,\dots,f_{n+1})&+\cF(f_1,f_2f_3,\dots,f_{n+1})\dots \pm\cF(f_1,\dots, f_{n-1}f_{n})
\\&\mp \cF(f_1,\dots, f_{n-1})f_{n} 
}
The \emph{Hochschild cohomology} of $\cB$ is defined as the cohomology of the complex $\cM^\bullet$:
\[
 \Hoch(\cB) := \H(\cM^\bullet).
\]

A finite group $\grp G$ acting functorially on $\cB$ will give us a linear $\grp G$-action
on the space of multifunctors:
\[
 g : \cF \mapsto g^{-1}\circ \cF \circ g^{\otimes k}.
\]
This action commutes with the differential $d$ and the Hochschild cohomology splits as a direct sum
according to the different characters of $\grp G$.

The importance of the Hochschild cohomology of $\cB$ can be seen by the following $2$ lemmas.
\begin{lemma}\cite{Auroux}
Let $\mu_i$, $i\le k$ be a sequence of multifunctors with $\mu_1=0$ and $\mu_2=\cdot$ then we can rewrite $\mathrm{[M_k]}$ as
\[
 d\mu_k = \Phi
\]
where $\Phi$ is an expression calculated from the $\mu_i, i<k$ and 
$d\Phi=0$ if all $\mathrm{[M_i]}$ for $<k$ hold.

Moreover if the $\mu_i$ are invariant under the $\grp G$-action then $\Phi$ is invariant as well.
\end{lemma}
\begin{lemma}\cite{Auroux}
Let $\mu$ and $\nu$ be two $\cA_{\infty}$-structures on $\cB$ and 
let
$\cF_i$, $i\le k$ be a sequence of multifunctors with $\cF_1=\Id$ then we can rewrite $\mathrm{[F_{k+1}]}$ as
\[
 d\cF_k = \Psi
\]
where $\Psi$ is an expression calculated from the $\mu_i$, the $\nu_i$ and the $\cF_i, i<k$. 
Moreover $d\Psi=0$ if all $\mathrm{[F_i]}$ for $i<k$ hold.

Moreover if the $\mu_i, \nu_i$ and $\cF_i$ are invariant under the $\grp G$-action then $\Psi$ is invariant as well.
\end{lemma}
These lemmas imply that we can find a solution for $\mu_k$ (or $\cF_k$) if and only if 
the homology class of $\Phi$ (or $\Psi$) is trivial. 

\begin{remark}
The $\grp G$-invariance is not discussed in \cite{Auroux} but follows easily from the expressions for $\Psi$ and $\Phi$.
\end{remark}

\begin{lemma}\label{gaction}
Let $\cB$ be an algebra, $\grp G$ a finite group of automorphisms of $\cB$, and $\mu$, $\nu$ two $\grp G$-invariant
$\cA_\infty$-structures.
If $\mu$ and $\nu$ are $\cA_\infty$-isomorphic then the corresponding $\cA_\infty$-structures on the ring of invariants
$\cB^\grp G$ are also $\cA_\infty$-isomorphic.
\end{lemma}
\begin{proof}
Let $\cF$ be the $\cA_\infty$-isomorphism between $\mu$ and $\nu$. We can split $\cF_k$ as a direct sum of a $\grp G$-invariant 
part and some other stuff $\cF^\grp G_k\oplus \text{rest}$.
The equation $d\cF_k=\Psi$ splits in $d\cF_k^{\grp G}=\Psi$ and $d\text{rest}=0$.
This shows that $\cF_k^{\grp G}$ is also an $\cA_\infty$-isomorphism between $\mu$ and $\nu$.
Clearly $\cF_k^{\grp G}|_{{\cB^\grp G}^\otimes k}$ maps to ${\cB^\grp G}$, so it gives us an
$\cA_\infty$-isomorphism between $\mu|_{\cB^\grp G}$ and $\nu|_{\cB^\grp G}$.
\end{proof}

\subsection{Hochschild cohomology of the gentle categories}

We will calculate the Hochschild cohomology of the gentle category $\cB = \cQA(\rect\qpol)$ where $\qpol$ is any embedded quiver.
We will view $\cB$ as an algebra in this section.
For the calculation 
we will use a minimal bimodule resolution for $\cB$, which can be obtained by a result by Bardzel \cite{bardzell}. 
\begin{theorem}[Bardzel]
Suppose $\cB$ is the path algebra of a quiver $Q$ modulo relations that are all paths of length two.
Let $Z_k \subset \C Q$ be the vector space spanned by paths of length $k$ of which all subpaths of length $2$ are zero in $\cB$.

The terms in minimal bimodule resolution $\mathtt{P}^\bullet$ 
are given by $\mathtt{P}_k = \cB\otimes Z_k \otimes \cB$ where the tensor product is taken over $\C Q_0$.
The maps between the terms have the following form
\[
 1\otimes b_1\dots b_k \otimes 1 \mapsto b_1\otimes b_{2}\dots b_k \otimes 1 - (-1)^k \otimes b_1\dots b_{k-1}\otimes b_k
\]
\end{theorem}
As is well known the Hochschild cohomology of $\cB$ is the cohomology of the complex
\[
 \Hom_{\cB\otimes \cB^{opp}}(\mathtt{P}^\bullet,\cB).
\]
In the case gentle algebras $\cB = \cQA(\rect\qpol)$,
the complex $\mathtt{P}^{\bullet>1}$ splits as a direct sum where each summand $\mathtt{P}_\beta^\bullet$ 
only contains the paths that go around one cycle
$\beta \in \rect\qpol_2^+$. So let us focus on one such cycle $\beta=\beta_1\dots \beta_l$ and 
set $\beta_k$ to be $\beta_{k \text{ mod } l}$ for all $k \in \Z$. By construction $\beta_{k}=\beta_{k'}$ if and only
if $k=k' \text{ mod } l$. Now
\[
 \Hom_{\cB\otimes \cB^{opp}}(\cB\otimes \beta_i\dots \beta_j\otimes \cB,\cB) \cong  h(\beta_i)\cB t(\beta_j)
\]
and we will use the notation $[\beta_i\dots \beta_j\to \xi]$ for the morphism that maps
$\beta_i\dots \beta_j$ to $\xi \in h(\beta_i)\cB t(\beta_j)$. 

\begin{proposition}
For a gentle algebra $\cQA(\rect \qpol)$ the homology of $\Hom(\mathtt{P}^{\bullet}_\beta,\cB)$ 
in degree $nl$ with $n\ge 1$ is one and spanned by 
\[
\Omega_0^{\beta,n} = \sum_{i=1}^{l} (-1)^{i(nl-1)}[\beta_{i}\dots \beta_{i+nl-1}\to h(\beta_i)],  
\]
and in degree $nl+1$ with $n\ge 1$ it is one and spanned by
\[
\Omega_1^{\beta,n} = \sum_{i=1}^{l} (-1)^{i(nl+1)}[\beta_i \dots \beta_{i+nl} \to \beta_i]. 
\]
in all other degrees $>1$ it is zero.
\end{proposition}
\begin{proof}
By lemma \ref{basispaths}, every element in $h(\beta_i)\cB t(\beta_j)$ is a linear combination of
paths $\xi=\alpha_u\dots \alpha_{v}$ where the arrows are consecutive arrows in a negative cycle $\alpha$.
If $h(\beta_i)=t(\beta_j)$ we also have the degenerate case $h(\beta_i) \in h(\beta_i)\cB t(\beta_j)$.

With the notation above, the differential on $\Hom_{\cB\otimes \cB^{opp}}(\mathtt{P}^\bullet,\cB)$
becomes
\[
 d[\beta_i\dots \beta_j \to \xi] = [\beta_{i-1}\dots \beta_j \to \beta_{i-1}\xi] -(-1)^{j-i} [\beta_{i}\dots \beta_{j+1} \to  \xi\beta_{j+1}] 
\]
and both $\Ker d$ and $\Image d$ split as a direct sum of spaces $(\Ker d)_{r,s}$ and $(\Image d)_{r,s}$ generated by elements
$$[\underbrace{\beta_i\dots \beta_j}_{r=j-i+1} \to \underbrace{\alpha_u\dots \alpha_{v}}_{s=v-u+1}]$$ 
with fixed lengths $r,s$.

Observe that if $r,s>0$ we have that either $\alpha_u=\beta_{i}$ or $\alpha_{u-1}=\beta_{i-1}$ (by $\alpha_{u-1}$ we mean the arrow
proceeding $\alpha_u$ in the negative cycle). This follows from the fact that $\rect\qpol$ is a rectified quiver: in every vertex 
each of the two positive cycles shares an arrow with each of the two negative cycles.
Similarly either $\alpha_v=\beta_{j}$ or $\alpha_{v+1}=\beta_{j+1}$.

The first term of the differential $d[\beta_i\dots \beta_j \to \xi]$ is zero if and only if $\alpha_u=\beta_{i}$ and the second term
if and only if $\alpha_v=\beta_{j}$. Note also that if a certain term appears as the left term of
$d[\beta_i\dots \beta_j \to \xi]$ then it can only appear a second time as the right term of $d[\beta_{i-1}\dots \beta_{j-1} \to \beta_{i-1}\xi\beta_{j}^{-1}]$.
Similarly, if a term appears as a right term of $d[\beta_i\dots \beta_j \to \xi]$ 
then it can only appear a second time as the left term of $d[\beta_{i+1}\dots \beta_{j+1} \to \beta_{i}^{-1}\xi\beta_{j+1}]$.  

To calculate $(\Ker d)_{r,s}/(\Image d)_{r,s}$ with $r>1$, we consider 3 cases:
\begin{itemize}
 \item[\framebox{$s>1$}] 
If $\xi$ has length at least $2$ and the left term of $d[\beta_i\dots \beta_j \to \xi]$ also occurs in another $d[\dots]$ 
then  $\beta_{i-1}\xi\beta_{j}^{-1}$ is nonzero. Therefore $\xi$ ends in $\beta_j$ and the right term of 
$d[\beta_i\dots \beta_j \to \xi]$ is zero. Similarly if the right term of $d[\beta_i\dots \beta_j \to \xi]$ occurs twice
then the left term is zero.

Therefore $(\Ker d)_{r,s}$ with $s>1$ is spanned by 
elements of the form $[\beta_i\dots \beta_j \to \xi] -(-1)^{j-i} [\beta_{i-1}\dots \beta_{j-1} \to \beta_{i-1}\xi\beta_{j}^{-1}]$
with $\alpha_v=\beta_j$ and $\alpha_u\ne \beta_i$
and elements of the form $[\beta_i\dots \beta_j \to \xi]$ with both $\alpha_v=\beta_j$ and $\alpha_u= \beta_i$.

Both types can be written as $d[\beta_i\dots \beta_{j-1} \to \xi\beta_{j}^{-1}]$ so 
$(\Ker d)_{r,s}/(\Image d)_{r,s}=0$. 
\item[\framebox{$s=1$}]
If $\xi$ has length $1$ then either $\xi$ is an arrow of the positive cycle $\beta$ or not.
In the latter case $d[\beta_i\dots \beta_j \to \xi]$ will have two nonzero terms, which each
do not appear in any other $d[\dots]$ so we cannot combine $[\beta_i\dots \beta_j \to \xi]$ with other things to create 
an element in $(\Ker d)_{r,1}$.

In the former case $d[\beta_i\dots \beta_j \to \xi]=0$ and we have that $\xi=\beta_i=\beta_j$ and $j-i=0 \mod l$. Now
$[\beta_i\dots \beta_j \to \beta_i]- (-1)^{j-i+1} [\beta_{i-1}\dots \beta_{j-1} \to \beta_{i-1}]$ is in $(\Image d)_{r,1}$ so
$(\Ker d)_{r,1}/(\Image d)_{r,1}$ is generated by
\[
\Omega_1^{\beta,n} =  \sum_{0\le u\le l-1}(-1)^{u(j-i+1)}[\beta_{i+u}\dots \beta_{j+u}\to \beta_{i+u}]
\] 
if $r=nl+1$ for some $n$ and zero otherwise.
\item[\framebox{$s=0$}] If $\xi$ has length zero and $\beta_i\dots \beta_j$ is not a power of the full cycle $\beta$ then
both terms of $d[\beta_i\dots \beta_j \to \xi]$ do not occur in another $d[\dots]$ so we cannot combine
such $[\beta_i\dots \beta_j \to \xi]$ to a cocycle. Therefore $(\Ker d)_{r,0}=0$ if $r\ne 0\text{ mod } l$.

If $j=i+nl-1$ then $d[\beta_i\dots \beta_j \to h(\beta_i)]=\beta_{i-1}\dots \beta_j \to \beta_{i-1} -(-1)^{j-i} \beta_{i}\dots \beta_{j+1} \to \beta_{j+1}$
and we can make a cocycle by adding all $l$ cyclic shifts together.
Hence, the kernel
is spanned by 
\[
\Omega_0^{\beta,n} = \sum_{0\le u< l}(-1)^{u(l-1)}[\beta_{i+u}\dots \beta_{j+u}\to h(\beta_{i+u})]
\]
while $(\Image d)_{r,0}=0$.
\end{itemize}
\end{proof}

\begin{remark}
In \ref{Xgrading} we explained that a vector field $X$ on the surface $|\qpol|\setminus \qpol_0$ induces a grading $\deg_X$ on $\cQA(\rect \qpol)$ and by consequence also a $\Z$-grading on the Hochschild cohomology. 
The degree of any positive cycle of length $l$ in $\cQA(\rect \qpol)$ is $l-2$, so
according to this grading we have
\se{
\deg_X \Omega^{\beta,n}_0 = \deg_X \Omega^{\beta,n}_1= n(2-l)
}
\end{remark}

Now we go back to the original viewpoint of the Hochschild cohomology. 
Given a multifunctor $\Psi$ of length $k$ with $d\Psi=0$, how can we detect whether its cohomology class is zero?
The idea is to look at the image of $k$-tuples of the form $(\beta_1,\dots,\beta_k)$ where the $\beta_i$ are consecutive arrows in a positive cycle:

\begin{lemma}\label{whenhomzero}
Let $\Psi$ be a multifunctor of length $u$ with $d\Psi=0$ then
$\Psi \in \Image d$ if $\Psi(\beta_1,\dots,\beta_u)$ has no length $0$ terms or length $1$ terms for
every $(\beta_1,\dots,\beta_u)$ where the $\beta_i$ are consecutive arrows in a positive cycle.
\end{lemma}
\begin{proof}
For $u=0\mod l$ the homology class $\Omega_0^{\beta,n}(\beta_i,\dots, \beta_{i+nl-1})=h(\beta_i)$ contains a length $0$ term.
Now if $\Pi$ is any multifunctor of length $nl-1$ then one can check that
\[
 d\Pi(\beta_i,\dots, \beta_{i+nl-1})=
\beta_i \Pi(\beta_{i+1},\dots, \beta_{i+nl-1}) -0 +\dots \mp 0 \pm \Pi(\beta_i,\dots, \beta_{i+nl-2})\beta_{i+nl-1}
\]
has no terms of length $\le 1$ because neither  $\Pi(\beta_{i+1},\dots, \beta_{i+nl-1})$ nor $\Pi(\beta_i,\dots, \beta_{i+nl-2})$ 
can be a vertex. So if $\Psi= \kappa \Omega_0^{\beta,n} +d\Pi$ has no length zero terms then $\kappa=0$ and $\Psi \in \Image d$. 

Similarly for $u=1\mod l$, if $\Pi$ is any multifunctor of length $nl$ then
we can write $d\Pi(\beta_i,\dots, \beta_{i+nl})=\lambda^{\Pi}_i \beta_i + \dots$ where $\lambda^{\Pi}_i\in \C$ can be determined
from the constant terms in $\Pi(\beta_{i+1},\dots, \beta_{i+nl})$ and $\Pi(\beta_i,\dots, \beta_{i+nl-1})$.
Therefore one easily checks that $$\sum_{i=1}^{nl} (-1)^{(nl+1)i}\lambda_i=0.$$

On the other hand
$\Omega_1^{c,n}(\beta_i,\dots, \beta_{i+nl})=\lambda^{\Omega}_i \beta_i$  with $\lambda^{\Omega}_i=(-1)^{i(nl+1)}$
so in that case we get $\sum_{i=1}^{nl} (-1)^{(nl+1)i}\lambda^{\Omega}_i={nl}\ne 0$.
So if $\Psi= \kappa \Omega_0^{\beta,1} +d\Pi$ has no length 1 terms then all $\lambda^{\Pi}_i=0$ so
$$
0 = \sum_{i=1}^{nl} (-1)^{(nl+1)i}\lambda^{\Psi}_i = \sum_{i=1}^{nl} (-1)^{(nl+1)i}(\lambda^{\Pi}_i +\kappa \lambda^{\Omega}_i) 
$$
So $\kappa=0$ and $\Psi \in \Image d$. 
\end{proof}

\subsection{An $\cA_{\infty}$-structure on $\cQA(\rect\qpol)$}
We will now describe a specific $\cA_{\infty}$-structure on $\cQA(\rect\qpol)$, which can be constructed inductively. 
For any sequence of paths $\rho_1,\dots,\rho_k$ and any cycle $\beta_1\dots \beta_l \in \rect \qpol_2^+$ with $h(\beta_1)=t(\rho_i)$ 
we set 
\[
\genmu(\rho_1,\dots, \rho_i\beta_{1},\beta_{2},\dots, \beta_{l-1},\beta_{l}\rho_{i+1},\dots, \rho_k) := 
(-1)^s \genmu(\rho_1,\dots,\rho_k).
\]
with sign convention $s =l(\rho_1+\dots + \rho_{i} + k-i)$. This can be illustrated by the following diagram: 
\[
\genmu\left(
\vcenter{\xymatrix@=.3cm{&\ar[dr]&\\
\ar[ur]&&\ar[lldd]_(.8){\rho_{i}\beta_1}\\
&&\\
\dots&&\dots\ar[lluu]_(.2){\beta_l\rho_{i+1}}
}}
\right)=\pm 
\genmu\left(
\vcenter{\xymatrix@=.3cm{
&\ar[ld]_{\rho_{i}}&\\
\dots&&\dots\ar[lu]_{\rho_{i+1}}
}}
\right).
\]
For $k>2$ we set $\genmu(\sigma_1,\dots, \sigma_k)=0$ if we cannot perform any reduction of the form above and for 
$k=2$ we use the ordinary product on $\cQA(\rect \qpol)$.

\begin{lemma}
The rule above makes $\genmu_u$ well defined for all $u$.
\end{lemma}
\begin{proof}
We have to check that if there are two positive cycles $\alpha_1\dots \alpha_r$ and $\beta_1\dots \beta_s$ that can be used to reduce 
the product, the end result will not depend on the order of the reduction. 
If we look at an entry that is equal to an arrow, we can uniquely identify the cycle we need to reduce this
arrow away. This is because each arrow sits in just one positive cycle. 

If we have two cycles that can be reduced, our big product either looks like
\[
  \genmu_{r+s+k-4}(\rho_1,\dots, \rho_{i}\alpha_1,\alpha_{2},\dots, \alpha_{r-1},\alpha_r\rho_{i+1},\dots,\rho_{j}\beta_1,\beta_{2},\dots \beta_{s-1},\beta_s\rho_{j+1},\dots, \rho_k) 
\]
or
\[
  \genmu_{r+s+l-4}(\rho_1,\dots, \rho_{i}\alpha_1,\alpha_{2},\dots, \alpha_{r-1},\alpha_r\rho_{i+1}\beta_1,\beta_{2},\dots \beta_{s-1},\beta_s\rho_{i+2},\dots, \rho_k) 
\]
In both cases it is clear that, up to a sign, first reducing $\alpha_1\dots \alpha_r$ and then $\beta_1\dots \beta_s$  
gives the same result as reducing in 
the opposite order. To show that the signs are the same, let us look at the first $\genmu_{r+s+l-4}$. 
If we first reduce $\alpha_1\dots \alpha_r$ and then $\beta_1\dots \beta_s$, we get a total sign 
\[
s_{tot} =  r(\rho_1+\dots + \rho_{i} + k-i-1+s-2) + s(\rho_1+\dots + \rho_{j} +k-j-1).
\]
If we do it the other way round we get
\[
s_{tot} =   s(\rho_1+\dots + \rho_{i} + \alpha_1+\dots + \alpha_r +k-j-1) + r(\rho_1+\dots + \rho_{i} + k-i-1).
\]
Because of the construction of the gradation in section \ref{rect}, $\alpha_1\dots \alpha_r$ 
has odd degree if and only if $r$ is odd. 
Therefore $s(\alpha_1+\dots +\alpha_r)+r(s-2)=0\mod 2$.
The second sign calculation is similar.
\end{proof}

\begin{lemma}\label{rules}
The $\genmu_u$, $u>2$ have the following properties: 
For any paths $\rho_1,\dots,\rho_u$ we have 
\begin{enumerate}
 \item $\genmu(\rho_1,\dots ,\rho_u)$ is homotopic to $\rho_1 \cdots \rho_u$ viewed as a path in $|\qpol|\setminus \qpol_0$, i.e. the surface in which the original quiver is embedded with the original vertices removed.
 \item If $\genmu(\rho_1,\dots, \rho_u)=\pm \sigma$ is not a trivial path then either $\rho_1:=\sigma\rho_1'$ or $\rho_u=\rho_u'\sigma$. 
In the first case for any 
$\sigma'$ such that $\sigma'\rho'_1\ne 0$ we have $\genmu(\sigma'\rho'_1,\dots, \rho_u)=\sigma'\genmu(\rho'_1,\dots, \rho_u)$
 in the second we have 
$\genmu(\rho_1,\dots, \rho'_u\sigma')=\mu(\rho_1,\dots, \rho'_u)\sigma'$ if $\rho_u'\sigma'\ne 0$.
 \item $\genmu(\rho_1,\dots ,\rho_u)=0$ if $\rho_{i}\rho_{i+1}\ne 0$ for some $i<u$ or in particular when $\rho_i$ is trivial.
\end{enumerate}
\end{lemma}
\begin{proof}
We prove this by induction on the number of cycles we need to reduce $\genmu_u$ to an ordinary product.
\begin{enumerate}
\item[1]
The first statement clearly holds for $\genmu_2$ and by induction for higher multiplications because any positive cycle 
$\beta_1\cdots \beta_l$ is contractible in 
$|\qpol|\setminus \qpol_0$.
\item[2] 
If there is just one cycle then $\genmu(\rho_1\beta_1,\dots,\beta_l\rho_2)=\rho_1\rho_2$ but $\beta_1$ and $\beta_l$ 
sit in 2 different negatives cycle so $\rho_1\rho_2$ is zero unless one of the 2 paths has length zero.
It is also clear that if $\rho_1$ has nonzero length then $\genmu(\sigma\rho_1\beta_1,\dots,\beta_l\rho_2)=\sigma \rho_1$ 
and a similar statement holds for nontrivial $\rho_2$.

If the statement holds up to $k$ reductions then it also holds for $k+1$ reductions because the first and last path
in $\genmu$ after a reduction are subpaths starting from the outer ends of the first and last paths of the unreduced $\genmu$.
\item[3]
If there is just one cycle and $\genmu(\rho_1\beta_1,\dots,\beta_l\rho_2)\ne 0$ then $\beta_i\beta_{i+1}=0$.  
If the statement holds up to $k$ reductions then it also holds for $k+1$ reductions because $\beta_{i}\beta_{i+1}=0$ for 
any positive cycle.
\end{enumerate}
\end{proof}

\begin{proposition}\label{welldef}
The products $\genmu$ turn $\cQA(\rect\qpol)$ into an $\cA_{\infty}$-category. 
\end{proposition}
\begin{proof}
We will show that the identity $[M_k]$
\[
\sum_{s+r+t=k}(-1)^{s+rt+(2-r)(\rho_1+\dots +\rho_s)}\genmu(\rho_1,\dots \rho_{s},\genmu_r(\rho_{s+1},\dots,\rho_{s+r}),\rho_{s+r+1}\dots, \rho_k)=0
\]
holds for all possible collections of paths $\rho_1,\dots,\rho_k$ with $t(\rho_i)=h(\rho_{i+1})$.
We will do this using induction on $k$.

For $k\le 3$ the identities hold because $\cQA(\qpol)$ is an associative algebra with zero differential.
Suppose now that the identity $[M_j]$ holds for all $j<k$.

Because $k>3$, every term in $M_k$ will have at least one higher order multiplication in it. So if 
there is a nonzero term, there is at least one cycle we can reduce, so we can assume that the sequence of paths looks like
\[
 \rho_1,\dots,\rho_i\beta_1,\beta_2,\dots, \beta_{l}\rho_{i+1},\dots, \rho_k.
\]
Suppose now that we have a term that is nonzero, then the inner $\genmu$ must act on at least one entry containing a $\rho$.

If the inner $\genmu$ overlaps partially with the cycle then we can consider two situations:
\begin{itemize}
\item[A]
If the outer $\genmu$ is a higher multiplication we get expressions like
\[
 \genmu(\dots, \genmu(\dots,\beta_j),\beta_{j+1}, \dots ) \text{ or } \genmu(\dots, \beta_j,\genmu(\beta_{j+1},\dots), \dots ).
\]
In the first case $\genmu(\dots,\beta_j)$ must evaluate to something ending in $\beta_j$
otherwise we cannot reduce $\beta_{j+1}$. 

If $j\ge 2$, $\genmu(\dots,\beta_j)$ we first reduce $\genmu(\dots,\beta_j)$ until we
get to the reduction that will remove $\beta_{j-1}$. If this reduction does not remove $\beta_j$ then 
we end up with a higher order $\genmu$ that contains a trivial path which contradicts lemma \ref{rules} because we assumed the term was nonzero.
If this reduction does remove $\beta_j$ the only way we could get something nonzero is if the situation looked like
\[
 \genmu(\sigma\beta_{j+1},\beta_{j+2},\dots,\beta_j) \to \genmu(\sigma,t(\beta_j)) =\sigma.
\]
But as $\sigma$ must end in $\beta_j$ we have that $\sigma\beta_{j+1}=0$. 
Hence if the term is nonzero we must assume $j=1$ and 
then for the same reason no reduction for the inner $\genmu$ can reduce $\beta_1$. This implies we are in case C4 of the summary below.

Similarly for $\genmu(\dots, \beta_j,\genmu(\beta_{j+1},\dots), \dots )$ we have that $j>l-2$ 
and no inner reduction can reduce $\beta_{l}$. This implies we are in case C3 of the summary below.
\item[B]
If the outer $\genmu$ is the ordinary multiplication then it looks like
\[
 \genmu(\genmu(\dots,\beta_{l-2} ,\beta_{l-1}),\beta_{l}\rho_k)
\]
or a similar expression with the inner $\genmu$ on the right.
Now reduce the inner $\genmu$ until we get to the reduction that gets rid of $\beta_{l-2}$. It must also
remove $\beta_{l-1}$ because otherwise we get a higher multiplication with a trivial path. This implies that the nonzero term looks like
\[
 \genmu(\genmu(\sigma_1\beta_{l}\sigma_2,\dots,\sigma_u\beta_1,\dots,\beta_{l-2} ,\beta_{l-1}),\beta_{l}\rho_k)
\]
(The inner reduction first gets rid of $\sigma_2,\dots,\sigma_u$ and after that there can only be $l$ terms 
in the inner $\genmu$
because otherwise reducing $\beta_{l}\dots \beta_{l-1}$ will give a higher multiplication with a trivial path).

If $\sigma_2,\dots,\sigma_u$ is nontrivial, we pick instead of $\beta_{1}\dots \beta_{l}$ a cycle in there. 
For this cycle case B can never occur.

If $\sigma_2,\dots,\sigma_u$ is trivial our sequence of paths looks like
$$\rho_1\beta_{l},\beta_1,\beta_2,\dots,\beta_{l-1},\beta_l\rho_2 ~~(*).$$ 
It is easy to check that for such cases $[M_k]$ holds.
\end{itemize}

So, if our sequence is not of the form $(*)$ we can find a reducible cycle $\beta$ such that 
all nonzero terms fall in $4$ categories.
\begin{itemize}
 \item[C1] The inner $\genmu$ contains no part of the cycle $\beta_1,\dots,\beta_{l}$.
 \item[C2] The inner $\genmu$ contains the whole cycle $\beta_1,\dots,\beta_{l}$.
 \item[C3] The first entry of the inner $\genmu$ is $\beta_{l}\rho_{i+1}$ and the inner $\genmu$ evaluates to a right multiple of $\beta_{l}$.
 \item[C4] The last entry of the inner $\genmu$ is $\rho_{i}\beta_1$ and the inner $\genmu$ evaluates to a left multiple of $\beta_{1}$.
\end{itemize}
In each of these cases we can get rid of the $\beta's$ and the term equals
\[
\pm\genmu(\rho_1,\dots, \rho_s,\genmu(\rho_{s+1},\dots, \rho_{s+r}),\rho_{s+r+1},\dots,\rho_k)
\]
for the appropriate $s,r,t$. Different terms will give different simplified terms.
We get
that the expression $[M_{k+nl-2}]$ for $$ \rho_1,\dots,\rho_i\beta_1,\beta_2,\dots, \beta_{l}\rho_{i+1},\dots, \rho_k$$ is
$[M_k]$ for $\rho_1,\dots,\rho_k$ up to a sign.

The fact that the signs match up requires some computation. In particular we need to show that
the product of the sign in $[M_{k+nl-2}]$, the sign of before the outer $\genmu$ and the sign in $[M_k]$ do not
depend on $s,r,t$ and the $4$ different cases.

We will only do this calculation for the first case with $i<s$ (i.e. the cycle comes before $\genmu$)
\se{
&\underbrace{s+(l-2)+rt+(2-r)(\rho_1+\dots +\rho_s+\beta_1+\dots+\beta_{l})}_{[M_{k+l-2}]} \\
&+\underbrace{s+rt+(2-r)(\rho_1+\dots +\rho_s)}_{M_k} + \underbrace{l(\rho_1+\dots + \rho_i)+s+t+1-i}_{\kappa}
\\&=(l-\cancel 2) + (\cancel 2-r)(\underbrace{\beta_1+\dots+\beta_{l}}_{=l}) + l(s+t+1-i+\rho_1+\dots+\rho_i)
\\&= l(\underbrace{r+s+t}_{=k}+1 - i +\rho_1+\dots+\rho_i)\mod 2
}
and the second case (when the cycle is inside the inner $\genmu$: $s+t>i>s$)
\se{
&\underbrace{s+(r+l-\cancel 2)t+(\cancel 2-r-l+\cancel 2)(\rho_1+\dots +\rho_s)}_{[M_{k+l-2}]} \\
&+ \underbrace{s+rt+(\cancel 2-r)(\rho_1+\dots +\rho_s)}_{M_k} + \underbrace{l(\rho_{s+1}+\dots + \rho_i)+r-i+s}_{\kappa}
\\&=lt - l(\rho_1+\dots +\rho_s) + l(r-i+s+\rho_{s+1}+\dots+\rho_i)
\\&= l(\underbrace{r+s+t}_{=k}+1 - i +\rho_1+\dots+\rho_i)\mod 2
}
\end{proof}

\subsection{Identifying the $\cA_{\infty}$-structure on $\cQA(\rect\qpol)$}

In this section we will investigate how one can recognize $\genmu$ up to isomorpism or $\cA_\infty$-isomorphism.

\begin{lemma}\label{cyclic}
If $\mu$ is an $\cA_\infty$-structure on $\cQA(\rect\qpol)$ 
such that for every positive cycle $\beta_1\dots \beta_l$ we have
$$\mu(\beta_i,\dots,\beta_{j})=\begin{cases}
\kappa_{\beta,i} h(\beta_i)&j-i+1=l\\
0&\text{ otherwise,}                        
                        \end{cases}
$$
for some nonzero constants $\kappa_{\beta,i}\in \C$
then $\mu$ is isomorphic to an $\cA_\infty$-structure $\mu'$ satisfying the same equalities but with all $\kappa_{\beta,i}=1$.
\end{lemma}
\begin{proof}
From the identity $[M_{l+1}]$ applied to the paths $\beta_1,\dots,\beta_{l},\beta_1$:
\[
 \mu(\beta_1,\mu(\beta_2,\dots, \beta_{l},\beta_1)) - \mu(\mu(\beta_1,\dots, \beta_{l}),\beta_1)=0.
\]
we can conclude that $\kappa_{\beta,i}=\kappa_{\beta,j}$ for all $i,j \in \Z$.

Now for each positive cycle we rescale the arrow $\beta_1$ to $\beta_1'=\beta_1/\kappa_{\beta,1}$.
Because every arrow $\beta$ sits in just one positive cycle it is clear that
$\mu(\beta_1',\beta_2,\dots,\beta_{j}) = h(\beta_1')$. The rescaling algebra-isomorphism 
$\phi \in \Aut(\cQA(\rect \qpol))$ allows us to construct $\mu' := \phi^{-1}\mu\circ \phi^{\otimes k}$, which satisfies the required property. 
\end{proof}

\begin{theorem}\label{uptoiso}
Suppose that $\mu$ is an $\cA_\infty$-stucture on $\cQA(\rect \qpol)$ such that for each
sequence of paths $\rho_1,\dots,\rho_k$ we can find a $\lambda\in \C^*$ with 
\[
\mu(\rho_1,\dots,\rho_k)= \lambda\genmu(\rho_1,\dots,\rho_k),
\]
then $\mu$ is isomorphic to $\genmu$.
\end{theorem}
\begin{proof}
By the previous lemma we can assume that $$\mu(\beta_i,\dots,\beta_{i+l-1})=h(\beta_i)=\genmu(\beta_i,\dots,\beta_{i+l-1})$$ 
for each positive cycle.
In this case we will prove that $\mu_k=\genmu_k$ by induction on $k$.
If $k$ is $2$ the statement is true because $\genmu$ and $\mu$ are both $\cA_\infty$-structures on $\cQA(\rect \qpol)$.

If $k>2$ then we will distinguish 2 cases
\begin{itemize}
 \item If $\genmu(\beta_1,\dots,\beta_l\rho)\ne 0$ then the axiom $[M_{l+1}]$  implies
$$
\genmu(\beta_1,\dots,\beta_l\rho) \pm \genmu(\beta_1,\dots,\beta_l)\rho =0.
$$
For $\mu$ the axiom $[M_{l+1}]$ has the same nonzero terms as for $\genmu$ because terms are scalar multiples of each other.
The right $\genmu$-term is equal to its $\mu$-equivalent by assumption, hence we 
get equality of the left terms: so $\genmu(\beta_1,\dots,\beta_l\rho)=\mu(\beta_1,\dots,\beta_l\rho).$ 
Similarly we get $\genmu(\rho,\beta_1,\dots,\beta_l)=\mu(\rho,\beta_1,\dots,\beta_l)$. 
\item
If $\genmu(\rho_1,\dots,\rho_i\beta_1,\dots,\beta_l\rho_{i+1},\dots,\rho_k)\ne 0$, we can assume
that either $\rho_i$ or $\rho_{i+1}$ is not trivial, otherwise the reduced expression will contain two trivial paths
and will be zero unless it is an ordinary multiplication. In that case the unreduced expression is $\genmu(\beta_1,\dots,\beta_{l})$,
so our initial assumption implies the statement.

If $\rho_i$ is nontrivial we will apply $[M_{k+l-1}]$ to $\rho_1,\dots,\rho_i,\beta_1,\dots,\beta_l\rho_{i+1},\dots,\rho_k$.
The only terms in the $\genmu$-version of $[M_{k+l-1}]$  that do not fall under the induction hypothesis are 
$$\genmu_2(\genmu(\dots),\rho_k),~~ \genmu_2(\rho_1,\genmu(\dots)),~~ \genmu(\dots,\genmu_2(\dots), \dots).$$

By lemma \ref{rules}[3] the first term is zero because $\rho_i\beta_1\ne 0$.
The second term can only be nonzero if $i=1$ and $\genmu(\dots)=\genmu(\beta_1,\dots,\beta_l\rho_k)$ 
which is equal to its $\mu$-version by the previous case. For the last terms, lemma \ref{rules} implies that the only combination of 2 consecutive terms for which the ordinary product is nonzero is $\rho_i,\beta_1$.

All nonzero terms in $[M_{k+l-1}]$ except $$\genmu(\rho_1,\dots,\rho_i\beta_1,\dots,\beta_l\rho_{i+1},\dots,\rho_k)$$
are equal to their $\mu$-versions by the induction hypothesis 
and therefore this last term is also equal to its $\mu$-version.
\end{itemize}
\end{proof}

\begin{theorem}\label{uptoaiso}
Let $\qpol$ be a dimer and $\cQA(\rect\qpol)$ be the gentle category coming from its rectified quiver $\rect\qpol$. 
Suppose $\rect\qpol$ is well-behaved and $\deg_X$ is a $\Z$-grading on $\cQA(\rect\qpol)$ as in \ref{Xgrading}. If $\mu$ is a $\deg_X$-graded $\cA_\infty$-structure 
on $\cQA_X(\rect\qpol)$ such that for every positive cycle $\beta_1\dots \beta_l$ we have
$$\mu(\beta_i,\dots,\beta_{j})=\begin{cases}
h(\beta_i)&j-i+1=l\\
0&\text{ otherwise,}                        
                        \end{cases}
$$
then $\mu$ is $\cA_\infty$-isomorphic to $\genmu$.
\end{theorem}
\begin{proof}
Clearly the $\genmu$ is also homogeneous for $\deg$ because by the condition
above every positive cycle $\beta_1\dots \beta_l$ has degree $2-l$.

In order to prove that $\genmu$ and $\mu$ are $\cA_\infty$-isomorphic we construct a $\cA_{\infty}$-functor 
$\cF$ with $\cF_1=\id$.
We show that we can do this by constructing the $\cF_i$ one at a time.
Suppose we have constructed $\cF_i$ for $i<r$.

Now look at the identity $[F_n]$. We already know it is of the form
\[
 d\cF_n = \Psi
\]
with $d\Psi=0$ and $\Psi$. In order to show that we can find a $\cF_n$ we have 
to show that $\Psi\in \Image  d$ or equivalently that the homology
class of $\Psi$ is zero. If we prove that $\Psi(\beta_{i},\dots, \beta_{i+vl-1})$ and $\Psi(\beta_{i},\dots, \beta_{i+vl})$ 
contain no length $0$ or length $1$ terms, then by lemma \ref{whenhomzero} we are done. 

The $\deg_X$-degree of $\Psi$ is $1-n$ and for $n=vl-1$ we get that  
$$\deg_X \Psi(\beta_{i},\dots, \beta_{i+vl-1})= 1-(vl-1) +v(l-2)=2(v-1)$$
This can only contain a length $\le 1$ term if $v=1$. Note that we used the fact that the total degree of the cycle $\beta_1\dots\beta_{l}$ is $l-2$. 

Similarly if $n=vl$ then
$$\deg_X \Psi(\beta_{i},\dots, \beta_{i+vl})= 1-(vl) +v(l-2)+\deg_X \beta_i$$ which is even but all length $1$ terms have odd degree. Moreover because 
$h(\beta_i)\ne t(\beta_i)$ a length zero term is also impossible.

So we can assume $v=1$ and $n=l-1$.
The terms in the expression $\Psi(\beta_i,\dots,\beta_{i+l-1})$ are of the form
\begin{itemize}
\item $\mu(\beta_{i},\dots,\beta_{i+l-1})-\genmu(\beta_{i},\dots,\beta_{i+l-1})$.
This expression is zero by the assumptions on $\mu$.
\item
$\cF(\beta_i,\dots ,\beta_{j},\mu(\beta_{j+1},\dots,\beta_{u}),\beta_{u+1},\dots,\beta_{r})$. Such a term is zero because 
by the condition in the theorem $\mu(\beta_{j+1},\dots,\beta_{u})$ can only be nonzero if $u-j+1=l$.
\item
$\genmu(\cF(\beta_1,\dots, \beta_{i_1}),\dots, \cF(\beta_{i_j},\dots, \beta_{r}))$.  
If $\cF(\beta_{i},\dots, \beta_{i+s})$ is nonzero then by the fact that $\rect\qpol$ is well-behaved
$\cF(\beta_{i},\dots, \beta_{i+s})$ can only contain a length one term if $s=0 \mod l$.  
In other words no terms have length 1. By construction $\genmu$ of such an expression is zero,  because there is no position where we can start to reduce it.
\end{itemize}
\end{proof}

\section{Appendix on the construction of the wrapped Fukaya category\\~\\ Mohammed Abouzaid}\label{appendixabouzaid}
\vspace{.5cm}
\newcommand{\bZ}{{\mathbb Z}}
\newcommand{\Spin}{\operatorname{Spin}}
\def\co{\colon\thinspace}
\renewcommand{\th}{${}^{\mathrm th}$}

\subsection{The wrapped Fukaya category}

We recall the construction of the wrapped Fukaya category: given a Liouville domain $M$ such that all Reeb orbits on $\partial M$ are non-degenerate, and a collection $(L_1, \cdots, L_{N})$ of exact Lagrangians with Legendrian boundary so that there is no Reeb chord of integral length starting on $\partial L_i$ and ending on $\partial L_j$ for any pair $(i,j)$, we choose a Hamiltonian $H$ on $M$ which agrees with the linear coordinate near $\partial M$, so that the intersection between the starting points of integral $X_{H}$-chords and their endpoints is empty.

In \cite{Abou2}, the above conditions are proved to be generic if the real dimension of $M$ is $4$ or greater. We assume that all Lagrangians are equipped with fixed $\Spin$ structures, which ensures that all Floer complexes and operations are defined over $\bZ$. The case when $M$ is a Riemann surface is recovered by a stabilisation procedure where we take the product with $T* S^1$, as described in \cite{Abou2}.  We require moreover that $2c_{1}(M) = 0$, and that we have a fixed quadratic complex volume form on $M$, as well as \emph{gradings} on the Lagrangians $L_i$. This implies that all Floer complexes are $\bZ$ graded, and that the $d$\th operation in the $A_{\infty}$ structure is homogeneous of degree $2 - d$. Stabilisation does not change the degree of generators of Floer complexes, so the argument we give apply in the surface case as well.

 In this case, we define
\begin{equation}
CW^*(L_i, L_j) \equiv \bigoplus_{w = 1}^{\infty} CF^*(L_i, L_j; w H)[\theta],
\end{equation}
where $\theta$ is a formal variable of degree $-1$ such that $\theta^2 = 0$, and $ CF^*(L_i, L_j; w H)  $ is the Lagrangian Floer complex generated by time-$w$ Hamiltonian chords of $H$ with endpoints on $L_i$ and $L_j$. The differential is the sum of three terms: the internal Floer differential of each summand, the identity as a map
\begin{equation}
CF^*(L_i, L_j; w H) \cdot \theta\to CF^*(L_i, L_j; w H),
\end{equation}
and the continuation map
\begin{equation}
\kappa \co  CF^*(L_i, L_j; w H) \cdot \theta\to CF^*(L_i, L_j; (w+1) H).
\end{equation}
To describe the $A_{\infty}$ structure, consider the natural direct sum decomposition of tensor powers of the wrapped Floer complex by the $\theta$-order. The terms which preserve the $\theta$-order are given by the usual $A_{\infty}$ operations counting pseudo-holomorphic maps from a $d$-sided polygon to $M$:
\begin{multline}\label{eq:a_infty_operation}
\mu^{d}_{CF} \co    CF^*(L_{i_{d-1}}, L_{i_d}; w_{d} H) \otimes  \cdots \otimes  CF^*(L_{i_{0}}, L_{i_{1}}; w_{1} H)  \\ \to CF^*(L_{i_{0}}, L_{i_{d}}; (w_{1} + \cdots + w_{d}) H),
\end{multline}
with the convention that $\theta$ Koszul commutes with $\mu^{d}_{CF}$. The remaining terms of the $A_{\infty}$ operations are obtained by counting solutions to parametrised equations called \emph{popsicle maps;} the only thing that we need to recall about these is:
\begin{lemma} \label{lem:popsicles_vanish}
  All moduli spaces of popsicle maps with inputs $(x_1, \ldots, x_d)$ and output $x_0$ have negative virtual dimension if
\begin{equation} \label{eq:degree_condition_homogeneous}
  \deg(x_0) \geq 2 - d + \sum \deg(x_i).
\end{equation}
In particular, if this condition holds, these moduli spaces are generically empty and do not contribute to the $A_{\infty}$ structure. \qed
\end{lemma}
\begin{proof}
Since the $\theta$-order of the output is at most $1$,  these moduli spaces contribute terms to the $A_{\infty}$ operations which strictly decrease the $\theta$-order.  The result follows immediately from the fact that $\theta$ has degree $-1$,  and that the degree of the $d$\th higher product is $2-d$.
\end{proof}

\subsection{In the absence of breaking}
\label{sec:topol-cond}

Let us now assume that all the Lagrangians we are considering are contractible. This implies that there is a unique way to concatenate a pair of homotopy classes of paths from $L_i$ to $L_j$ and from $L_j$ to $L_k$, to obtain a homotopy class of paths from $L_i$ to $L_j$. This operation is associative, and we write
\begin{equation}
[x_1] \star \cdots \star [x_d], 
\end{equation}
for the homotopy class obtained by concatenating chords $(x_1, \ldots, x_d)$, where $x_k$ is a time-$w_{k}$ Hamiltonian chord starting on $L_{i_k - 1}$ and ending on $L_{i_k}$. We now assume:
\begin{equation}
  \label{eq:topological}
\parbox{35em}{for each integer $w$, and each $w_{0} X_{H}$ chords $x_0$ satisfying $[x_0] = [x_1] \star \cdots \star [x_d] $, Equation \eqref{eq:degree_condition_homogeneous} is satisfied.}
\end{equation}

The first consequence of this assumption is that the differential on $CF^*(L_i, L_j ; w H) $ vanishes for any triple $(i,jw)$. Next, we observe:
\begin{lemma}
  The $A_{\infty}$ operation in Equation \eqref{eq:a_infty_operation} is independent of the choice of Floer data.
\end{lemma}
\begin{proof}
  Any two choices can be connected by a $1$-parameter family if Floer data, and the homotopy between the two operations is obtained by counting the corresponding parametrised moduli spaces which have virtual dimension $0$. The assumption implies that all such parametrised moduli spaces are empty for homotopical reasons, and the equation for a homotopy therefore implies that the operations are independent of Floer data.
\end{proof}

The same argument shows that the continuation map $\kappa$ strictly commute with the $A_\infty$ operations on Floer cochains, which we identify with Floer cohomology. We therefore obtain an $A_\infty$ structure on
\begin{equation}
\lim_{\kappa}  HF^{*}(L_i, L_j; w H) \equiv HW^*(L_i, L_j),
\end{equation}
and, by the above Lemma the operations obtained is strictly independent of the choice of Floer data.

Returning to the definition of the wrapped Fukaya category, Condition \eqref{eq:topological} and Lemma \ref{lem:popsicles_vanish} imply that there are no moduli spaces of popsicles maps of vanishing virtual dimension. The natural map
\begin{equation}
p \co CW^{*}(L_i, L_j) \to HW^*(L_i, L_j),
\end{equation}
which sends $\theta$ to $0$ is a (strict) $A_{\infty}$ homomorphism.  We conclude
\begin{proposition}\label{propabouzaid}
  If Condition \eqref{eq:topological} holds, the $A_{\infty}$ structure on wrapped Floer cohomology descends to cohomology as follows: 
\begin{equation}
\mu^{d}_{HW}(p(x_d), \cdots, p(x_1)) = p( \mu^{d}_{CF}( x_d, \ldots, x_1) ).
  \end{equation} \qed
\end{proposition}

\bibliographystyle{amsplain}
\def\cprime{$'$}
\providecommand{\bysame}{\leavevmode\hbox to3em{\hrulefill}\thinspace}
\providecommand{\MR}{\relax\ifhmode\unskip\space\fi MR }
\providecommand{\MRhref}[2]{%
  \href{http://www.ams.org/mathscinet-getitem?mr=#1}{#2}
}
\providecommand{\href}[2]{#2}

\end{document}